\documentclass[twoside,a4paper,11pt]{article}
\usepackage[top=4cm, bottom=4cm, left=3cm, right=3cm]{geometry}

\usepackage{macros}

\renewcommand{\phi}{\varphi}

\usepackage{amsmath}               
  {
      \theoremstyle{plain}
      \newtheorem{assumption}{Assumption}
  }

\begin{document}
\title{Minimum intrinsic
dimension scaling for entropic optimal transport}
\author{Austin J. Stromme\thanks{Department of EECS at the Massachusetts Institute of Technology, \texttt{austinjstromme.work@gmail.com}.}
}

\maketitle

\begin{abstract}
Motivated by the manifold hypothesis,
which states that data with a high extrinsic dimension
may yet have a low intrinsic dimension,
we develop refined statistical
bounds for entropic optimal transport
that are sensitive
to the intrinsic dimension of the data.
Our bounds involve a robust notion
of intrinsic dimension, measured at only
a single distance scale
depending on the regularization parameter,
and show that
it is only the minimum
of these single-scale intrinsic dimensions which 
governs the rate of convergence. We call this the Minimum
Intrinsic Dimension scaling (MID scaling) phenomenon, and establish MID scaling
with no assumptions on the data distributions
so long as the cost is bounded and Lipschitz,
and
for various entropic optimal transport quantities beyond just
values, %
with stronger analogs when one distribution
is supported on a manifold.
Our results significantly
advance
the theoretical state of the art
by showing that MID scaling is a generic
phenomenon, and 
 provide the first
rigorous interpretation
of the statistical
effect of entropic regularization as a distance
scale.
\end{abstract}

\section{Introduction}

Optimal transport (OT) is a powerful paradigm for comparing
probability distributions based on a minimum-energy criterion,
and has
recently been employed throughout applied science,
including statistics, computer science,
biology
and beyond~\cite{villani2008optimal,santambrogio2015optimal,peyre2019computational, panaretos2019statistical,schiebinger2019optimal}.
Rather than comparing distributions pointwise, OT searches
for the most efficient way to transform one into the other,
and thus gives practitioners a geometrically meaningful method
of contrasting and interpolating data that can be represented
as probability distributions.
For probability measures $\mu, \nu$ on $\R^d$
the OT problem, with respect to the squared Euclidean
cost $\|\cdot\|^2$, is defined as
\begin{equation}\label{eqn:OT}
    W_2^2(\mu, \nu) := \min_{\pi \in \Pi(\mu, \nu)} \int \|x - y\|^2\ud \pi(x, y),
\end{equation}
where $\Pi(\mu, \nu)$ denotes the set of probability distributions
$\pi$ on $\R^d \times \R^d$ which couple $\mu$ to $\nu$,
namely such that
their marginal on the first $d$ coordinates is $\mu$ and
their marginal on the last $d$ coordinates is $\nu$.
The fundamental theorem of OT
guarantees that, so long as $\mu, \nu$ have finite second moments,
there is a solution to~\eqref{eqn:OT};
moreover if $\mu$ has a density with respect to the Lebesgue measure,
there is a unique minimizer which is supported on the graph of a deterministic
map~\cite{villani2008optimal}. For many
applications, such as domain adaptation~\cite{flamary2016optimal} or cellular biology~\cite{schiebinger2019optimal},
it is the coupling minimizing~\eqref{eqn:OT},
rather than the cost $W_2^2(\mu, \nu)$ itself,
which is of greatest interest.

\paragraph*{Entropic optimal transport.} When discretized, the OT problem~\eqref{eqn:OT}
becomes a linear program (LP) with two equality constraints.
Although this problem can be feasibly solved at moderate
scale with specialized LP solvers
which have computational complexity scaling cubically in the
support size~\cite{dong2020study},
in practice, OT is most often approximated with an entropic regularization
term~\cite{cuturi2013sinkhorn}. This entropically regularized problem
is then solved with a simple iterative rounding algorithm
known as Sinkhorn's algorithm~\cite{sinkhorn1964relationship,sinkhorn1967concerning},
and is preferred for its quadratic scaling in the support size, simplicity,
and 
parallelizability.
The entropically regularized OT (entropic OT)
problem is defined, for
a regularization parameter $\eps > 0$, as
\begin{equation}\label{eqn:entropic_OT}
    S_\eps (\mu, \nu) := \min_{\pi \in \Pi(\mu, \nu)}
    \int \|x -y\|^2 \ud \pi(x, y) + \eps \KL{\pi}{\mu \otimes \nu},
\end{equation} where $\mu \otimes \nu$ denotes the joint
law of $(x, y)$ when $x \sim \mu$ and $y \sim \nu$ are independent,
and the Kullback--Leibler divergence is defined as
$$
\KL{P}{Q} :=\begin{cases}
\int \ln \big( \frac{\ud P}{\ud Q}(x) \big) \ud P(x) & P\ll  Q, \\
\infty & P \not \ll Q,
\end{cases}
$$ and the notation $P \ll Q$ means $P$ is absolutely continuous with
respect to $Q$.

\subsection{Statistical aspects of
(entropic) OT}
Unfortunately, the un-regularized OT problem~\eqref{eqn:OT}
is known to suffer from a statistical curse of dimensionality.
To describe this barrier, consider the practical
situation in which one does not have access to the entire
distributions $\mu, \nu$, and instead only 
has access to iid samples of size $n$
from each distribution, which we write as
$\X := (x_1, \ldots, x_n) \sim \mu^{\otimes n}$ and $\Y:= (y_1, \ldots, y_n)
\sim \nu^{\otimes n}$.
Let $\hat \mu$ 
and $\hat \nu$ denote the empirical measures
supported on $\X$ and $\Y$, respectively, namely
$$
\hat \mu := \frac{1}{n} \sum_{i = 1}^n \delta_{x_i}
\quad \quad \textrm{ and } \quad \quad \hat \nu := \frac{1}{n} \sum_{j =1 }^n
\delta_{y_j}.
$$
Given the empirical measures $\hat \mu, \hat \nu$, the most
natural way to estimate the population quantity $W_2^2(\mu, \nu)$
is with the plug-in estimator $W_2^2(\hat \mu, \hat \nu)$.
The discrepancy between the empirical OT value
and its population counterpart is an old and
well-studied area, and
it has long been understood that rates
like $n^{-2/d}$ are typical -- for instance,
if $d \geqslant 5$, and
$\mu, \nu$ are absolutely continuous
with respect to the Lebesgue measure on $[0, 1]^d$~\cite{dudley1969speed,manole2021sharp}.
It is natural to wonder whether these rates
can be improved with other estimators,
but it was recently shown they are essentially un-improvable,
both for OT value estimation~\cite{niles2022estimation,manole2021plugin} and
OT map estimation~\cite{hutter2021minimax}.

\paragraph*{Statistical entropic OT.}
Motivated by its computational benefits
and corresponding prevalence in practice,
as well as the curse of dimensionality for its
un-regularized counterpart,
a line of recent work has
endeavored to understand the statistical
consequences of entropic regularization~\cite{luise2019sinkhorn,genevay2019sample,mena2019statistical,rigollet2022sample,del2023improved}.
This work has shown that the entropic OT problem~\eqref{eqn:entropic_OT}
offers significantly improved statistical performance,
essentially transferring the curse of dimensionality
from the sample size $n$ to the regularization parameter $\eps$;
the following result describes the current
state of the art when the measures are bounded
and $\eps$ is small.

\begin{theorem}[\cite{genevay2019sample,mena2019statistical}]\label{thm:genevay_rate}
Suppose $\mu, \nu$ are probability measures on $\R^d$
with support contained in the unit ball $B(0, 1)$. Then for a constant $C_d$ depending only
on the dimension $d$,
$$
\E [|S_\eps(\hat \mu, \hat \nu) - S_\eps(\mu, \nu)|] \leqslant C_d \cdot \frac{1 + \eps^{-\lfloor d/2\rfloor }}{\sqrt{n}},
$$ where the expectation is taken over the iid samples
$\X \sim \mu^{\otimes n}, \Y \sim \nu^{\otimes n}$.
\end{theorem}
By giving
an explicit trade-off between regularization
and statistical error,
this result
offers a strong
and flexible description of the performance of entropic
OT in practice.
In this work, we build on Theorem~\ref{thm:genevay_rate}
and develop
a refined theory of the statistical
behavior of entropic OT.

\subsection{Effective statistical
dimension of entropic OT}

To describe our refinements, observe that
the estimate in Theorem~\ref{thm:genevay_rate}
is {\it extrinsic} in the sense that
the dimension $d$ is appearing because the ambient
space is $\R^d$, rather than anything to do with the {\it intrinsic}
dimensions of $\mu$ and $\nu$.
This extrinsic dimension dependence is fundamental
to the proof,
and can be very pessimistic.
For a toy example, suppose $\mu$ and $\nu$ are supported on the same
$k$-dimensional hyperplane in $\R^d$: then Theorem~\ref{thm:genevay_rate}
itself
implies that the statistical rate has the potentially
much milder dependence $\eps^{-\lfloor k/2\rfloor}$ on the regularization parameter
$\eps$.
In such a case, we can say that
$\mu, \nu$ have extrinsic
dimension $d$ yet have intrinsic dimension
(at most) $k$. Well beyond toy settings,
the widely believed
manifold hypothesis
states that natural data with a high
extrinsic dimension
is typically near to or on a low-dimensional
manifold embedded in Euclidean 
space, and so has low intrinsic
dimension~\cite{fefferman2016testing,bengio2017deep}. %
Given the ubiquity of entropic OT in practice,
it is therefore
of major interest
to determine the effective
statistical dimension of entropic OT,
and identify how it relates
to the extrinsic dimension $d$
as well as the intrinsic dimensions
of $\mu$ and $\nu$.
In this work, we thus study the
following question:
\begin{center}
    {\it Suppose $\mu$ and $\nu$
    are supported on $\R^d$,
    yet have (informally speaking) intrinsic dimensions
$d_\mu$ and
    $d_\nu$. How does the statistical
    rate of convergence of entropic OT
    depend on $d, d_\mu$, and $d_\nu$?}
\end{center}

Our results will measure intrinsic dimension
through the covering numbers of the supports
of $\mu$ and $\nu$.
To introduce covering numbers,
write the closed ball centered at $z_0 \in \R^d$ 
with radius $r > 0$ as
$$
B^{\rm cl}(z_0, r) := \big\{ z \in \R^d \colon
\|z - z_0 \| \leqslant r \big\}.
$$
Then the covering number of a set
$A \subset \R^d$ at scale $\delta > 0$
is defined as
\begin{equation} \label{eqn:covering_defn}
\mathcal{N}(A, \|\cdot\|, \delta):=
\min \Big\{ K \in \mathbb{N}_{> 0}
\, \, \Big| \, \, \exists
z_1, \ldots z_K \in \R^d \, \, \colon \,\,
A \subset
\bigcup_{k = 1}^K B^{\rm cl}(z_k,
\delta) \Big\}.
\end{equation} It will be convenient to define
$$
\mathcal{N}(\mu, \delta):= \mathcal{N}(\supp(\mu), \|\cdot\| , \delta),
\quad \quad 
\mathcal{N}(\nu, \delta):= \mathcal{N}(\supp(\nu), \|\cdot\|, \delta), \quad \quad \delta > 0.
$$

\paragraph*{Contributions.}

Using this notation, we can give our answer to the above
question:
\begin{center}
    {\it {\bf MID scaling:} 
    the statistical rate of convergence
    of entropic OT can be upper bounded
    by quantities whose dimension-dependence
    is purely contained within the minimum
    covering number at scale $\eps$,
    namely
    $\mathcal{N}(\mu, \eps) \land 
    \mathcal{N}(\nu, \eps)$.}
\end{center}
We call this the minimum intrinsic dimension scaling (MID scaling)
phenomenon.
We emphasize that
the MID scaling phenomenon encapsulates two related
yet distinct phenomena:
\begin{enumerate}
\item[1.]\label{item:min_dep} {\bf Minimum:}
only the {\it minimum} of the intrinsic dimensions
governs the convergence
    \item[2.] {\bf Intrinsic Dimension:} the dimension-dependence is
intrinsic to $\mu$ and $\nu$,
and in fact is 
intrinsic {\it at a single distance scale depending on the regularization parameter}.

\end{enumerate} 

As we elaborate on in section~\ref{sec:main_results},
the MID scaling phenomenon provides a strong
statistical justification for entropic OT
in the context of the manifold hypothesis.
It shows that distributions with 
support which may not be smooth
or even locally low-dimensional can still have
significantly improved dimension-dependence
in their statistical rates, and clearly describes
the statistical role of the regularization parameter
$\eps$ as a distance scale. Moreover,
it identifies the surprising fact that the
dimension-dependence of entropic OT is driven
by the lower-dimensional (in this single-scale
sense) distribution.

The above statement of MID scaling is in the setting
of Theorem~\ref{thm:genevay_rate},
but our main results
establish
MID scaling more generally for bounded,
Lipschitz costs, and for entropic OT quantities
beyond values, including maps and densities.
As our main results in this setting, we prove MID scaling for entropic OT values in Theorem~\ref{thm:main_cost_covering},
for entropic OT maps in empirical norms in Theorem~\ref{thm:main_map_covering}, and
for entropic OT densities in empirical norms
in Theorem~\ref{thm:density_estimation_slow}.
Under the additional assumption that
one distribution is supported on an embedded
manifold, we show that MID scaling
also holds for entropic dual potentials
with fast, population norm convergence
in Theorem~\ref{thm:dual_convergence},
and apply this result to prove stronger
analogs of the previous results on value, map,
and density estimation.

\paragraph*{Organization of the paper.}
We conclude this section with a discussion
of related work and a summary of our notation. In section~\ref{sec:duality},
we state our assumption
on the cost function and
give background and more notation
for the entropic OT problem.
In section~\ref{sec:main_results},
we state our main results and give examples
and discussion.
In section~\ref{sec:preliminary},
we describe some preliminary observations
that form the foundation of our approach.
In section~\ref{sec:proof_overview},
we prove our main results on MID scaling
with a thorough exposition of our proof
strategy in the case of value estimation.
In section~\ref{sec:proofs_fast_rates},
we prove our main results
in the case where one distribution
is supported on
an embedded manifold.
In Appendix~\ref{sec:background_manifolds},
we collect some background
on embedded manifolds that we use
in section~\ref{sec:proofs_fast_rates},
and in Appendix~\ref{sec:deferred}
we include some deferred proofs.

\paragraph*{Related work.}

Taken as a whole, our results provide strong
evidence of the MID scaling
phenomenon for entropic OT.
In fact, there is
additional evidence for MID scaling
and related phenomena
in the literature.
Recent work studied
the continuous to discrete case,
and established that
entropic OT maps achieve
dimension-free rates of
convergence, consistent with MID scaling~\cite{pooladian2023minimax}.
And in the un-regularized setting,
a similar phenomenon was 
recently established
for value estimation,
where it was dubbed
``lower complexity adaptation"
~\cite{hundrieser2022empirical}.
The lower complexity adaptation phenomenon
also states that the minimum of the intrinsic
dimensions of $\mu$ and $\nu$ dictates the
rate of convergence of un-regularized OT,
but their notion of intrinsic dimension is different
from ours. Indeed, as is common in the un-regularized
literature, the intrinsic dimensions $d_\mu$
and $d_\nu$ in lower complexity adaptation
are defined using the covering
numbers at many distance scales,
and so are distinct from the single-scale
intrinsic dimensions in MID scaling.
Also, their results require some structural assumptions,
whereas our results show that MID scaling
in the entropic setting is quite generic.
In light of these results,
it is natural to expect that the minimum part of MID scaling
holds in greater generality,
potentially including un-regularized maps and plans,
and
alternate forms of regularization.

To the best of our knowledge,
there is only one prior work which
considers the sample complexity
of regularized OT and intrinsic dimension~\cite{bayraktar2022stability}.
Those results apply to more general forms of
regularization, but
are only for value estimation
and incur worse $n$ dependence
without identifying MID scaling.
Concurrently with this work,~\cite{groppe2023lower}
extended the lower complexity
adaptation phenomenon to entropic OT values and Gromov-Wasserstein
distances,
yielding bounds which are similar
to those of the un-regularized problem,
with the curse of dimensionality
residing primarily in the sample size,
and which are mostly incomparable
to our own.

As we discuss in detail in section~\ref{sec:proof_overview},
our technical approach is an intrinsic dimension-sensitive refinement of~\cite{rigollet2022sample}
where fully dimension-free bounds were established for entropic OT
values, maps, and densities, but with exponential factors in $1/\eps$.
An alternative estimator of the
entropic OT map which achieves sub-exponential
dependence on $1/\eps$ was proposed in~\cite{stromme2023sampling}.
Other works have proven convergence
of the dual potentials~\cite{luise2019sinkhorn},
used entropic OT quantities
as computationally efficient estimators for their un-regularized counterparts~\cite{chizat2020faster,pooladian2021entropic},
considered the convergence
of entropic OT maps~\cite{masud2021multivariate,werenski2023rank},
and studied the sample
complexity of entropically regularized
Gromov-Wasserstein distances~\cite{zhang2022gromov}.
Significant recent effort has been devoted
to developing central limit theorems for entropic OT~\cite{bigot2019central,mena2019statistical,klatt2020empirical,del2023improved,goldfeld2022statistical,gonzalez2022weak,gonzalez2023weak}.
Finally, a related form of minimum dimension-dependence appears in the study of asymptotics
for entropic OT as $\eps \to 0^+$~\cite{eckstein2023convergence}.

Regarding the provenance of Theorem~\ref{thm:genevay_rate},
we remark that it was originally proven for smooth, Lipschitz costs
beyond $\ell_2^2$, but with an exponential
dependence on $1/\eps$~\cite{genevay2019sample}.
The follow-up work~\cite{mena2019statistical}
showed how to remove
this exponential factor in the compactly
supported case that we study here, but primarily 
concentrated on extending Theorem~\ref{thm:genevay_rate}
to un-bounded distributions.

Finally, our work can be seen
as an entropic analog to the long
line of work on the convergence of the
un-regularized OT problem, which dates
back to Dudley~\cite{dudley1969speed}
and encompasses precise asymptotics~\cite{dobric1995asymptotics}
as well as finer behavior in lower dimensions~\cite{ajtai1984optimal,leighton1986tight}.
In fact, the analysis of un-regularized OT
is naturally sensitive to intrinsic dimension,
at least
when intrinsic
dimension
is measured through covering numbers at many scales~\cite{dudley1969speed,boissard2014mean,weed2019sharp}.
The discrepancy between Theorem~\ref{thm:genevay_rate}
and the natural appearance
of intrinsic dimension in the un-regularized OT problem
provides part of the motivation for this work.
More discussion of the similarities and differences
between this line of work and our results
is provided in section~\ref{sec:main_results}.

\paragraph*{Notation.}
Given $a, b \in \R$, we write
the minimum $a \land b := \min(a, b )$,
and the maximum $a \lor b := \max(a, b)$.
For a non-negative integer $K \in \mathbb{N}_{>0}$,
we write $[K] = \{1 \ldots, K\}$.
We always work with Borel probability distributions.
Given probability distributions $P, Q$
on $\mathbb{R}^d$, we denote their
trivial coupling $P \otimes Q$,
which is uniquely defined by taking
Borel sets $A, B \subset \R^d$
to $(P \otimes Q)(A\times B) := P(A)Q(B)$.
The support of $P$ is denoted
$\supp(P)$, and is defined to be
the set of all points $x \in \R^d$
such that $P$ assigns
positive value to all open sets containing
$x$.
The $\ell_2$ Euclidean norm
is always written $\|\cdot\|$, without
a subscript.
Given a Borel measurable $\alpha \colon \R^d \to \R^k$
and $p \in [1, \infty)$,
we define the $L^p(P)$ norm as
$$
\|\alpha\|_{L^p(P)} := \Big(\int \|\alpha(x)\|^p\ud P(x) \Big)^{1/p}.
$$ The sup norm $\|\alpha\|_{L^{\infty}(P)}$
is defined to be the essential supremum
of $\|\alpha\|$ with respect to $P$.
For a Borel measurable $\beta \colon \R^d \to \R$,
we will variously write
$$
\int \beta(x) \ud P(x) = \E_P[\beta]
= P(\beta).
$$
We will also write
$$
\Var_P(\beta) := \E_P[(\beta - \E_P[\beta])^2].
$$
The notation
$u \lesssim v$ indicates
$u \leqslant Cv $ for constant $C$; whether
the constant $C$ is numerical
or problem-dependent is a matter of context.
The suppressed constants in our main results
on MID scaling, described in sections~\ref{subsec:main_intro}
and~\ref{subsec:main_maps_densities}
are numerical, while the suppressed
constants for our results on embedded manifolds,
described in section~\ref{subsec:embedded},
depend on the low-dimensional distribution.

Given a metric $\dist_N$ on a set $N \subset \R^d$,
we will write the $\dist_N$-ball of radius $r$
around $z_0 \in N$ as
$$
B_{\dist_N}(z_0, r):=  \{ z \in N \colon \dist_N(z, z_0) < r\}.
$$ If $\dist_N = \|\cdot\|$ and $N = \R^d$,
we simply write $B(z_0, r)$. The
closed ball is written
$$
B_{\dist_N}^{\rm cl}(z_0, r):=  \{ z \in N \colon \dist_N(z, z_0) \leqslant r\},
$$ and again, if $\dist_N =\|\cdot\|$
and $N = \R^d$, we simply
write $B^{\rm cl}(z_0, r)$.

Throughout, we will assume that $\mu, \nu$
are probability measures on $\R^d$,
and we have iid samples $\mathcal{X} := (x_1, \ldots, x_n)
\sim \mu^{\otimes n}$
and $\mathcal{Y} := (y_1, \ldots, y_n) \sim \nu^{\otimes n}$.
Expectations $\E$ without a subscript will always
refer to integration over both of these
samples.

\section{Assumptions and background
on entropic OT}
\label{sec:duality}
In this section, we state our
assumption on the cost function,
and then establish important notation
while reviewing
duality for the entropic OT
problem.
\subsection{Assumptions on cost}

Our results are most
naturally stated for bounded, Lipschitz
costs rather than the squared
Euclidean cost $\|\cdot\|^2$ and so we will work at this level
of generality throughout the remainder
of the paper. 
We make the following formal assumption
on our cost function $c$.
\begin{assumption}[Bounded and uniformly Lipschitz cost]\label{assumption:cost}
We assume
the cost $c$ is Borel
measurable and
$c \in L^{\infty}(\mu \otimes \nu)$.
    By re-scaling the problem we assume without loss of generality
    that
    $$
    \|c\|_{L^{\infty}(\mu \otimes \nu)} \leqslant 1.
    $$
    Also, we assume
    there exists $L > 0$ such that
    for all $y \in \supp(\nu)$,
    $$
    |c(x, y) - c(x', y)|\leqslant L \|x - x'\| \quad\quad
    \forall x, x' \in \supp(\mu),
    $$ and, similarly, for all $x \in \supp(\mu)$,
    $$
    |c(x, y) - c(x, y')|\leqslant L \|y - y'\| \quad\quad
    \forall y, y' \in \supp(\nu).
    $$
\end{assumption}

We remark that in the case where $c$
is the squared Euclidean cost $\|\cdot\|^2$, Assumption~\ref{assumption:cost}
just means the supports of $\mu$ and $\nu$
are compact.

\subsection{Strong duality for entropic OT}
Given the cost function $c$,
we define the population entropic OT
problem as
\begin{equation}\label{eqn:pop_primal}
S_\eps(\mu, \nu) :=
\inf_{\pi \in \Pi(\mu, \nu)}
\big\{ \pi(c) + \eps \KL{\pi }{\mu \otimes \nu}\big\}.
\end{equation}
The empirical entropic OT
problem is then defined
\begin{equation}\label{eqn:emp_primal}
S_\eps(\hat \mu, \hat \nu) :=
\inf_{\pi \in \Pi(\hat \mu, \hat \nu)}
\big \{\pi(c) + \eps \KL{\pi}{\hat\mu \otimes \hat\nu} \big\}.
\end{equation}
There are unique optimal solutions
to the primal problems~\eqref{eqn:pop_primal}
and~\eqref{eqn:emp_primal},
which we write as $\pi_\eps$
and $\hat \pi_\eps$, respectively.
We distinguish between $c$
and $S_\eps(\mu, \nu)$
by referring to the former
as the {\it cost function},
and the latter as the {\it entropic OT value}.
\paragraph*{Strong duality.}Under Assumption~\ref{assumption:cost}
on the cost $c$, a form
of strong duality holds 
for both the population and empirical
entropic OT problems. We refer to~\cite{di2020optimal}
for these results as well as a thorough
review of the literature on this topic.
For the population entropic OT problem,
strong duality is
$$
S_\eps(\mu, \nu) = \sup_{(f, g) \in L^{\infty}(\mu) \times L^{\infty}(\nu)}
\Phi_\eps(f, g) := 
\mu(f) + \nu(g) - \eps (\mu \otimes \nu)(e^{-\eps^{-1}(c - f - g)} - 1).
$$
The function $\Phi_\eps \colon L^{\infty}(\mu)
\times L^{\infty}(\nu) \to \R$
is the population entropic OT dual function.
Similarly for the empirical entropic OT
problem,
strong duality is stated
$$
S_\eps(\hat\mu, \hat\nu) = \sup_{(f, g) \in L^{\infty}(\hat\mu) \times L^{\infty}(\hat\nu)}
\hat\Phi_\eps(f, g) := 
\hat\mu(f) + \hat\nu(g) - \eps (\hat \mu \otimes \hat\nu)(e^{-\eps^{-1}(c - f - g)} - 1).
$$ And this $\hat \Phi_\eps \colon L^{\infty}(\hat \mu) \times L^{\infty}(\hat \nu) \to \R$
is the empirical entropic OT dual function.

\paragraph*{Population and empirical
dual potentials.}
The population and empirical
dual problems
have solutions
$(f_\eps, g_\eps) \in L^{\infty}(\mu)
\times L^{\infty}(\nu)$
and $(\hat f_\eps, \hat g_\eps) \in L^{\infty}(\hat \mu)
\times L^{\infty}(\hat \nu)$, respectively.
These solutions are unique up to
the translation $(f, g) \mapsto 
(f - a, g + a)$ for $a \in \R$,
and we thus specify the solutions
we consider by assuming
$\nu(g_\eps) = 0$ and $\hat \nu(\hat g_\eps) =0$.

\paragraph*{Population and empirical densities.}For notational convenience, we also define the population
density, for $\mu$-almost
every $x$ and $\nu$-almost every $y$, as
\begin{equation}\label{eqn:pop_scaling}
p_\eps(x,y):= \frac{\ud \pi_{\eps}}{\ud (\mu \otimes \nu)}(x,y)
= e^{-\eps^{-1}(c(x, y) - f_\eps(x) - g_\eps(y))}.
\end{equation}
And we similarly define, for all $x \in \X$ and $y \in \Y$, the
empirical density
\begin{equation}\label{eqn:emp_scaling}
\hat p_\eps(x, y) := \frac{\ud \hat \pi_{\eps}}{\ud (\hat \mu
\otimes \hat \nu)}(x, y) 
= e^{-\eps^{-1}(c(x, y) - \hat f_\eps(x) - \hat g_\eps(y))}
\end{equation}

\paragraph*{Marginal constraints
for the dual potentials.} 
The marginal constraints $\pi_\eps \in \Pi(\mu, \nu)$ and 
$\hat \pi_\eps \in \Pi(\hat \mu, \hat \nu)$
in fact define necessary and sufficient
optimality conditions
for the corresponding dual potentials
through equations~\eqref{eqn:pop_scaling}
and~\eqref{eqn:emp_scaling}, respectively.
The resulting system of equations is
sometimes known as the {\it Schrödinger system},
and is fundamental to entropic optimal transport.
For the population potentials,
the marginal constraints imply that
for $\mu$-almost every $x_0$ and $\nu$-almost every $y_0$,
\begin{equation}\label{eqn:pop_marginals}
1 = \int e^{-\eps^{-1}(c(x, y) - f_\eps(x_0) - g_\eps(y))} \ud \nu(y),
\quad \quad
1 = \int e^{-\eps^{-1}(c(x, y) - f_\eps(x) - g_\eps(y_0))} \ud \mu(x).
\end{equation}

The empirical potentials $\hat f_\eps, \hat g_\eps$ satisfy
the analogous marginal equations:
for all $x \in \X$ and $y \in \Y$,
\begin{equation}\label{eqn:emp_marginals}
    1 = \frac{1}{n} \sum_{j = 1}^n e^{-\eps^{-1}(c(x, y_j)
    - \hat f_\eps(x) - \hat g_\eps(y_j))}, \quad \quad
    1 = \frac{1}{n} \sum_{i = 1}^n e^{-\eps^{-1}(c(x_i, y)
    - \hat f_\eps(x_i) - \hat g_\eps(y))}.
\end{equation}

\paragraph*{Canonical extensions
of the dual potentials.} In fact, the marginal equations~\eqref{eqn:emp_marginals}
yield a canonical means of extending the empirical entropic
dual potentials to functions on all of $\R^d$~\cite{berman2020sinkhorn,pooladian2021entropic}.
We observe that solving for $\hat f_\eps$ in~\eqref{eqn:emp_marginals}
yields, for $x \in \mathcal{X}$,
\begin{equation}\label{eqn:feta_extension}
    \hat f_\eps(x) = - \eps \ln \big( \frac{1}{n}
    \sum_{j = 1}^n e^{-\eps^{-1}(c(x, y_j) - \hat g_\eps(y_j)) }\big).
\end{equation} Since
$c$ is defined everywhere, this equation
actually makes sense for all $x \in \supp(\mu)$,
and so we thus {\it define} $\hat f_\eps$
on all of $\supp \mu$ with this formula.
Similarly, we put,
\begin{equation}\label{eqn:geta_extension}
    \hat g_\eps(y) := - \eps \ln \big( \frac{1}{n}
    \sum_{i = 1}^n e^{-\eps^{-1}(c(x_i, y) - \hat f_\eps(x_i)) }\big)
    \quad \quad y \in \supp(\nu).
\end{equation}
Using Equation~\eqref{eqn:emp_scaling},
we extend the empirical density $\hat p_\eps$ to
all of $\supp(\mu)\times \supp(\nu)$ as well.

\section{Main results}\label{sec:main_results}
In this section we describe our main results
on the effective statistical dimension
of entropic optimal transport.
In section~\ref{subsec:main_intro},
we introduce and discuss MID scaling
in the context of entropic OT value estimation.
In section~\ref{subsec:main_maps_densities},
we state our results which establish MID
scaling for entropic OT map and density estimation.
And in section~\ref{subsec:embedded}, we 
show how these results can be strengthened
in the case where one of the measures
is supported on an embedded manifold.
Throughout, we work under Assumption~\ref{assumption:cost}
on the cost function.

\subsection{Introduction to MID scaling
with value convergence}
\label{subsec:main_intro}
In this section, we introduce
the reader to MID scaling in the context
of the convergence of entropic OT 
values.
We state our main result on MID scaling
for values,
Theorem~\ref{thm:main_cost_covering},
give several examples,
and then conclude with discussion
of its statistical significance and tightness.

\paragraph*{MID scaling for value estimation.}
Recall the definition of covering numbers
from~\eqref{eqn:covering_defn},
and that we define
$$
\mathcal{N}(\mu, \delta):= \mathcal{N}(\supp(\mu), \|\cdot\|,  \delta),
\quad \quad 
\mathcal{N}(\nu, \delta):= \mathcal{N}(\supp(\nu), \|\cdot\|, \delta), \quad \quad \delta > 0.
$$

Our main result on entropic OT
value estimation follows.
\begin{theorem}[MID scaling for values]\label{thm:main_cost_covering}
For numerical constants independent of all problem
parameters,
$$
\mathbb{E}[|S_\eps(\hat \mu, \hat \nu) - S_\eps(\mu, \nu)|] 
\lesssim (1 + \eps)
\sqrt{\frac{ \mathcal{N}(\mu, \frac{\eps}{L})\land \mathcal{N}(\nu, \frac{\eps}{L})}{n}}.
$$
\end{theorem}
The only dimensional quantity in this estimate
is contained in the minimum covering
numbers at scale $\eps$, demonstrating
the MID scaling phenomenon.
We emphasize that this result
only requires that the cost $c$
is bounded and Lipschitz (Assumption~\ref{assumption:cost}),
in contrast to the smoothness
assumptions in most of the previous literature.

\paragraph*{Examples of MID scaling.} 
To gain a feeling for Theorem~\ref{thm:main_cost_covering},
let us consider some examples.

\begin{example}[Generic distributions
in $B(0,1)$]
For generic probability distributions, we
can apply Theorem~\ref{thm:main_cost_covering}
with
standard upper bounds
on covering numbers in Euclidean space~\cite[Proposition 4.2.12]{vershynin2018high}.
We find that if $\mu, \nu$ are supported in $B(0, 1) \subset \R^d$,
then for numerical constants independent of all problem parameters,
and for all $\eps > 0$,
$$
\mathbb{E}[|S_\eps(\hat \mu, \hat \nu) - S_\eps(\mu, \nu)|] 
\lesssim (1 + \eps)\cdot \Big(1 + \frac{2L}{\eps} \Big)^{d/2} \cdot \frac{1}{\sqrt{n}}.
$$
Specializing to the case of 
the $\ell_2^2$ cost,
this bound nearly recovers
Theorem~\ref{thm:genevay_rate},
being worse by a factor of $\eps^{-1/2}$
in the case where $d$ is
odd. In fact,
Theorem~\ref{thm:genevay_rate}
applies more generally to costs which are
both Lipschitz and
smooth to degree $\lceil d/2\rceil $~\cite{genevay2019sample}.
The authors of that work
observed empirically that their smoothness assumption seemed
unnecessary and left it as an open problem to 
remove that assumption. 
Theorem~\ref{thm:main_cost_covering} therefore
resolves this problem.\footnote{For large $\eps$,
our bound
diverges
while Theorem~\ref{thm:genevay_rate}
becomes $O(1/\sqrt{n})$,
but it is straightforward to modify our proofs
to fully recover Theorem~\ref{thm:genevay_rate}
in this case.}
\end{example}

In cases where just one of the measures is
assumed to have low intrinsic dimension, we can obtain
bounds which only depend on that measure.
\begin{example}[Semi-discrete]
\label{example:semi_discrete}
Suppose $\nu$ is supported on $K$ points.
Then for numerical constants
independent of all problem parameters,
and for all $\eps > 0$,
$$
\mathbb{E}[|S_\eps(\hat \mu, \hat \nu) - S_\eps(\mu, \nu)|] 
\lesssim (1 + \eps) \cdot \sqrt{\frac{K}{n}}.
$$
\end{example}

\begin{example}[Embedded manifold]
When $\nu$ is supported on a $d_\nu$-dimensional
compact,
smooth, embedded Riemannian
manifold without boundary, we can apply Theorem~\ref{thm:main_cost_covering}
with the bound $\mathcal{N}(\nu, \delta)
\leqslant C_\nu \delta^{-d_\nu}$, valid for some $C_\nu > 0$
and $\delta$ sufficiently small (Proposition~\ref{prop:N_covering_numbers} in
Appendix~\ref{sec:background_manifolds}
formally verifies this fact).
In this case, Theorem~\ref{thm:main_cost_covering}
implies that for $\eps > 0$ sufficiently small and
numerical constants independent
of all problem parameters,
$$
\mathbb{E}[|S_\eps(\hat \mu, \hat \nu) - S_\eps(\mu, \nu)|] 
\lesssim \sqrt{C_\nu} (1 + \eps) \cdot \Big(\frac{L}{\eps} \Big)^{d_\nu/2}
\cdot \frac{1}{\sqrt{n}}.
$$
\end{example}

Because MID scaling involves only
a single distance scale, the above examples
can be generalized to sets which are only
low dimensional at some scales.

\begin{example}[$\delta$-fattening of sets~\cite{weed2019sharp}]
For
$\delta \geqslant 0$ and $A \subset \R^d$,
the $\delta$-fattening of $A$ is
$$
A_\delta := \bigcup_{a \in A} B^{\rm cl}(a, \delta).
$$
Suppose $\supp(\nu) \subset A_\delta$
for some $\delta > 0$ and $A$ such that
$ \mathcal{N}(A, \|\cdot\|, \tau) \leqslant C_A \tau^{-k}$
for all $\tau$ sufficiently small and some
constants $C_A, k \geqslant 0$. Note that
covering numbers on $A_\delta$
can be compared to those on $A$ itself,
since
for $\tau \geqslant \delta$ we have
$\mathcal{N}(A_\delta, \|\cdot\|, \tau)
\leqslant \mathcal{N}(A, \|\cdot\|, \tau - \delta)$.
Hence
for $\eps > L\delta$ 
sufficiently small,
$$
\mathbb{E}[|S_\eps(\hat \mu, \hat \nu) - S_\eps(\mu, \nu)|] 
\lesssim \sqrt{C_A}(1 + \eps) \cdot
\Big(\frac{L}{\eps - L\delta} \Big)^{k/2}
\cdot \frac{1}{\sqrt{n}}.
$$
\end{example}
The above example illustrates
that, so long as the ratio $\eps/L$ is
significantly larger than the fattening scale $\delta$,
the rates are essentially the same as if $\supp(\nu)$
were actually contained in $A$,
despite the fact that $\supp(\nu)$
may be full-dimensional at scale $\delta$. In this way,
the dimension-dependence of MID scaling is oblivious
to features below scale $\eps/L$.

\paragraph*{Discussion of MID scaling.}
We emphasize that MID scaling is similar to, but distinct
from, the intrinsic dimension-dependence
in the un-regularized OT literature~\cite{dudley1969speed,boissard2014mean,weed2019sharp,hundrieser2022empirical}.
While in both settings
the minimum of the intrinsic dimensions
governs the rate of convergence,
the notion of intrinsic dimension is different.
In the un-regularized setting,
intrinsic dimension
is characterized through uniform covering number control at small scales, whereas in MID scaling the covering
numbers play a role, but only at a single distance scale.
The convergence of un-regularized OT is, in fact, adaptive
to multi-scale behavior, but the relevant scales are
determined by the sample size $n$, meaning that
milder covering numbers at some scales
only translate to improved rates while $n$ is not
too large; in fact, such rates are known to be tight~\cite{weed2019sharp}.
In contrast, MID scaling shows that entropic 
OT benefits from better covering number
control at some scales {\it for all sample sizes $n$}.
Essentially, entropic regularization
decouples the sample size and the distance
scale of the problem, allowing for a flexible
trade-off between the intrinsic dimension of
the data distribution and the amount of regularization.\footnote{Thanks to Yann Ollivier for
stimulating comments on this point.}

MID scaling helps clarify the statistical
role of entropic regularization,
showing that beyond its 
well-known computational virtues,
entropy also provides statistical regularization by specifying a distance scale, allowing the user to balance the
intrinsic curse of dimensionality of the data with the statistical
difficulty of the problem.
Because of the manifold hypothesis, which states that natural
data is typically supported on or near
a low-dimensional embedded manifold,
we expect data to have significantly 
smaller intrinsic dimension than extrinsic dimension.
And so a means to flexibly adapt optimal transport
to such intrinsic low-dimensional structure,
particularly approximate low-dimensional structure,
is of major interest.
MID scaling demonstrates that entropic regularization
provides this benefit for optimal transport.

\paragraph*{Remarks on tightness.}
While the upper bounds in this work
substantiate MID scaling and the statistical benefits
of entropic regularization, a complete
statistical account of entropic OT requires
lower bounds.
We leave a thorough study of this interesting
direction to future work, but do give some indications on the tightness
of Theorem~\ref{thm:main_cost_covering}
here.
For this reason, we note that
the CLT for entropic OT
states that
the asymptotic variance of
$\sqrt{n}(S_\eps(\hat \mu, \hat \nu) -
S_\eps(\mu, \nu))$
is $\Var_\mu(f_\eps) + \Var_\nu(g_\eps)$~\cite{gonzalez2023weak}.
In particular, it follows
that
the dependence on $n$ in Theorem~\ref{thm:main_cost_covering}
is optimal, and moreover that
Theorem~\ref{thm:main_cost_covering}
is tight in the case
where $\nu$ is supported on two points (see
Example~\ref{example:semi_discrete}).

However, this example 
doesn't
rule out
the possibility that the 
dependence on the covering numbers
of $\mu$ and $\nu$ can be improved.
To reason about such issues,
we can reduce to lower bounds
for the un-regularized OT problem.
For example, it is known
that the approximation
error $|S_\eps - S_0|$ is
$O(\eps)$, up to logarithmic factors~\cite{genevay2019sample}.
Combining this approximation error
with minimax lower bounds
for estimating un-regularized OT values then
implies
a firm speed limit on statistical
bounds for entropic OT values.
Arguing along these lines,
we show in Appendix~\ref{subsec:tightness}
that Theorem~\ref{thm:main_cost_covering}
implies that entropic estimators
can estimate $W_1$ distances
at the rate $n^{-1/(d+  2)}$,
close to the minimax optimal $n^{-1/d}$
rate~\cite{niles2022estimation},
and that moreover
the covering number dependence
in the bound of Theorem~\ref{thm:main_cost_covering}
cannot be improved in general.

\subsection{MID scaling for maps and densities}
\label{subsec:main_maps_densities}

Theorem~\ref{thm:main_cost_covering} gives
a strong instance
of the MID scaling
phenomenon for
entropic optimal transport,
yet in many applications
of entropic OT, such as 
trajectory reconstruction~\cite{schiebinger2019optimal} and domain adaptation~\cite{flamary2016optimal,seguy2018large},
it is of greater interest
to estimate the solution to the entropic OT
problem than the value of the entropic OT objective alone.
Motivated by this fact, we also study 
the performance of empirical plug-in estimators 
for the problems of entropic OT map
estimation
as well as density estimation,
and develop results analogous to Theorem~\ref{thm:main_cost_covering},
showing that natural plug-in estimators also
sport MID scaling.

We define the entropic OT map 
$$
T_\eps(x) := \E_{\pi_\eps}[y \, |\, x].
$$ Note that the analogous
map defined over $y$ is symmetric
with this one, and so we study only
this case without loss of generality.
The map $T_\eps$ is an entropic analog of the OT map,
and has been the subject of much recent work
as it offers greater
computational
and statistical efficiency
than its un-regularized
counterpart~\cite{pooladian2021entropic,masud2021multivariate,rigollet2022sample,pooladian2023minimax}.
We study the empirical analog of $T_\eps(x)$:
\begin{equation}\label{eqn:empirical_map}
\hat T_\eps(x) := \E_{\hat \pi_\eps}[y \, | \, x] = 
\frac{1}{n} \sum_{j =1 }^n y_j \hat p_\eps(x, y_j).
\end{equation}
We consider $\hat T_\eps$ as an estimator for its
population analog $T_\eps$ and show it enjoys
a $1/\sqrt{n}$ rate of convergence with MID scaling.
\begin{theorem}\label{thm:main_map_covering}
Suppose the diameter of $\supp(\nu)$ is at most $R$.
Then for numerical constants
independent of all problem parameters,
$$
\mathbb{E}[\| \hat T_\eps - T_\eps \|^2_{L^2(\hat \mu)}]
\lesssim R^2\big(1 +  \frac{1}{\eps}
\big)\sqrt{\frac{\mathcal{N}(\mu, \frac{\eps}{L})\land\mathcal{N}(\nu, \frac{\eps}{L})}{n}}.
$$
\end{theorem}
We note that this result
is measured with respect to the empirical
norm $\|\cdot \|_{L^2(\hat \mu)}$,
rather than the population norm
$\|\cdot\|_{L^2(\mu)}$. In the next
section we assume that one distribution is supported on
an embedded manifold and show that this result
can be strengthened to population norm convergence with
fast $1/n$ rates.
As before, we may derive
bounds from this result in cases where
we have {\it a priori} control
on the relevant covering numbers.
Previous work
on the convergence of $\hat T_\eps$
to $T_\eps$ for generic distributions incurred an exponential
dependence on $1/\eps$~\cite{masud2021multivariate,rigollet2022sample}.
And an
 alternative estimator which achieves
sub-exponential dependendence on $1/\eps$
was considered in~\cite{stromme2023sampling},
but at the cost of worse
rates in $n$.
Theorem~\ref{thm:main_cost_covering}
avoids exponential
factors and achieves MID scaling
with a $1/\sqrt{n}$ rate.

We finally remark that Theorem~\ref{thm:main_map_covering}
only uses the boundedness of $y$,
and can be extended to bound
the convergence
of other conditional expectations
$\E_{\pi_\eps}[\alpha(x, y) \, | \, x]$
for bounded $\alpha \colon \R^d \times \R^d \to \R^k$.
Such generality is of interest, since
some researchers study alternative definitions of entropic
OT maps
when the cost is
not $\ell_2^2$.
For example, when $c(x, y) = h(x - y)$
for strictly convex $h$, a very recent work
suggested a different definition of $T_\eps$ 
based on analogy to the un-regularized OT map
in this case~\cite{cuturi2023monge}.
If $h$ satisfies
Assumption~\ref{assumption:cost}
and is additionally strongly convex and differentiable
everywhere, it is not hard
to check that Theorem~\ref{thm:main_map_covering}
also holds for the maps introduced in that work.

\paragraph*{MID scaling for entropic OT density estimation.}
We can also consider estimating
the full entropic OT density $p_\eps$,
defined in~\eqref{eqn:pop_scaling}.
 A natural way to estimate $p_\eps$
 from samples
 is with its plug-in counterpart,
 $\hat p_\eps$ from~\eqref{eqn:emp_scaling}.
We obtain the following result.
\begin{theorem}\label{thm:density_estimation_slow}
For numerical constants independent
of all problem parameters,
$$
\E[\|\hat p_\eps - p_\eps\|_{L^1(\hat \mu \otimes \hat \nu)}] \lesssim \big(1+  \frac{1}{\sqrt{\eps}} 
\big)
\Big(\frac{\mathcal{N}(\mu, \frac{\eps}{L})\land\mathcal{N}(\nu, \frac{\eps}{L})}{n}\Big)^{1/4}.
$$
\end{theorem}
This result shows that MID scaling even
applies when estimating the full density $p_\eps$,
albeit in the empirical $L^1(\hat\mu \otimes \hat\nu)$ norm.
Previously known bounds for the 
convergence of $\hat p_\eps$ to $p_\eps$
incurred an exponential dependence
on $1/\eps$~\cite{rigollet2022sample},
and Theorem~\ref{thm:density_estimation_slow}
removes this dependence while only incurring
MID scaling.

We remark that, unlike
Theorem~\ref{thm:main_cost_covering}
on MID scaling for values,
MID scaling for maps and densities
doesn't yet have analogs in the un-regularized
setting, to the best of our knowledge.
Developing such analogs is
an interesting direction for future work.
In the next section, we demonstrate a
stronger form of MID scaling for these problems
when the low-dimensional distribution
is supported on an embedded manifold.

\subsection{Fast rates with MID scaling on embedded manifolds}
\label{subsec:embedded}

In this section, we describe how we can
strengthen the results from the previous section
by assuming the low-dimensional measure
is supported on an embedded manifold.

\paragraph*{Assumptions.}
We work under the following assumptions,
and for simplicity assume that $\nu$ is the 
low-dimensional measure.
See Appendix~\ref{sec:background_manifolds}
for a review of embedded Riemannian manifolds,
as well as a discussion of the tools
from the theory of random geometric graphs on
embedded manifolds that we use in our proofs.

\begin{assumption}[$\nu$ is supported on an embedded manifold] \label{assum:nu_manifold}
Assume $\nu$ is supported on a compact, smooth, connected,
Riemannian manifold $(N, h)$ of dimension $d_\nu \geqslant 3$ without boundary,
where $N$ is endowed with the submanifold
geometry from its inclusion in $\R^d$.
\end{assumption}

We also assume that $\nu$ is compatible with $N$
in the following sense.
\begin{assumption}[$\nu$ is Lipschitz and non-vanishing]\label{assum:nu}
Assume that $\nu$ has a Lipschitz, non-vanishing density with respect
to the Riemannian volume on $N$.
\end{assumption}

These assumptions represent a natural
instantiation of the manifold hypothesis
while still furnishing enough analytical
structure to yield stronger 
results than in the previous section.
In this section and associated proofs, Assumptions~\ref{assum:nu_manifold}
and~\ref{assum:nu} are made throughout,
in addition to Assumption~\ref{assumption:cost}.

\paragraph*{Suppressing constants depending
on $N$ and $\nu$.}
In contrast to the previous section,
in this section and associated proofs
we will generally suppress constants depending on $N$
and $\nu$.
The suppressed
constants are strictly a function of the intrinsic
dimension $d_\nu$ (and not the extrinsic dimension $d$),
the geometry
of $N$, and uniform bounds on the 
density of $\nu$ and its Lipschitz constant.
In Appendix~\ref{sec:background_manifolds},
where
we collect all the background and necessary
results on embedded manifolds and random geometric graphs,
we describe the relevant geometric
quantities from the manifold $N$.

\paragraph*{Main result on dual potential
convergence.}
The purpose of this section is to show
that under these assumptions
on $\nu$ and $N$,
 we can derive significantly
 stronger instances of MID scaling
 than in the previous section.
 Our main result concerns
the convergence of the empirical dual potentials
$\hat f_\eps, \hat g_\eps$ to their
population counterparts $f_\eps, g_\eps$.
Such convergence is essentially stronger
than that of entropic OT values, maps, and densities, since 
each of those quantities is defined
in terms of the dual potentials.
Since this result
measures convergence in
population norms, we remind
the reader that the out of sample
extensions of $\hat f_\eps, \hat g_\eps$
as defined in~\eqref{eqn:feta_extension} 
and~\eqref{eqn:geta_extension}, respectively,
are in full effect.

\begin{theorem}\label{thm:dual_convergence}
If $\eps/L$ is sufficiently
small,
then
$$
\E\big[\|\hat f_\eps - f_\eps\|^2_{L^2(\mu)}
+ \|\hat g_\eps - g_\eps\|^2_{L^2(\nu)}\big]
\lesssim 
\big( \eps^2 + \frac{1}{\eps^2}\big)
\cdot \Big(\frac{L}{\eps} \Big)^{13 d_\nu + 8} \cdot \frac1n.
$$
We also have convergence
in empirical norms
$$
\E\big[\|\hat f_\eps - f_\eps\|^2_{L^2(\hat\mu)}
+ \|\hat g_\eps - g_\eps\|^2_{L^2(\hat \nu)}\big]
\lesssim  
\big(\eps^2 + \frac{1}{\eps^2}\big)\cdot \Big(\frac{L}{\eps} \Big)^{9d_\nu + 4} \cdot \frac1n.
$$
\end{theorem}
Some remarks on this result are in order.
First, as in all the results
of this section, we require $\eps/L$ to be
sufficiently small; the precise
size is given in Equation~\eqref{eqn:scale}.
Second, the result concerns
the sum of the squared norms,
but this is merely for a convenient
statement, and the proofs for each term
are separate and different.
Third, the empirical norm convergence is
useful for some of our applications below,
while also being an important
preliminary step to the full population
norm results.
Fourth, we remark that,
to the best of our knowledge,
all previously known bounds on entropic OT dual
potentials
incurred an exponential dependence
on $\eps$ ~\cite{luise2019sinkhorn,del2023improved,rigollet2022sample}.
The significance of Theorem~\ref{thm:dual_convergence}
is that it avoids exponential
dependence on $\eps$ while
only incurring minimum intrinsic dimension-dependence,
in accordance with MID scaling.

\paragraph*{MID scaling in the embedded
manifold setting.}
As do all the results
in this section, Theorem~\ref{thm:dual_convergence}
evinces a slightly modified
form of MID scaling as compared
to the previous sections.
In our embedded manifold
setting, the minimum part of MID scaling
simply means that the intrinsic dimension
of $\mu$ doesn't appear in the bounds (remarkably,
nothing about $\mu$
appears in the bound at all);
if $\mu$ were also assumed
to satisfy the above assumptions on a manifold
of dimension $d_\mu$, we could write bounds with an 
appropriate minimum.
In terms of the intrinsic dimension
part of MID scaling,
there is only a superficial
difference, since although the rates
in this section do not explicitly involve
covering numbers at scale $\eps/L$ but instead
 the power $(L/\eps)^{d_\nu}$,
 the two are, in fact, comparable (Proposition~\ref{prop:N_covering_numbers}).
And finally, as we mentioned above, there
are hidden constants depending on $N$ and $\nu$, and especially the intrinsic dimension $d_\nu$,
as
opposed to the previous results which had only numerical constants.

\paragraph*{Applications
of dual convergence.}
We demonstrate the power
of Theorem~\ref{thm:dual_convergence}
with a few applications.
Our first two results
improve upon the map and density
estimation results 
from the previous section.
For each of these applications,
we again require $\eps/L$
to be sufficiently
small, as
specified in Equation~\eqref{eqn:scale}.

Recall the setting of Theorem~\ref{thm:main_map_covering}:
we wish to estimate $T_\eps(x) := \E_{\pi_\eps}[y\, | \, x]$
from samples, and we use the
plug-in estimator $\hat T_\eps(x) :=
\E_{\hat \pi_\eps}[y\, | \, x]$.
Using the canonical extensions of
the empirical dual potentials
$\hat f_\eps,\hat g_\eps$
and correspondingly the empirical 
density $\hat p_\eps$ as in equations~\eqref{eqn:feta_extension},~\eqref{eqn:geta_extension}, and~\eqref{eqn:emp_scaling},
we can extend $\hat T_\eps$
to a map in $L^2(\mu)$,
and consider convergence in
this population norm.
This extended map $\hat T_\eps$ 
can be evaluated in linear time once the empirical dual
potentials are known, and was originally
proposed as a computationally efficient
estimator for the un-regularized OT map~\cite{pooladian2021entropic}.
We obtain the following corollary
of Theorem~\ref{thm:dual_convergence}.
\begin{corollary}\label{cor:manifold_map}
If $\eps/L$ is sufficiently small, and the diameter
of $\supp(\nu)$ is at most $R$, then
$$
\E[\|\hat T_\eps - T_\eps \|^2_{L^2(\mu)}]
\lesssim R^2 \cdot  \big( 1 + \frac{1}{\eps^4}\big)
\cdot \Big(\frac{L}{\eps} \Big)^{11 d_\nu + 4} \cdot \frac1n 
$$
\end{corollary}
Corollary~\ref{cor:manifold_map}
improves over Theorem~\ref{thm:main_map_covering}
in that it achieves a fast $1/n$
rate of convergence in the full population
norm $L^2(\mu)$.

We can also strengthen
our result on density
estimation from the previous section,
Theorem~\ref{thm:density_estimation_slow}.
As in the case of map estimation,
the canonical extensions of
the empirical density $\hat p_\eps$
described in section~\ref{sec:duality},
allow us to consider
$\hat p_\eps$ as an out of sample estimator
for $p_\eps$.

\begin{corollary}\label{cor:fast_density}
If $\eps/L$ is sufficiently small, then
$$
\E[ \|\hat p_\eps - p_\eps\|_{L^2(\mu \otimes \nu)}^2] \lesssim 
\big( 1 + \frac{1}{\eps^4}\big)
\cdot \Big(\frac{L}{\eps} \Big)^{15 d_\nu + 8} 
\cdot \frac1n.
$$
\end{corollary}
This result shows
that, in the embedded manifold setting,
Theorem~\ref{thm:density_estimation_slow}
can be strengthened to full population norm convergence,
$L^2$ rather than $L^1$, and to faster rates.

Our final application
concerns the \emph{bias}
of entropic OT, namely
the quantity $\E[S_\eps(\hat \mu, \hat \nu)]
- S_\eps(\mu, \nu)$.
Recent work established
that the bias converges at a fast $1/n$
rate, yet with
an exponential
dependence on the regularization parameter $\eps$~\cite{rigollet2022sample, del2023improved}.
The following result
removes this exponential dependence
and achieves MID scaling.
We state this result as a Theorem
since, unlike the previous Corollaries,
it doesn't directly follow from Theorem~\ref{thm:dual_convergence}.
\begin{theorem}\label{thm:manifold_bias}
If $\eps/L$ is sufficiently small, then
$$
|\E[S_\eps(\hat \mu, \hat \nu)]-
S_\eps(\mu, \nu)|\lesssim  
\big(\eps^2 + \frac{1}{\eps^2}\big)\cdot \Big(\frac{L}{\eps} \Big)^{9 d_\nu + 4} \cdot \frac1n.
$$
\end{theorem}

In sum, the results in this section demonstrate
a strong form of the MID scaling phenomenon 
in the relatively general setting
of Assumptions~\ref{assum:nu_manifold}
and~\ref{assum:nu}.
These results underscore
the main message of our work:
MID scaling is a generic phenomenon
for entropic OT.
They demonstrate that MID scaling
is certainly not limited to
value estimation or map and density estimation
in empirical norms with slow $1/\sqrt{n}$ rates, but in the embedded manifold setting
even holds
for dual potentials, maps, densities, and biases,
all
with fast $1/n$ rates in population norms.

\section{Preliminary results}\label{sec:preliminary}
In this section, we fix final
pieces of notation
and describe some simple observations that form
the foundation of our approach.

\subsection{Concavity for the empirical
dual}
\label{subsec:concavity}
Given probability measures $P, Q$ on $\R^d$,
we define the entropic OT dual function $\Phi^{PQ}_\eps \colon L^{\infty}(P) \times L^{\infty}(Q) \to \R$ as
$$
\Phi^{PQ}_\eps(f, g) := P(f) + Q(g) - \eps (P \otimes Q)(e^{-\eps^{-1}(c - f - g)} - 1).
$$
With this notation,
we can express the population dual 
$\Phi_\eps = \Phi_\eps^{\mu \nu}$ and the empirical
dual $\hat \Phi_\eps = \Phi_\eps^{\hat \mu \hat \nu}$.
We collect some basic facts about entropic dual
functions below.
\begin{prop}\label{prop:entropic_dual_basics}
Let $P, Q$ be probability measures on $\R^d$.
The corresponding entropic OT dual function $\Phi^{PQ}_\eps$
has the following properties:
\begin{enumerate}
    \item[1.] {\bf (Definition of gradient.)} For every pair $h_1 = (f_1, g_1) \in L^{\infty}(P)
    \times L^{\infty}(Q)$, there exists an 
    element of $L^{\infty}(P) \times L^{\infty}(Q)$
    which we denote
    by $\nabla \Phi_\eps^{PQ}(f_1, g_1)$ such that
    for all $h_0 = (f_0, g_0) \in L^{\infty}(P) \times L^{\infty}(Q)$,
    \begin{align*}
    \langle \nabla \Phi_\eps^{PQ}(h_1), h_0 \rangle_{L^2(P)\times L^2(Q)}
    &= \int f_0(x) \Big(1 - \int e^{-\eps^{-1}(c(x, y) - f_1(x) - g_1(y))}
    \ud Q(y) \Big) \ud P(x) \\
    &+
     \int g_0(y) \Big(1 - \int e^{-\eps^{-1}(c(x, y) - f_1(x) - g_1(y))}
    \ud P(x) \Big) \ud Q(x).
    \end{align*}
    In other words, the gradient of $\Phi_\eps^{PQ}$
    at $(f_1, g_1)$ is
    the marginal error corresponding to $(f_1, g_1)$.
    \item[2.] {\bf (Concavity.)}
    For any two pairs of dual potentials
    $h_0, h_1 \in L^{\infty}(P) \times L^{\infty}(Q)$, we have the inequalities
    \begin{equation}\label{eqn:entropic_dual_concavity_upper}
         \Phi_\eps^{PQ}(h_0) - 
        \Phi_\eps^{PQ}(h_1)
        \leqslant 
        \langle \nabla \Phi_\eps^{PQ}(h_1),
        h_0 - h_1 \rangle_{L^2(P) \times L^2(Q)},
    \end{equation} and 
    \begin{equation}\label{eqn:entropic_dual_concavity_lower}
        \langle \nabla \Phi_\eps^{PQ}(h_0),
        h_0 - h_1 \rangle_{L^2(P) \times L^2(Q)}
        \leqslant \Phi_\eps^{PQ}(h_0) - \Phi_\eps^{PQ}(h_1).
    \end{equation}
    \item[3.] {\bf (Marginal rounding improves the dual objective.)}
    Let $(f, g) \in L^{\infty}(P) \times L^{\infty}(Q)$, and
    set
    $$
    \tilde f(x) := - \eps \log \Big( \int e^{-\eps^{-1}(c(x, y)
    - g(y))} \ud Q(y) \Big).
    $$ Then
    $$
    \Phi_\eps^{PQ}(f, g) \leqslant \Phi_\eps^{PQ}(\tilde f, g).
    $$ And analogously when marginal rounding in the $g$ variable.
\end{enumerate}
\end{prop}

\begin{proof}[Proof of Proposition~\ref{prop:entropic_dual_basics}]
The first item is a definition.
For the second item, write
$$
\Phi_\eps^{PQ}(f, g) = (P \otimes Q)\big(f(x) + g(y) - \eps e^{-\eps^{-1}
(c(x, y) - f(x) - g(y))} \big) + \eps,
$$ and observe that the function $t\mapsto  t - \eps e^{-\eps^{-1}(c -
t)}$ is concave. The result follows by applying the
concavity of the integrand pointwise and collecting terms.
The third item follows from the first two.
\end{proof}

\subsection{Pointwise
control and log-Lipschitz bounds}
\label{subsec:pointwise}
In this section we describe
some simple yet powerful implications
of Assumption~\ref{assumption:cost}.

The following pointwise
bounds appear in a number of places~\cite{mena2019statistical,di2020optimal,rigollet2022sample}, and so their
proof is omitted.
\begin{prop}[Pointwise control 
on dual potentials]\label{prop:dual_pointwise_control}
Under our normalization conventions that $\nu(g_\eps) = 0$
and $\hat \nu(\hat g_\eps) = 0$, 
we have the uniform bounds
$$
\|\hat f_\eps \|_{L^{\infty}(\mu)}, \, \, \|\hat g_\eps \|_{L^{\infty}(\nu)}
\leqslant 2,
\quad \quad
\|f_\eps\|_{L^{\infty}(\mu)}, \, \, \|g_\eps\|_{L^{\infty}(\nu)}
\leqslant 1.
$$
\end{prop}

The following Lipschitz bounds are also well-known~\cite{di2020optimal},
but since they are
the foundation of our approach
we include
their proof.
\begin{prop}[Lipschitz bounds]
\label{prop:consequences_of_lipschitz_cost}
The population dual potentials $f_\eps$ and $g_\eps$
are $L$-Lipschitz over $\supp(\mu)$ and $\supp(\nu)$,
respectively.
The extended empirical dual potentials
$\hat f_\eps$ and $\hat g_\eps$ are also $L$-Lipschitz
over $\supp(\mu)$ and $\supp(\nu)$, respectively.
In particular,
the population density $p_\eps$ and
the extended empirical density $\hat p_\eps$ are
each $\frac{2L}{\eps}$-log-Lipschitz in each
of their variables
over $\supp(\mu) \times \supp(\nu)$.
\end{prop}

\begin{proof}
By the marginal constraints~\eqref{eqn:pop_marginals},
we have that for all $x \in \supp(\mu)$,
$$
e^{-\frac{1}{\eps} f_\eps(x)}
= \int e^{-\frac{1}{\eps}(c(x, y) - g_\eps(y)} \ud \nu(y).
$$ Using Assumption~\ref{assumption:cost} on the cost $c$
being Lipschitz, we see that for $x, x' \in \supp(\mu)$
$$
e^{-\frac{1}{\eps}f_\eps(x)} \leqslant e^{-\frac{1}{\eps}f_\eps(x') + \frac{L}{\eps}\|x - x'\|}
$$ By symmetry it follows that
$$
|f_\eps(x) - f_\eps(x')| \leqslant L \|x - x'\|.
$$ The analogous argument establishes
the claim for $g_\eps$, and also proves
the claims for the extended empirical
potentials $\hat f_\eps$ and $\hat g_\eps$.

For the claims about the densities, note that
$$
\log p_\eps(x, y) = -\frac{1}{\eps}\big(
c(x, y) - f_\eps(x) - g_\eps(y) \big),
$$ and so the claim follows from the previous
ones for $f_\eps$ and $g_\eps$, as well as Assumption~\ref{assumption:cost}.
Similarly for $\log \hat p_\eps$.
\end{proof}

\section{Proofs of MID scaling with
slow rates}\label{sec:proof_overview}
In this section, we give our proofs
of MID scaling for values, Theorem~\ref{thm:main_cost_covering},
MID scaling for maps in empirical norms,
Theorem~\ref{thm:main_map_covering},
and MID scaling for densities in empirical norms,
Theorem~\ref{thm:density_estimation_slow}.
The proofs are organized in the following manner.
In section~\ref{subsec:value_proof},
we give a detailed exposition
of our approach in the case
of value estimation.
In section~\ref{subsec:sub_exp},
we give the proof of one of our main
technical lemmas, and the source of MID scaling,
on sub-exponential bounds for $\|p_\eps\|_{L^2(\mu \otimes \nu)}^2$.
In section~\ref{subsec:empirical_map_proofs},
we show how to reduce map estimation to
the estimates from the previous sections,
and do the same for density estimation
in section~\ref{subsec:empirical_density_proofs}.
\subsection{Proof of value rate
and overview of technical approach}
\label{subsec:value_proof}
In this section, we illustrate the main ideas
of our technical approach by describing our proof
of Theorem~\ref{thm:main_cost_covering} on the
convergence of entropic OT values.
To situate our approach, note that
most existing statistical
work on entropic OT studies
smooth cost functions and shows
that the smoothness of the cost
functions implies smoothness
for the dual potentials,
leading to small function classes
over which empirical processes can be well-controlled~\cite{genevay2019sample,luise2019sinkhorn,mena2019statistical,del2023improved,bayraktar2022stability}.
The present work builds on a different
approach, developed in~\cite{rigollet2022sample},
which entirely avoids empirical process theory
by using strong concavity of the empirical 
dual objective. The benefit
of this approach is that it is simple, requires no smoothness assumptions 
on the cost, leads
to dimension-free bounds,
and can be used to prove fast
rates for most entropic OT quantities~\cite{rigollet2022sample}.
However, because it is dimension-free, it is not suited to a fine understanding
of the dimension-dependence of entropic OT,
and in particular
it incurs exponential factors
of $1/\eps$
at many points in the arguments.
The main technical goal of this work
is to refine this strong concavity of the
empirical dual approach 
by replacing
the exponential
dependence on $1/\eps$ with a very fine
dependence on the dimension of the problem:
MID scaling.

To this end,
note that
the exponential dependence of~\cite{rigollet2022sample}
primarily arises from two separate sources:
first, the strong concavity
of the empirical dual objective, and second,
pointwise bounds on the entropic density $p_\eps$.
For the results in this section, we avoid the exponential factors
arising from strong concavity 
by using more delicate arguments
which rely on only concavity,
and remove the exponential
factors from the pointwise
bounds on the entropic density $p_\eps$
with a novel estimate (Lemma~\ref{lem:variance_control} below)
which is itself the source of MID scaling.

To present the ideas as clearly
as possible, we will suppress
numerical constants with the notation $u \lesssim v$;
they can easily be extracted from the proof.
Recall that $\Phi_\eps$
is the population entropic OT dual function,
and $\hat \Phi_\eps$
is its empirical counterpart.
Then Theorem~\ref{thm:main_cost_covering} concerns the
quantity
$$
\E[|S_\eps(\hat \mu, \hat\nu) - S_\eps(\mu, \nu)|]
= \E[|\hat \Phi_\eps(\hat f_\eps, \hat g_\eps)
- \Phi_\eps(f_\eps, g_\eps)|].
$$
To bound this quantity, we first decompose it 
in the following manner:
\begin{align*}
\E[|\hat \Phi_\eps(\hat f_\eps, \hat g_\eps) -
\Phi_\eps(f_\eps, g_\eps)|]
&\leqslant \E[|\hat \Phi_\eps(\hat f_\eps, \hat g_\eps)
- \hat \Phi_\eps(f_\eps, g_\eps)|]
+ \E[|\hat \Phi_\eps(f_\eps, g_\eps) - \Phi_\eps(f_\eps, g_\eps)| ]\\
&= \E[\hat \Phi_\eps(\hat f_\eps, \hat g_\eps) - \hat \Phi_\eps(f_\eps, g_\eps)] + \E[|\hat \Phi_\eps(f_\eps, g_\eps) - \Phi_\eps(f_\eps, g_\eps)|].
\end{align*}
For convenience, we refer to the first term as a bias term,
and the second as a variance term.
Note that our usage of ``bias" and ``variance"
in this context is not intended to be standard, merely
suggestive.
\paragraph*{Variance term.} If we just examine this
variance term, we can compute:
\begin{align*}
\E[&|\hat \Phi_\eps(f_\eps, g_\eps) - \Phi_\eps(f_\eps, g_\eps)|] \\
    &= \E[|(\hat \mu - \mu)(f_\eps) + (\hat \nu
    - \nu)(g_\eps)
    - \eps (\hat \mu \otimes \hat \nu - \mu \otimes \nu)(p_\eps)|] \\
    &\leqslant \E[\big((\hat \mu - \mu)(f_\eps) + (\hat \nu
    - \nu)(g_\eps)
    - \eps (\hat \mu \otimes \hat \nu - \mu \otimes \nu)(p_\eps)\big)^2]^{1/2}\\
    &\lesssim \big(\E[\big((\hat \mu - \mu)(f_\eps)
    \big)^2] + \E[\big((\hat \nu
    - \nu)(g_\eps)\big)^2]
    + \eps^2 \E[\big((\hat \mu \otimes \hat \nu - \mu \otimes \nu)(p_\eps)\big)^2] \big)^{1/2} \\
    &\leqslant
    \E[\big((\hat \mu - \mu)(f_\eps)
    \big)^2]^{1/2} + \E[\big((\hat \nu
    - \nu)(g_\eps)\big)^2]^{1/2}
    + \eps \E[\big((\hat \mu \otimes \hat \nu - \mu \otimes \nu)(p_\eps)\big)^2]^{1/2}.
\end{align*}
For the first two terms, we use
the pointwise boundedness of the dual potentials
from Proposition~\ref{prop:dual_pointwise_control}
to yield
$$
\E[\big((\hat \mu - \mu)(f_\eps)
    \big)^2]^{1/2} + \E[\big((\hat \nu
    - \nu)(g_\eps)\big)^2]^{1/2} = 
    \sqrt{\frac{\Var_\mu(f_\eps)}{n}}
    + \sqrt{\frac{\Var_\nu(g_\eps)}{n}}
    \lesssim \frac{1}{\sqrt{n}}.
$$
For the third term, the marginal constraints
for $p_\eps$ from~\eqref{eqn:pop_marginals}
allow us to cancel cross-terms and compute
\begin{align*}
    \E[(\hat \mu \otimes \hat \nu)(p_\eps - (\mu\otimes \nu)(p_\eps))^2]^{1/2}
    &=\E[(\hat \mu \otimes \hat \nu)(p_\eps - 1)^2]^{1/2} \\
    &= \Big(\frac{1}{n^4} \sum_{i,j,k,l = 1}^n
    \E[(p_\eps(x_i,y_j) - 1)(p_\eps(x_k, y_l) - 1)]
    \Big)^{1/2}\\
    &= \frac{1}{n}\sqrt{\Var_{\mu \otimes \nu}(p_\eps)} \leqslant 
    \frac{1}{n} \|p_\eps\|_{L^2(\mu \otimes \nu)}.
\end{align*}
We note that this kind of calculation is used
repeatedly throughout the paper.
We finally arrive at the bound
\begin{align}
\E[|\hat \Phi_\eps(f_\eps, g_\eps) - \Phi_\eps(f_\eps, g_\eps)|] &\lesssim 
\frac{1}{\sqrt{n}} + \frac{\eps\|p_\eps\|_{L^2(\mu \otimes \nu)}}{n} \nonumber \\
&\lesssim (1 + \eps) \frac{\|p_\eps\|_{L^2(\mu \otimes \nu)}}{\sqrt{n}},
\label{eqn:overview_variance_final}
\end{align} 
where the last inequality follows
because $(\mu \otimes \nu)(p_\eps) = 1$ 
implies $1 \leqslant \|p_\eps\|_{L^2(\mu \otimes \nu)}$ by Cauchy-Schwarz.
\paragraph{Bias term.}

We use
concavity of the empirical dual to control the bias term
$\E[\hat\Phi_\eps(\hat f_\eps, \hat g_\eps) - \hat\Phi_\eps(f_\eps, g_\eps)]$
by a quantity involving 
$\|p_\eps\|_{L^2(\mu \otimes \nu)}$ again.
Indeed, by Proposition~\ref{prop:entropic_dual_basics}
on
concavity of $\hat \Phi_\eps$ and
Proposition~\ref{prop:dual_pointwise_control}
on pointwise
control for the dual potentials,
we have
\begin{align*}
\E [\hat \Phi_\eps(\hat f_\eps, \hat g_\eps)
- \hat \Phi_\eps(f_\eps, g_\eps)]& \leqslant 
\E[ \langle \nabla \hat \Phi_\eps(f_\eps, g_\eps), (\hat f_\eps
- f_\eps, \hat g_\eps - g_\eps) \rangle_{L^2(\hat \mu) \times
L^2(\hat \nu)}] \\
&\leqslant 
\E\big[ \|\nabla \hat \Phi_\eps (f_\eps, g_\eps)\|_{L^2(\hat \mu) \times 
L^2(\hat \nu)} \cdot 
\| (\hat f_\eps
- f_\eps, \hat g_\eps - g_\eps)\|_{L^2(\hat \mu) \times L^2( \hat \nu)}\big] \\
&\lesssim \E\big[\|\nabla \hat \Phi_\eps (f_\eps, g_\eps)\|_{L^2(\hat \mu) \times 
L^2(\hat \nu)}^2  \big]^{1/2}.
\end{align*}
Using the marginal
constraints~\eqref{eqn:pop_marginals} to cancel
cross terms once more, we find
\begin{align*}
    \E[\|\nabla \hat \Phi_\eps(f_\eps, g_\eps)\|_{L^2(\hat \mu)
    \times L^2(\hat \nu)}^2]
    &= \E\Big[\frac{1}{n} \sum_{i = 1}^n
    \Big(\frac1n \sum_{j = 1}^n  p_\eps(x_i, y_j) - 1\Big)^2\\
    &+ \frac{1}{n} \sum_{j = 1}^n \Big(\frac1n \sum_{i = 1}^n
    p_\eps(x_i, y_j) - 1 \Big)^2\Big] \\
    &= \frac{1}{n^3}
    \sum_{i, j_0, j_1 = 1}^n
    \E[(p_\eps(x_i, y_{j_0}) - 1)(p_\eps(x_i,
    y_{j_1}) - 1)] \\
    &+ \frac{1}{n^3}
    \sum_{i_0, i_1, j = 1}^n
    \E[(p_\eps(x_{i_0}, y_{j}) - 1)(p_\eps(x_{i_1},
    y_{j}) - 1)] \\
    &= \frac{2}{n}
    \Var_{\mu \otimes \nu}(p_\eps)\lesssim \frac{\|p_\eps\|_{L^2(\mu \otimes \nu)}^2}{n}.
\end{align*}
We arrive at an important, if elementary,
bound that we will use frequently
throughout this work,
\begin{equation}\label{eqn:gradient_calc}
    \E[\|\nabla \hat \Phi_\eps(f_\eps, g_\eps)\|_{L^2(\hat \mu)
    \times L^2(\hat \nu)}^2]
    \lesssim \frac{\|p_\eps\|_{L^2(\mu \otimes \nu)}^2}{n}.
\end{equation}
In particular, the bias is bounded as
\begin{equation}\label{eqn:overview_bias_final}
\E [\hat \Phi_\eps(\hat f_\eps, \hat g_\eps)
- \hat \Phi_\eps(f_\eps, g_\eps)]
\lesssim \frac{\|p_\eps\|_{L^2(\mu \otimes \nu)}}{\sqrt{n}}.
\end{equation}

\paragraph*{Controlling $\|p_\eps\|_{L^2(\mu \otimes \nu)}$.}
We can
 combine~\eqref{eqn:overview_variance_final}
and~\eqref{eqn:overview_bias_final} to obtain
\begin{equation}\label{eqn:overview_absolute_difference_final}
\E[|S_\eps(\hat \mu, \hat \nu) - S_\eps(\mu, \nu)|]
\lesssim ( 1 + \eps)\frac{\|p_\eps\|_{L^2(\mu \otimes \nu)}}{\sqrt{n}}.
\end{equation}
We have thus reduced the problem to
terms with a dimension-free dependence on the sample
size $n$, and with the dimension and $\eps$-dependence residing fully 
in the term $(1 + \eps)\|p_\eps\|_{L^2(\mu\otimes \nu)}$.
Up to this point, the differences
between the techniques introduced in~\cite{rigollet2022sample}
and our approach are essentially contained
in two steps, each taken
to avoid factors of $e^{1/\eps}$.
First, we control the bias
term with concavity rather than strong
concavity. And second, we refrain
from using the
pointwise control of
Proposition~\ref{prop:dual_pointwise_control}
to bound the $p_\eps$ terms,
which would
result in the bound $\|p_\eps\|_{L^{\infty}(\mu \otimes \nu)} \leqslant e^{C/\eps}$ for a constant $C$.
Such pointwise
and exponential bounds
were, to the best of our knowledge,
all that was known about $p_\eps$ until this work.

One of the core technical
innovations in this work, and indeed
the source of MID scaling,
is to provide the following
sub-exponential, dimension-dependent
bound on $p_\eps$.

\begin{lemma}[MID scaling for the norm of entropic densities]\label{lem:variance_control}
We have
$$
\|p_\eps\|_{L^2(\mu \otimes \nu)}^2 \lesssim  \mathcal{N}(\mu, \frac{\eps}{L}) \land \mathcal{N}(\nu, \frac{\eps}{L}).
$$
\end{lemma}

This result is proven, in the following
section, 
by using the marginal constraints
that $p_\eps$ satisfies in Equation~\eqref{eqn:pop_marginals}
to gain pointwise control on $p_\eps$ that, when integrated,
yields the covering numbers.
In the embedded manifold setting of section~\ref{sec:proofs_fast_rates},
the $L^2$ bound of Lemma~\ref{lem:variance_control}
can be strengthened to pointwise 
$L^{\infty}$ control,
using
the same technique (Lemma~\ref{lem:manifold_density_sup_control}).
The proof technique is used again in section~\ref{subsec:reduction},
and we indeed expect that it will
be of some broader utility. For example,
it may be helpful for
calculating quantities
arising in the asymptotic distributions
of entropic OT quantities~\cite{gonzalez2023weak}.

For our subsequent results,
it is convenient to record the following
Lemma, which results from 
 combining Equation~\eqref{eqn:overview_bias_final}
with Lemma~\ref{lem:variance_control}.

\begin{lemma}[MID scaling
for the bias]\label{lem:bias_covering}
We have
$$
\E[S_\eps(\hat \mu, \hat \nu)
- S_\eps(\mu, \nu)] \lesssim 
\sqrt{\frac{\mathcal{N}(\mu, \frac{\eps}{L}) \land \mathcal{N}(\nu, \frac{\eps}{L})}{n}}.
$$
\end{lemma}

\paragraph*{Summary of the technical approach.}
In summary, we decomposed the difference $|S_\eps(\hat \mu, \hat \nu)
- S_\eps(\mu, \nu)|$ into a bias term and a 
variance term. For the variance term, an elementary
calculation allows us to control it
by something with the dimension
of the problem residing purely in the
$\|p_\eps\|_{L^2(\mu \otimes \nu)}$ term.
And for the bias term, a judicious use of concavity
allows us to control it by
another expression with
 the dimension of the problem purely
 in the $\|p_\eps\|_{L^2(\mu \otimes \nu)}$ term.
We then conclude with Lemma~\ref{lem:variance_control}.
The proofs of MID scaling with
slow rates
for maps and densities, Theorem~\ref{thm:main_map_covering}
and Theorem~\ref{thm:density_estimation_slow},
use additional estimates that also follow the pattern of splitting
into bias and variance terms (loosely construed) to then
reduce to Lemma~\ref{lem:variance_control},
and are included in sections~\ref{subsec:empirical_map_proofs}
and~\ref{subsec:empirical_density_proofs},
respectively.

\subsection{MID scaling for
the norm of entropic densities}
\label{subsec:sub_exp}
In this section, we will prove Lemma~\ref{lem:variance_control},
which gives sub-exponential,
MID scaling-type bounds
on the entropic density $p_\eps$.
To this end, the following fact will be useful.
Recall that we write the closed ball of size $\delta$
at $z \in \R^d$ as 
$$
B^{\rm cl}(z, \delta) := \big\{z' \in \R^d \colon 
\|z - z'\|\leqslant \delta \big\}.
$$
\begin{prop}[Average inverse mass is bounded
by the covering number]\label{prop:inverse_mass_by_covering}
Suppose $P$ is a compactly supported
probability measure on $\R^d$.
Then for all $\delta > 0$,
$$
\int P(B^{\rm cl}(z, \delta))^{-1} \ud P(z) \leqslant 
\mathcal{N}(P, \delta/4).
$$
\end{prop}
To give the proof of this Proposition, recall
that for $\tau > 0$, a {\it proper} $\tau$-covering
of a set $A \subset \R^d$, is a 
covering of $A$ at scale $\tau$ with centers contained
within $A$. We write the minimal size of 
a proper covering
with respect to the Euclidean norm
at scale $\tau > 0$ as $\mathcal{N}^{\rm pr}(A, \|\cdot\|, \tau)$.
Note that proper covering numbers 
are comparable to vanilla covering numbers since
$\mathcal{N}^{\rm pr}(A, \|\cdot\|, \tau)
\leqslant \mathcal{N}(A, \|\cdot\|, \tau/2)$.
\begin{proof}[Proof of Proposition~\ref{prop:inverse_mass_by_covering}]
Let $z_1, \ldots, z_K \in \supp(P)$ be a proper $\delta/2$
covering of $\supp(P)$ achieving $K = \mathcal{N}^{\rm pr}(\supp(P), \|\cdot\|, \delta/2)$.
Since this is a proper covering, we know that
for all $k \in [K]$,
$P(B^{\rm cl}(z_k, \delta/2)) > 0$.
The triangle inequality then implies
$$
z \in B^{\rm cl}(z_k, \delta/2)
\implies 0 < P(B^{\rm cl}(z_k, \delta/2))
\leqslant P(B^{\rm  cl}(z, \delta)).
$$ Thus
\begin{align*}
\int P(B^{\rm cl}(z, \delta))^{-1} \ud P(z) 
&\leqslant \sum_{k = 1}^K
\int_{B^{\rm cl}(z_k, \delta/2)}
P(B^{\rm cl}(z, \delta))^{-1} \ud P(z) \\
&\leqslant 
\sum_{k = 1}^K
\int_{B^{\rm cl}(z_k, \delta/2)}
P(B^{\rm cl}(z_k, \delta/2))^{-1} \ud P(z) \\
&= K  = \mathcal{N}^{\rm pr}(P, \delta/2)
\leqslant \mathcal{N}(P, \delta/4).
\end{align*}
\end{proof}

\begin{proof}[Proof of Lemma~\ref{lem:variance_control}]
We show the bound for $\mathcal{N}(\nu, \frac{\eps}{L})$,
the bound for $\mathcal{N}(\mu, \frac{\eps}{L})$ follows
by symmetry.
We proceed via the marginal constraints on $p_\eps$ from~\eqref{eqn:pop_marginals}: for all $x \in \supp(\mu)$
and $y \in \supp(\nu)$,
\begin{equation}\label{eqn:relative_marginal}
1 = p_\eps(x, y) \int  \frac{p_\eps(x, y')}{p_\eps(x, y)}
\ud \nu(y').
\end{equation} Applying Proposition~\ref{prop:consequences_of_lipschitz_cost}
to
Equation~\eqref{eqn:relative_marginal} yields,
\begin{align*}
1 &\geqslant  p_\eps(x, y) \int_{B^{\rm cl}(y, \frac{4\eps}{L})}
 \frac{p_\eps(x, y')}{p_\eps(x, y)}\ud \nu(y')  \\
& \geqslant p_\eps(x, y) \int_{B^{\rm cl}(y, \frac{4\eps}{L})}
e^{- \frac{2L}{\eps}\|y'- y\|} \ud \nu(y') \\
&\geqslant
e^{-8} \cdot p_\eps(x, y)\cdot \nu(B^{\rm cl}(y, \frac{4\eps}{L})).
\end{align*}
Since $y \in \supp(\nu)$, 
$\nu(B^{\rm cl}(y, \frac{4\eps}{L})) > 0$,
so we can rearrange this inequality and apply
the marginal 
 constraints~\eqref{eqn:pop_marginals} once more
 to yield
\begin{align*}
    \|p_\eps\|_{L^2(\mu \otimes \nu)}^2 &=
    \int p_\eps(x, y)^2 \ud(\mu \otimes \nu)(x, y) \\
    &\lesssim \int  \nu(B^{\rm cl}(y, \frac{4\eps}{L}))^{-1}
    p_\eps(x, y) \ud (\mu \otimes \nu)(x, y) \\
    &= \int \nu(B^{\rm cl}(y, \frac{4\eps}{L}))^{-1} \ud \nu(y).
\end{align*}
The result follows from Proposition~\ref{prop:inverse_mass_by_covering}.
\end{proof}

\subsection{MID scaling for maps in empirical norm}
\label{subsec:empirical_map_proofs}
As we discuss in the previous
section, to prove Theorem~\ref{thm:main_map_covering}
and Theorem~\ref{thm:density_estimation_slow},
we reduce to the bound on the bias
from
Lemma~\ref{lem:bias_covering}.
We emphasize that bounds on entropic
maps and densities with exponential
dependence on $1/\eps$ are known 
from~\cite{rigollet2022sample}.
To prove such bounds without incurring
exponential factors requires further
work beyond the previous sections.
As in some previous works, we reduce from
the error of map estimation
to the error of value estimation;
the difference in our approach is that we avoid
factors of the KL-divergence~\cite{pooladian2023minimax}
and it applies to the empirical entropic OT map
rather than a modified estimator~\cite{stromme2023sampling}.

The idea is to consider
the following rounded dual potential,
defined for all $x \in \supp(\mu)$ as 
\begin{equation}\label{eqn:tilde_f}
\tilde{f}_\eps(x) := - \eps \log \Big( \frac{1}{n}
\sum_{j = 1}^n e^{-\eps^{-1}(c(x, y_j) - g_\eps(y_j))} \Big).
\end{equation}
Let the corresponding density
be defined for all $x \in \supp(\mu)$ 
and $y \in \supp(\nu)$ as
$\tilde{p}_\eps(x, y):=
e^{-\eps^{-1}(c(x, y) - \tilde f_\eps(x) - g_\eps(y))}$;
note that $\tilde p_\eps$ is such that
for all $x \in \supp(\mu)$,
$$
1 =\frac1n \sum_{j = 1}^n \tilde p_{\eps}(x, y_j).
$$ This fact allows us to compare
$\frac1n \hat p_\eps(x, \cdot)$
with $\frac 1n \tilde p_\eps(x, \cdot)$ as probability
distributions on $\Y$.
The following Lemma is then 
a consequence of
Pinsker's inequality,
cancellation of some terms, and a crucial
use of the monotonicity of the empirical
dual under marginal rounding from Proposition~\ref{prop:entropic_dual_basics}
to replace $\tilde f_\eps$
with its un-rounded counterpart $f_\eps$.

\begin{lemma}\label{lem:map_mse_kl}
Let $\tilde f_\eps$ and $\tilde p_\eps$ be as above, and suppose the diameter
of $\supp(\nu)$ is at most $R$. Then
$$
\E\Big[ \Big\|\hat T_\eps - \frac{1}{n} \sum_{j = 1}^n y_j
\tilde p_\eps(x, y_j) \Big\|_{L^2(\hat \mu)}^2 \Big] 
\lesssim  \frac{R^2}{\eps} \mathbb{E}[
S_\eps(\hat \mu, \hat \nu) - S_\eps(\mu, \nu)].
$$
\end{lemma}

However, this lemma
does not yet yield Theorem~\ref{thm:main_map_covering}
as it involves the rounded density $\tilde p_\eps$.
The following calculation addresses
this issue.

\begin{lemma}\label{lem:map_mse_massage}
Let $\tilde f_\eps$ and $\tilde p_\eps$
be as above. Then
$$
\E[\|\hat T_\eps - T_\eps\|_{L^2(\hat \mu)}^2] 
\lesssim \E\Big[ \Big\|\hat T_\eps - \frac{1}{n} \sum_{j = 1}^n y_j
\tilde p_\eps(x, y_j) \Big\|_{L^2(\hat \mu)}^2 \Big] + R^2\cdot \frac{\mathcal{N}(\mu, \frac{\eps}{L}) \land \mathcal{N}(\nu, \frac{\eps}{L})}{n}.
$$
\end{lemma}

We first dispatch Theorem~\ref{thm:main_map_covering},
and then prove these lemmas.

\begin{proof}[Proof of Theorem~\ref{thm:main_map_covering}]
Combining Lemma~\ref{lem:bias_covering}
with the above two lemmas yields
$$
\E[\|\hat T_\eps - T_\eps\|_{L^2(\hat \mu)}^2]  
\lesssim R^2 \Big\{\frac{1}{\eps}
\cdot \sqrt{\frac{\mathcal{N}(\mu, \frac{\eps}{L}) \land \mathcal{N}(\nu, \frac{\eps}{L})}{n}}
+
 \frac{\mathcal{N}(\mu, \frac{\eps}{L}) \land \mathcal{N}(\nu, \frac{\eps}{L})}{n} \Big\}.
$$ Now observe that the statement
of Theorem~\ref{thm:main_map_covering}
is trivial when $\mathcal{N}(\mu, \frac{\eps}{L}) \land \mathcal{N}(\nu, \frac{\eps}{L}) \geqslant n$, since
we always have the bound $2R^2$.
Hence it suffices to consider the case 
$\mathcal{N}(\mu, \frac{\eps}{L}) \land \mathcal{N}(\nu, \frac{\eps}{L}) \leqslant n$, in which case we may conclude from
the above inequality.
\end{proof}

\begin{proof}[Proof of Lemma~\ref{lem:map_mse_kl}]
To prove this result,
let us first assume that
$0 \in \supp(\nu)$
so that for each $y \in \supp(\nu)$,
we have
$\|y\|\leqslant R$. Then for each $x \in \X$,
we apply triangle inequality and
Pinsker's inequality to see that
\begin{align*}
    \Big\| \hat T_\eps(x) - \frac{1}{n} \sum_{j = 1}^n 
    y_j \tilde p_\eps(x, y_j)\Big\| 
    &\leqslant 
    \frac{R}{n} \sum_{j = 1}^n |\tilde p_\eps(x, y_j)
    - \hat p_\eps(x, y_j)| \\
    &\lesssim R \sqrt{\KL{\frac1n \hat p_\eps(x, \cdot)}{\frac 1n\tilde p_\eps(x, \cdot)}}.
\end{align*}
Therefore,
\begin{align*}
\Big\| \hat T_\eps - \frac{1}{n} \sum_{j = 1}^n 
    y_j  \tilde p_\eps(x,y_j)&\Big\|^2_{L^2(\hat \mu)}
    \lesssim R^2 \frac{1}{n}
    \sum_{i = 1}^n\KL{\frac1n \hat p_\eps(x_i, \cdot)}{\frac 1n\tilde p_\eps(x_i, \cdot)} \\
    &= \frac{R^2}{\eps} \sum_{i, j = 1}^n \frac{1}{n^2} \hat p_\eps(x_i, y_j)(
    \hat f_\eps(x_i) - \tilde f_\eps(x_i) + \hat g_\eps(y_j) - g_\eps(y_j)) \\
    &= \frac{R^2}{\eps} \big(\hat \mu(
    \hat f_\eps - \tilde f_\eps)
    + \hat \nu (\hat g_\eps - g_\eps) \big),
\end{align*} where we used
the marginal constraints of $\hat p_\eps$ for the
last step. Using the notation,
as in the previous section,
$\hat \Phi_\eps := \Phi_\eps^{\hat \mu \hat \nu}$,
we recognize the result as a difference of $\hat \Phi_\eps$,
and apply Proposition~\ref{prop:entropic_dual_basics}
on marginal rounding improving dual objective value,
to conclude that
\begin{align*}
\Big\| \hat T_\eps - \frac{1}{n} \sum_{j = 1}^n 
    y_j  \tilde p_\eps(x,y_j)\|^2_{L^2(\hat \mu)}
    &\lesssim 
    \frac{R^2}{\eps} \big(\hat \mu(
    \hat f_\eps - \tilde f_\eps)
    + \hat \nu (\hat g_\eps - g_\eps) \big)\\
    &=
    \frac{R^2}{\eps} \big(
    \hat \Phi_\eps(\hat f_\eps, \hat g_\eps) - 
    \hat \Phi_\eps(\tilde f_\eps, g_\eps)\big) \\
    & \leqslant 
    \frac{R^2}{\eps}  \big( 
     \hat \Phi_\eps(\hat f_\eps, \hat g_\eps) - 
    \hat \Phi_\eps(f_\eps, g_\eps)\big).
\end{align*} Taking expectations, $\E [\hat \Phi_\eps(f_\eps,
g_\eps)]= \Phi_\eps(f_\eps, g_\eps)$, and we conclude the result.
For the general case
where $0$ may not be in $\supp(\nu)$,
we can perform the above argument
with suitably offset $y$.
\end{proof}

\begin{proof}[Proof of Lemma~\ref{lem:map_mse_massage}]
As above, let us first assume that
$0 \in \supp(\nu)$
so that for each $y \in \supp(\nu)$,
we have
$\|y\|\leqslant R$.
We apply Young's inequality twice, first comparing $T_\eps(x)$ to
the empirical version $\frac{1}{n} \sum_{j = 1}^n y_j p_\eps(x, y_j)$,
and then the empirical version to the version involving $\tilde p_\eps$:
\begin{align*}
    \E[\|\hat T_\eps - T_\eps\|_{L^2(\hat \mu)}^2] &\lesssim \E\Big[ \Big\|\hat T_\eps -  \frac{1}{n} \sum_{j = 1}^ny_j 
\tilde p_\eps(x, y_j) \Big\|_{L^2(\hat \mu)}^2 \Big]  \\
&+ \E\Big[ \Big\|\frac{1}{n} \sum_{j = 1}^n y_j
(\tilde p_\eps(x, y_j) - p_\eps(x, y_j)) \Big\|_{L^2(\hat \mu)}^2 \Big] \\
&+ \E\Big[ \Big\|\frac{1}{n} \sum_{j = 1}^n y_j
 p_\eps(x, y_j)  - T_\eps(x) \Big\|_{L^2(\hat \mu)}^2 \Big]
\end{align*}
The third term is controlled as
\begin{align*}
    \E\Big[ \Big\|\frac{1}{n} &\sum_{j = 1}^n y_j p_\eps(x, y_j)  - T_\eps(x)
\Big\|_{L^2(\hat \mu)}^2 \Big] \\
 &=
 \E \Big[ \frac{1}{n^2} \sum_{j, k = 1}^n \langle y_j
 p_\eps(x, y_j)
 - T_\eps(x), y_k 
 p_\eps(x, y_k) - T_\eps(x) \rangle_{L^2(\hat \mu)}\Big] \\
 &= \frac{1}{n} \|y p_\eps(x, y) - T_\eps(x)
 \|_{L^2(\mu \otimes \nu)}^2 
 \lesssim \frac{R^2}{n} \|p_\eps\|_{L^2(\mu \otimes \nu)}^2,
\end{align*} where the second equality follows because, for $j \neq k$,
$y_j$ and $y_k$ are iid draws from $\nu$, so that these terms cancel. 
For the second term, we begin by observing that for each $y_j$,
$$
\tilde p_\eps(x, y_j)  = \frac{p_\eps(x, y_j) }{\frac{1}{n} \sum_{k = 1}^n 
p_\eps(x, y_k)}.
$$
We thus apply triangle inequality and this equation to see that
\begin{align*}
    \Big\| \frac{1}{n} \sum_{j = 1}^n y_j
(\tilde p_\eps(x, y_j) - p_\eps(x, y_j)) \Big\| &\leqslant
\frac{R}{n} \sum_{j = 1}^n |\tilde  p_\eps(x, y_j) - p_\eps(x, y_j) | \\
&= \frac{R}{n} \sum_{j = 1}^n p_\eps(x, y_j)\Big|\frac{1}{\frac{1}{n} \sum_{k =1 }^n p_\eps(x, y_k)} - 1 \Big| \\
&= R \Big| \frac1n \sum_{j = 1}^n p_\eps(x, y_j) - 1 \Big|.
\end{align*}
Taking the $\|\cdot\|_{L^2(\hat \mu)}^2$ norm, 
we recognize this
as the squared norm of the first component of $\nabla \hat \Phi_\eps(f_\eps, g_\eps)$,
and use the bound in Equation~\eqref{eqn:gradient_calc}
to conclude
\begin{align*}
\E\Big[ \Big\|\frac{1}{n} \sum_{j = 1}^n y_j
(\tilde p_\eps(x, y_j) - p_\eps(x, y_j)) \Big\|_{L^2(\hat \mu)}^2 \Big]
&\leqslant R^2 \E[\|\nabla \hat \Phi_\eps(f_\eps,
g_\eps)\|^2_{L^2(\hat \mu)
\times L^2(\hat \nu)}]\\
&\lesssim \frac{R^2}{n}\|p_\eps\|^2_{L^2(\mu \otimes \nu)}.
\end{align*} Applying
Lemma~\ref{lem:variance_control}
yields the result.
For the general case
where $0$ may not be in $\supp(\nu)$,
we can perform the above argument
with suitably offset $y$.
\end{proof}

\subsection{MID scaling for densities in empirical norm}
\label{subsec:empirical_density_proofs}

We prove Theorem~\ref{thm:density_estimation_slow}
on empirical norm
convergence of the entropic OT
density
using the techniques introduced
in the previous section.
We prove the following two
lemmas.

\begin{lemma}\label{lem:density_l1_KL}
Let $\tilde f_\eps, \tilde p_\eps$ be as in 
the previous section.
Then
$$
\E[\|\hat p_\eps - \tilde p_\eps\|_{L^1(\hat \mu \otimes 
\hat \nu)} ]\lesssim \frac{1}{\sqrt{\eps}}
\E[S_\eps(\hat \mu, \hat \nu)
- S_\eps(\mu, \nu)]^{1/2}
$$
\end{lemma}

\begin{lemma}\label{lem:density_l1_massage}
Let $\tilde f_\eps, \tilde p_\eps$
be as in the previous section.
Then
$$
\E[\|p_\eps - \hat p_\eps\|_{L^1(\hat\mu \otimes \hat\nu)}]
\lesssim \E[\|\hat p_\eps - \tilde p_\eps\|_{L^1(\hat \mu \otimes \hat\nu)}]
+  \sqrt{\frac{\mathcal{N}(\mu, \frac{\eps}{L}) \land \mathcal{N}(\nu, \frac{\eps}{L})}{n}}.
$$
\end{lemma}

We first dispatch Theorem~\ref{thm:density_estimation_slow},
and then prove these lemmas.

\begin{proof}[Proof of Theorem~\ref{thm:density_estimation_slow}]
Combining the above two lemmas
with Lemma~\ref{lem:bias_covering}
yields
$$
\E[\|p_\eps - \hat p_\eps\|_{L^1(\hat\mu \otimes \hat\nu)}]
\lesssim \frac{1}{\sqrt{\eps}}
\Big(\frac{\mathcal{N}(\mu, \frac{\eps}{L}) \land \mathcal{N}(\nu, \frac{\eps}{L})}{n}\Big)^{1/4}
+
\sqrt{\frac{\mathcal{N}(\mu, \frac{\eps}{L}) \land \mathcal{N}(\nu, \frac{\eps}{L})}{n}}.
$$ Now, observe that the statement
of Theorem~\ref{thm:density_estimation_slow}
is trivial when $\mathcal{N}(\mu, \frac{\eps}{L}) \land \mathcal{N}(\nu, \frac{\eps}{L}) \geqslant n$, since
we always have the bound $\lesssim 1$ by triangle inequality.
Hence it suffices to consider the case 
$\mathcal{N}(\mu, \frac{\eps}{L}) \land \mathcal{N}(\nu, \frac{\eps}{L}) \leqslant n$, in which case we may conclude from
the above inequality.
\end{proof}

\begin{proof}[Proof of Lemma~\ref{lem:density_l1_KL}]
Pinsker's and Jensen's imply
\begin{align*}
\|\tilde p_\eps - \hat p_\eps\|_{L^1(\hat \mu \otimes \hat \nu)}
&\lesssim \frac1n \sum_{i = 1}^n
\sqrt{\KL{\frac1n \hat p_\eps(x_i, \cdot)}{
\frac1n \tilde p_\eps(x_i, \cdot)}}\\
&\leqslant 
\Big(\frac1n \sum_{i = 1}^n \KL{\frac1n \hat p_\eps(x_i, \cdot)}{
\frac1n \tilde p_\eps(x_i, \cdot)} \Big)^{1/2}.
\end{align*} Taking the expectation
and applying Jensen's once more yields
$$
\E [\|\tilde p_\eps - \hat p_\eps\|_{L^1(\hat \mu \otimes \hat \nu)}]
\lesssim \E\Big[ \frac1n \sum_{i = 1}^n \KL{\frac1n \hat p_\eps(x_i, \cdot)}{
\frac1n \tilde p_\eps(x_i, \cdot)} \Big]^{1/2}.
$$ 
The statement then
follows as in Lemma~\ref{lem:map_mse_kl}.
\end{proof}

\begin{proof}[Proof of Lemma~\ref{lem:density_l1_massage}]
Apply triangle inequality:
$$
\E[\|p_\eps - \hat p_\eps\|_{L^1(\hat\mu \otimes \hat\nu)}]
\leqslant
\E[\|\hat p_\eps - \tilde p_\eps\|_{L^1(\hat \mu \otimes \hat\nu)}]
+ \E[\|\tilde p_\eps - p_\eps\|_{L^1(\hat \mu \otimes \hat \nu)}].
$$ For the second term we use
the same reasoning as in the proof of
Lemma~\ref{lem:map_mse_massage} to observe
that
$$
\|\tilde p_\eps - p_\eps\|_{L^1(\hat \mu \otimes 
\hat \nu)}
= \frac{1}{n^2}\sum_{i,j = 1}^n
|\tilde p_\eps(x_i, y_j) - p_\eps(x_i, y_j)|=
\frac{1}{n} \sum_{i =1}^n 
\Big|\frac{1}{n} \sum_{j = 1}^n
p_\eps(x_i, y_j) - 1\Big|.
$$
Applying Cauchy-Schwarz
we find that
$$
\|\tilde p_\eps - p_\eps\|_{L^1(\hat \mu \otimes 
\hat \nu)}  \leqslant \Big(\frac1n
\sum_{i = 1}^n \Big(\frac1n \sum_{j = 1}^n
p_\eps(x_i, y_j) - 1 \Big)^2 \Big)^{1/2}
\leqslant \|\nabla \hat \Phi_\eps(f_\eps, g_\eps)\|_{L^2(\hat \mu) \times L^2(\hat \nu)}.
$$
Taking expectations and applying
Cauchy-Schwarz once more, we can
apply Equation~\eqref{eqn:gradient_calc}
to yield
$$
\E[ \|\tilde p_\eps - p_\eps\|_{L^1(\hat \mu \otimes 
\hat \nu)}] \leqslant 
\E[\|\nabla \hat \Phi_\eps(f_\eps, g_\eps)\|_{L^2(\hat \mu)
    \times L^2(\hat \nu)}^2]^{1/2}
 \lesssim \frac{\|p_\eps\|_{L^2(\mu \otimes \nu)}}{\sqrt{n}}.
$$ Finally, Lemma~\ref{lem:variance_control}
yields the result.
\end{proof}

\section{Proofs of MID scaling with fast rates
on manifolds}
\label{sec:proofs_fast_rates}
In this section, we give the proofs
of our stronger results
when $\nu$ is supported on an embedded Riemannian
manifold.
To introduce our approach, recall that fast
rates of convergence for the entropic OT dual potentials,
map, density, and bias are known, but at the
price of exponential factors in $1/\eps$~\cite{rigollet2022sample}.
As we mention in the previous section, those estimates
are established through strong concavity of the empirical dual objective,
but because the strong concavity parameter is exponentially small, such an approach is inadequate for establishing MID scaling.
The techniques in the previous section represent a very weak form of this approach,
using only concavity rather than strong concavity to replace exponential factors with MID scaling.
In this section, we power our arguments
with a \emph{quadratic growth} condition~\cite{karimi2016linear},
which is a stronger
condition than concavity alone,
yet still weaker than true strong concavity.

In this regard, our approach
is inspired by
recent work~\cite{dimarino2021linear},
which
has found that a sufficient condition for
a quadratic growth condition for the
entropic OT dual is a Poincaré inequality.
Indeed,~\cite{dimarino2021linear}
shows that
given probability measures $P, Q$ on $\R^d$,
if $P$ satisfies a Poincaré inequality with constant
$C_P$, then for any dual potentials $f, g$,
\begin{equation}\label{eqn:simone_qg}
\|f - f_\eps^{PQ}\|^2_{L^2(P)}
\lesssim  \frac{L^2}{\eps C_P}( \Phi_\eps^{PQ}(f_\eps^{PQ}, g_\eps^{PQ})
- \Phi_\eps^{PQ}(f, g)),
\end{equation} where $f_\eps^{PQ}, g_\eps^{PQ}$ are the optimal dual
potentials from $P$ to $Q$ and we are suppressing numerical constants.
However, this estimate cannot directly establish
quadratic growth for the {\it empirical} dual 
$\hat \Phi_\eps$ since it relies on a Poincaré inequality
for the source measure $P$.
We emphasize that our approach is fundamentally
about
the empirical dual rather than
the population dual,
since working with the
empirical dual is what allows us to bypass empirical processes
by reducing to computing
variances of population quantities (consider, for example, our proof of Lemma~\ref{lem:bias_covering}).

As our first major step in this section,
we develop quadratic growth inequalities
for the empirical dual.
In particular, in section~\ref{subsec:qg} we
use different techniques
than~\cite{dimarino2021linear}
to show that
under an empirical analog of a Poincaré inequality --
a spectral gap of the random geometric
graph (RGG) of $\nu$ -- an analog
of~\eqref{eqn:simone_qg} holds for the empirical dual $\hat \Phi_\eps$.
Using recent advances in the theory
of RGGs on embedded manifolds~\cite{garcia2020error}
which we describe in Appendix~\ref{subsec:rgg_spectral},
a spectral gap for the RGG of $\nu$ then follows
from a Poincaré inequality for $\nu$ itself,
which is entailed by Assumption~\ref{assum:nu_manifold}
and Assumption~\ref{assum:nu}.
To summarize, the primary uses
of our embedded manifold assumptions are
that they imply
a spectral gap
for a weighted Laplacian
associated to $\nu$ (Proposition~\ref{prop:weighted_laplacian}), which then implies
a spectral gap for the RGG of $\nu$
using 
spectral
convergence theory~\cite{garcia2020error},
which finally yields quadratic growth
for the empirical dual; this 
argument is encapsulated in Lemma~\ref{lem:qg_emp_dual}.

After we prove Lemma~\ref{lem:qg_emp_dual} in
section~\ref{subsec:qg}, in section~\ref{subsec:reduction}
we next bound the $f$ dual potentials in terms of the $g$ dual potentials,
both in population and empirical norms, while
only picking up MID scaling factors. These results are combined
in section~\ref{subsec:proofs_MID_bias} to yield 
convergence of the $g$ dual potentials in empirical norm,
which then implies
Theorem~\ref{thm:manifold_bias} 
and all of Theorem~\ref{thm:dual_convergence} on dual potential
convergence
except the convergence of the $g$ dual potentials in population norm. In section~\ref{subsec:g_pop_norm},
we complete the proof of Theorem~\ref{thm:dual_convergence}
by applying the previous results to establish the convergence
of the $g$ dual potentials in population norm.
Finally, in section~\ref{subsec:fast_mid_maps} we
prove Corollary~\ref{cor:manifold_map}
and Corollary~\ref{cor:fast_density}.

Throughout, we work under the assumption
that $n$ is large enough that
\begin{equation}\label{eqn:n_large}
\big(\frac{\eps}{L}\big)^{2d_\nu} \gtrsim
\frac{1}{n}\Big(\log\Big(\frac{Ln}{\eps}\Big) + \frac{1}{\eps} \Big).
\end{equation}
We may assume $n$ is this large without loss of generality,
since otherwise it is not
hard to check that our bounds are worse than
the trivial ones.

\subsection{Quadratic growth for the empirical
dual}\label{subsec:qg}

In this section,
we prove the following quadratic growth
inequality for the empirical entropic OT
dual function.

\begin{lemma}[Quadratic growth
for the empirical dual]\label{lem:qg_emp_dual} Let
$\Psi_\eps^{\hat \mu \hat \nu} \colon L^{\infty}(\hat \nu) \to \R$
denote the empirical dual objective with $f$ rounded, namely
$$
\Psi_\eps^{\hat \mu \hat \nu}(g) := \hat \mu \big(-\eps \log 
\big( \frac1n \sum_{j = 1}^n e^{-\eps^{-1}(c(x, y_j) - g(y_j))}\big) \big)
+ \hat \nu(g).
$$
Then, for $\eps$ sufficiently
small, we have
$$
\eps\Big( \frac{\eps}{L} \Big)^{d_\nu + 2} \cdot 
\|g - \hat g_\eps\|_{L^2(\hat \nu)}^2 \lesssim 
\|g - \hat g_\eps\|_{L^{\infty}(\hat \nu)}^2
(\Psi_\eps^{\hat \mu \hat \nu}(\hat g_\eps) - \Psi_\eps^{\hat \mu \hat \nu}(g)),
\quad \quad \forall g\in L^{\infty}(\hat \nu), \, \,
\hat \nu(g) = 0.
$$
\end{lemma}

This lemma is proven by reducing the claim
to a spectral gap for the random geometric
graph (RGG) of $\nu$. Using recent fundamental
work on the spectral convergence of RGGs of embedded
manifolds~\cite{garcia2020error},
we conclude the result.
We include this background material
in Appendix~\ref{subsec:rgg_spectral}.

\begin{proof}[Proof of Lemma~\ref{lem:qg_emp_dual}]
Fix $g \in L^{\infty}(\hat \nu)$
with $\hat \nu(g) = 0$,
and put $\alpha := g - \hat g_\eps$.
Let
$$
\psi(t) := -\Psi_\eps^{\hat \mu \hat \nu}(\hat g_\eps + t \alpha),
$$ and also define
$$
q_t(x_i, y_j) := \frac{e^{-\eps^{-1}(c(x_i, y_j) - (\hat g_\eps(y_j)
+ t\alpha(y_j))}}{\frac 1n \sum_{k =1}^n  e^{-\eps^{-1}(c(x_i, y_k) -
(\hat g_\eps(y_k) + t\alpha(y_k))}}
$$
Then it is straightforward to verify that
$\psi$ is a convex function with minimum attained at
$ t= 0$ such that
$$
\psi''(t) = \frac{1}{\eps} \hat \mu \big(\Var_{\frac{1}{n} q_t(x, \cdot)}
(\alpha)\big).
$$ Note that,
$$
q_t(x_i, y_j) \geqslant e^{-2t\eps^{-1}
\|\alpha\|_{L^{\infty}(\hat \nu)}}
q_0(x_i, y_j) = e^{-2t\eps^{-1}
\|\alpha\|_{L^{\infty}(\hat \nu)}}
\hat p_\eps(x_i, y_j).
$$ Therefore,
\begin{align*}
\psi''(t) &= \frac{1}{2\eps} 
\hat \mu\big( \E_{y, y' \sim \frac{1}{n} q_t(x, \cdot)}[(\alpha(y)
- \alpha(y'))^2] \big)  \\
&\geqslant  \frac{e^{-4t \eps^{-1}\|\alpha\|_{L^{\infty}(\hat \nu)}}}{2 \eps} \hat \mu \big(
\E_{y, y' \sim \frac{1}{n}\hat p_\eps(x, \cdot)}[(\alpha(y)
- \alpha (y'))^2]\big),
\end{align*}
It follows that, for $t > 0$,
\begin{align*}
   t^2 \frac{e^{-4t \eps^{-1}\|\alpha\|_{L^{\infty}(\hat \nu)}}}{
   4 \eps} \hat \mu \big(
\E_{y, y' \sim \frac1n \hat p_\eps(x, \cdot)}[(\alpha(y)
- \alpha (y'))^2]\big)
\leqslant \psi(t) - \psi(0) \leqslant \psi(1) - \psi(0),
\end{align*}
where the first inequality is strong
convexity
for $\psi$ on the interval $(-\infty, t]$,
and the second inequality follows because $\psi$
is a convex function with minimizer at $t = 0$.
If $\|\alpha\|_{L^{\infty}(\hat \nu)} = 0$,
then the inequality is trivially true.
Otherwise, we may set $t = \eps/4\|\alpha\|_{L^{\infty}(\hat \nu)}$
to yield
$$
\frac{\eps}{\|\alpha\|^2_{L^{\infty}(\hat \nu)}}
\E_{y, y' \sim \frac1n \hat p_\eps(x, \cdot)}[(\alpha(y)
- \alpha(y'))^2]\lesssim \psi(1) - \psi(0).
$$ It thus suffices to show that
\begin{equation}\label{eqn:qg_goal}
\hat \mu\big(\E_{y, y' \sim \frac1n \hat p_\eps(x, \cdot)}[(\alpha(y)
- \alpha(y'))^2]\big)
\gtrsim \Big( \frac{\eps}{L} \Big)^{d_\nu + 2} \hat \nu(\alpha^2),
\end{equation}
Observe that
$$
\hat \mu\big(\E_{y, y' \sim \frac1n \hat p_\eps(x, \cdot)}[(\alpha(y)
- \alpha(y'))^2]\big)
= \frac{1}{n^2}\sum_{j, k = 1}^n
\Big( \frac1n \sum_{i = 1}^n
\hat p_\eps(x_i, y_j)
\hat p_\eps(x_i, y_k)\Big)
(\alpha(y_j) - \alpha(y_k))^2.
$$ By Proposition~\ref{prop:consequences_of_lipschitz_cost},
$$
\hat p_\eps(x_i, y_j) \geqslant e^{-2\frac{L}{\eps}\|y_j - y_k\|}
\hat p_\eps(x_i, y_k), \quad \quad  \forall j, k \in [n].
$$ Hence
\begin{align*}
    \frac{1}{n^2}\sum_{j, k = 1}^n
\Big( \frac1n \sum_{i = 1}^n &
\hat p_\eps(x_i, y_j)
\hat p_\eps(x_i, y_k)\Big)
(\alpha(y_j) - \alpha(y_k))^2
\geqslant \\
&\frac{1}{n^2} \sum_{j, k = 1}^n
e^{-2 \frac{L}{\eps}
\|y_j - y_k\|} \Big( \frac{1}{n}
\sum_{i = 1}^n \hat p_\eps(x_i, y_k)^2\Big)
(\alpha(y_j) - \alpha(y_k))^2 \\
&\geqslant 
\frac{1}{n^2} \sum_{j, k = 1}^n
 e^{-2 \frac{L}{\eps} \|y_j -y_k\|}
 \Big( \frac{1}{n} \sum_{i = 1}^n
 \hat p_\eps(x_i, y_k) \Big)^2 (\alpha(y_j) - \alpha(y_k))^2 \\
 &= \frac{1}{n^2} \sum_{j, k = 1}^n
 e^{-2 \frac{L}{\eps}\|y_j - y_k\|}
 (\alpha(y_j) - \alpha(y_k))^2.
\end{align*}
Now, form the RGG with threshold
$\delta = \frac{\eps}{L}$, so that
$y_j \sim y_k$ if and only if
$\|y_j - y_k\| \leqslant \frac{\eps}{L}$.
Then
$$
\hat \mu\big(\E_{y, y' \sim \frac{1}{n}\hat p_\eps(x, \cdot)}[(\alpha(y)
- \alpha(y'))^2]\big)
\gtrsim \frac{1}{n^2} \sum_{j \sim k} 
(\alpha(y_j) - \alpha(y_k))^2.
$$ 
Applying Theorem~\ref{thm:graph_spectrum}
from Appendix~\ref{subsec:rgg_spectral},
we arrive at~\eqref{eqn:qg_goal},
and so the result follows.
\end{proof}

\subsection{Reduction from $f$ dual potentials to $g$ dual potentials}
\label{subsec:reduction}

In this section, we will prove the following
Lemma, which shows that although the $f$ dual potentials
are defined on the support of $\mu$ -- a potentially
higher dimensional set than the support of $\nu$ --
their convergence can be controlled by the convergence of the $g$ dual potentials
while only incurring MID scaling factors.

\begin{lemma}[Reduction from $f$ dual potentials
to $g$ dual potentials]\label{lem:f_control_by_g}
We have
$$
\E[\|\hat f_\eps - f_\eps\|_{L^2(\hat \mu)}^2]
\lesssim \Big(\frac{L}{\eps} \Big)^{3d_\nu} \big(\E[\|\hat g_\eps - g_\eps\|_{L^2(\hat \nu)}^2]
+ \frac{1 + \eps^2}{n}\big),
$$ and similarly
$$
\E[\|\hat f_\eps - f_\eps\|_{L^2(\mu)}^2]
\lesssim \Big(\frac{L}{\eps} \Big)^{3d_\nu} \big(\E[\|\hat g_\eps - g_\eps\|_{L^2(\hat \nu)}^2]
+ \frac{1 + \eps^2}{n}\big).
$$
\end{lemma}

To prove Lemma~\ref{lem:f_control_by_g},
as well as for the proofs in the following sections,
the following helper lemmas
are convenient (proven in Appendix~\ref{subsec:deferred_proofs_fast_rates}).

\begin{lemma}[Mass of empirical balls]\label{lem:empirical_mass_balls}
With probability at least $1 - \frac{1}{n}e^{-10/\eps}$,
$$
\inf_{z \in N} \hat \nu(B(z, \frac{\eps}{L})) \gtrsim 
\Big(\frac{\eps}{L}  \Big)^{d_\nu}.
$$
\end{lemma}

\begin{lemma}[Uniform control on entropic densities]\label{lem:manifold_density_sup_control}
We have the uniform bound,
$$
\|p_\eps\|_{L^{\infty}(\mu \otimes \nu)}
\lesssim \Big( \frac{L}{\eps} \Big)^{d_\nu}.
$$
And, with probability at least $1 - \frac{1}{n}e^{-10/\eps}$,
$$
\|\hat p_\eps\|_{L^{\infty}(\mu \otimes \nu)}
\lesssim \Big( \frac{L}{\eps} \Big)^{d_\nu}.
$$
\end{lemma}

\begin{proof}[Proof of Lemma~\ref{lem:f_control_by_g}]
Let
$$
\tilde f_\eps (x) := -\eps \log \Big(\frac1n \sum_{j = 1}^n
e^{-\frac{1}{\eps}(c(x, y_j) - g_\eps(y_j))} \Big).
$$ Note that, for all $x \in \supp \mu$,
$$
(\hat f_\eps(x) - f_\eps(x))^2
\lesssim
(\hat f_\eps(x) - \tilde f_\eps(x))^2
+ (\tilde f_\eps(x) - f_\eps(x))^2,
$$ and we'll begin by studying each of these terms separately.

For the first term, 
put $\alpha := g_\eps - \hat g_\eps$,
and let $\phi \colon [0, 1] \to \R$
be defined as
$$
\phi(t) := -\eps \log \Big( \frac{1}{n} \sum_{j = 1}^n
e^{-\eps^{-1}(c(x, y_j) - \hat g_\eps(y_j)  - t\alpha(y_j)}\Big),
$$ and notice $\phi(0) = \hat f_\eps(x)$ 
while $\phi(1) = \tilde f_\eps(x)$.
For $y \in \supp(\nu)$, put
$$
q_t(x, y) := \frac{e^{-\eps^{-1}(c(x, y) - \hat g_\eps(y) - t \alpha(y))}}
{\frac1n \sum_{k = 1}^n e^{-\eps^{-1}(c(x, y_k) - \hat g_\eps(y_k) - t \alpha(y_k))}}
$$
Then
\begin{align*}
|\phi'(t)| &=\Big|\frac1n \sum_{j = 1}^n 
\alpha(y_j) q_t(x, y_j)
\Big| \leqslant 
\|\alpha\|_{L^2(\hat \nu)} \Big(
\frac1n 
\sum_{j = 1}^n q_t(x, y_j)^2 \Big)^{1/2}.
\end{align*}
Using Proposition~\ref{prop:consequences_of_lipschitz_cost},
we find that for any $y, y' \in \supp(\nu)$
and $t \in [0,1]$,
\begin{align*}
|\log q_t(x, y) - \log q_t(x, y')| &= 
\frac{1}{\eps} |c(x, y) - \hat g_\eps(y) 
- t\alpha(y)
- (c(x, y') - \hat g_\eps(y') 
- t\alpha(y'))| \\
&\leqslant \frac{2L}{\eps}\|y - y'\|.
\end{align*}  Therefore, for any $j \in [n]$,
$$
1 = \frac1n \sum_{k = 1}^n q_t(x, y_k) 
\gtrsim q_t(x, y_j) 
\hat \nu(B(y_j, \frac{\eps}{L})) \geqslant q_t(x, y_j)
\inf_{z \in N} \hat \nu (B(z, \frac{\eps}{L})).
$$ So that
\begin{align}
    |\hat f_\eps(x) - \tilde f_\eps(x)|
    \leqslant \int_0^1 |\phi'(t)| \ud t
    &\leqslant 
    \|\alpha\|_{L^2(\hat \nu)}
    \int_0^1
    \Big(
\frac1n  \sum_{j = 1}^n q_t(x, y_j)^2 \Big)^{1/2} \ud t\nonumber \\
&\lesssim 
    \|\alpha\|_{L^2(\hat \nu)} \sup_{z \in N}
    \hat \nu(B(z, \frac{\eps}{L}))^{-1/2} \int_0^1
    \Big(\frac1n \sum_{j = 1}^n q_t(x, y_j) \Big)^{1/2} 
    \ud t \nonumber \\
    &= \|\alpha\|_{L^2(\hat \nu)}  \sup_{z \in N}
    \hat \nu(B(z, \frac{\eps}{L}))^{-1/2}. \label{eqn:marginal_rounded_bd_lemma}
\end{align}

Now consider the 
$(\tilde f_\eps(x) - f_\eps(x))^2$ term.
Note that
$$
|\tilde f_\eps(x) - f_\eps(x)| = \Big| 
\eps \log \Big( \frac1n \sum_{j = 1}^n p_\eps(x, y_j) \Big)
\Big|.
$$ Since
$$
1 = \int p_\eps(x, y) \ud \nu(y),
$$ there must exist some $y'\in \supp(\nu)$ such that $p_\eps(x, y')
\geqslant 1$. By Proposition~\ref{prop:consequences_of_lipschitz_cost},
this implies
\begin{align*}
\frac1n \sum_{j = 1}^n p_\eps(x, y_j) \geqslant 
\frac1n \sum_{j = 1}^n \frac{p_\eps(x, y_j)}{p_\eps(x, y')}
\geqslant \frac1n \sum_{j = 1}^n e^{-\frac{2L}{\eps}\|y_j - y'\|}
\gtrsim \hat \nu(B(y', \frac{\eps}{L})).
\end{align*}
Using Lipschitz-ness of $\log$,
\begin{equation}\label{eqn:marginal_rounded_bd2_lemma}
|\tilde f_\eps(x) - f_\eps(x)| 
\lesssim \eps \hat \nu( B(y', \frac{\eps}{L} ))^{-1}
\Big| \frac1n \sum_{j = 1}^n p_\eps(x, y_j)  - 1\Big|.
\end{equation} Now, let $\mathcal{E}$ denote the
event of Lemma~\ref{lem:empirical_mass_balls}.
Observe that by the pointwise control in
Proposition~\ref{prop:dual_pointwise_control}
and since $\P[\mathcal{E}^c] \leqslant \frac{1}{n}$,
$$
\E[\|\hat f_\eps - f_\eps\|^2_{L^2(\hat \mu)}]
\lesssim 
\E[\mathbbold{1}[\mathcal{E}]
\|\hat f_\eps - f_\eps\|^2_{L^2(\hat \mu)}
]
+ \frac{1}{n}.
$$ Combining Equations~\eqref{eqn:marginal_rounded_bd_lemma}
and~\eqref{eqn:marginal_rounded_bd2_lemma}
yields
\begin{align*}
\E[\|\hat f_\eps - f_\eps\|^2_{L^2(\hat \mu)}]
&\lesssim 
\E[\mathbbold{1}[\mathcal{E}]
\|\hat f_\eps - f_\eps\|^2_{L^2(\hat \mu)}
] + \frac{1}{n} \\
&\lesssim 
\E[\mathbbold{1}[\mathcal{E}]
(\|\hat f_\eps - \tilde f_\eps\|^2_{L^2(\hat \mu)}
+\|\tilde f_\eps - f_\eps \|^2_{L^2(\hat \mu)})
] + \frac{1}{n} \\
&\lesssim \Big( \frac{L}{\eps} \Big)^{2d_{\nu}}
\E\Big[\mathbbold{1}[\mathcal{E}]\Big\{
\|\hat g_\eps - g_\eps\|_{L^2(\hat \nu)}^2
+ \frac{\eps^2}{n}\sum_{i = 1}^n\Big( \frac1n \sum_{j = 1}^n
p_\eps(x_i, y_j) - 1 \Big)^2 \Big\}\Big]
+ \frac1n \\
&\leqslant \Big( \frac{L}{\eps} \Big)^{2d_{\nu}}
\E\Big[\mathbbold{1}[\mathcal{E}]\Big\{
\|\hat g_\eps - g_\eps\|_{L^2(\hat \nu)}^2
+ \frac{\eps^2}{n} \|\nabla \hat \Phi_\eps(f_\eps, g_\eps)\|^2_{L^2(\hat \mu) \times L^2(\hat \nu)} \Big\}\Big]
+ \frac1n \\
&\lesssim 
\Big( \frac{L}{\eps} \Big)^{2d_{\nu}}
\Big\{\E[\|\hat g_\eps - g_\eps\|^2_{L^2(\hat \nu)}]
+ \frac{\eps^2}{n} \|p_\eps\|_{L^2(\mu \otimes \nu)}^2
\Big\}
+ \frac1n,
\end{align*}
where the final inequality follows from
$\mathbbold{1}[\mathcal{E}] \leqslant 1$
and Equation~\eqref{eqn:gradient_calc}.
By 
Lemma~\ref{lem:variance_control}
and Proposition~\ref{prop:N_covering_numbers} on the covering numbers of $N$ (proved in
Appendix~\ref{subsec:additional_embedded}), we conclude the first
inequality. The second follows in the same manner.
\end{proof}

\subsection{MID scaling for the bias and $g$ dual potentials
in empirical norm}\label{subsec:proofs_MID_bias}
In this section, we use
our quadratic growth inequality
from section~\ref{subsec:qg} and 
our bound on the $f$ dual potentials in terms
of the $g$ dual potentials from section~\ref{subsec:reduction} to
prove the following result, on convergence of the $g$ dual potentials
in empirical norm.

\begin{lemma}[Convergence of $g$ dual potentials
in empirical norm]\label{lem:emp_g_convergence}
We have
$$
\E[\|\hat g_\eps - g_\eps\|^2_{L^2(\hat\nu)}]
\lesssim \big(\eps^2 + \frac{1}{\eps^2}\big) \cdot \Big(\frac{L}{\eps}\Big)^{6d_\nu + 4} \cdot \frac1n.
$$
\end{lemma}

By the previous section, this
implies convergence of the $f$ dual potentials
in both empirical and population norms, and therefore yields all parts
of Theorem~\ref{thm:dual_convergence} other than the population
norm convergence of the $g$ dual potentials, which is proved
in the next section.
As a byproduct of our proof, we will also obtain Theorem~\ref{thm:manifold_bias},
on MID scaling with fast rates for the bias.

Notice that a $1/\sqrt{n}$ rate
follows by combining
our result on quadratic growth
for the empirical dual, Lemma~\ref{lem:qg_emp_dual},
with our prior result on the bias, Lemma~\ref{lem:bias_covering}.
To obtain a faster rate, we proceed by a self-bounding argument.

We first bound
the bias $\E[\hat \Phi_\eps(\hat f_\eps, \hat g_\eps) - \hat \Phi_\eps(f_\eps, g_\eps)]$ in terms of 
    $\|\hat f_\eps - f_\eps\|^2_{L^2(\hat \mu)}$
    and $\|\hat g_\eps - g_\eps\|^2_{L^2(\hat \nu)}$
    plus a $1/n$ term arising from the squared
    norm of $\nabla \hat \Phi_\eps (f_\eps, g_\eps)$.

    \begin{lemma}\label{lem:manifold_bias}
For all $a > 0$,
$$
\E[\hat \Phi_\eps(\hat \mu, \hat \nu) - \hat \Phi_\eps(\mu, \nu)]
\lesssim \frac{1}{an}  \Big(\frac{L}{\eps} \Big)^{d_\nu}
+ a \E\big[\|\hat f_\eps - f_\eps \|^2_{L^2(\hat \mu)}+
\|\hat g_\eps - g_\eps\|^2_{L^2(\hat \nu)}\big].
$$
\end{lemma}

Using our result on the quadratic
    growth of $\hat \Phi_\eps$, Lemma~\ref{lem:qg_emp_dual},
    the previous step then implies
    a bound for $\|\hat g_\eps - g_\eps\|^2_{L^2(\hat \nu)}$
    in terms of itself, $1/n$, and $\|\hat f_\eps - f_\eps\|^2_{L^2(\hat \mu)}$.
    
\begin{lemma}\label{lem:manifold_g_conv}
For all $a > 0$,
$$
\E[\|\hat g_\eps - g_\eps\|^2_{L^2(\hat \nu)}]
\lesssim \frac{1}{\eps}
\Big( \frac{L}{\eps} \Big)^{d_\nu + 2}
\Big\{ \frac{1}{an}  \Big(\frac{L}{\eps} \Big)^{d_\nu}
+ a \E\big[\|\hat f_\eps - f_\eps \|^2_{L^2(\hat \mu)}+
\|\hat g_\eps - g_\eps\|^2_{L^2(\hat \nu)}\big]\Big\} +\frac1n.
$$
\end{lemma}

Combining this inequality with
Lemma~\ref{lem:f_control_by_g} from section~\ref{subsec:reduction}
and taking the free parameter
$a$ sufficiently small, we can re-arrange
and arrive at Lemma~\ref{lem:emp_g_convergence}.
Plugging our bounds
on the dual potentials
into
 Lemma~\ref{lem:manifold_bias}
 then yields
Theorem~\ref{thm:manifold_bias}.
Let's give these proofs first, before
we turn to the proofs
of Lemma~\ref{lem:manifold_bias}
and Lemma~\ref{lem:manifold_g_conv}.
\begin{proof}[Proof of Lemma~\ref{lem:emp_g_convergence}]
By Lemma~\ref{lem:manifold_g_conv}
and Lemma~\ref{lem:f_control_by_g}, for all $ a > 0$,
\begin{align*}
    \E[\|\hat g_\eps - g_\eps\|^2_{L^2(\hat\nu)}]
    \lesssim \frac{1}{\eps}
\Big( \frac{L}{\eps} \Big)^{d_\nu + 2}
\Big\{ \frac{1}{an}  \Big(\frac{L}{\eps} \Big)^{d_\nu}
+ a   \Big(\frac{L}{\eps} \Big)^{3d_\nu} \big(\E[\|\hat g_\eps - g_\eps\|_{L^2(\hat \nu)}^2]
+ \frac{1 + \eps^2}{n}\big) \Big\}+ \frac1n.
\end{align*}
If we take $a = C\eps\big(\frac{\eps}{L}\big)^{4d_{\nu} + 2}$
for a sufficiently small constant $C$, we can re-arrange
to yield
$$
 \E[\|\hat g_\eps - g_\eps\|^2_{L^2(\hat\nu)}]
\lesssim 
\big(1 + \eps^2 + \frac{1}{\eps^2}\big) \cdot
\Big( \frac{L}{\eps} \Big)^{6d_\nu + 4}
\cdot
\frac1n \lesssim 
\big(\eps^2 + \frac{1}{\eps^2}\big)
\cdot
\Big( \frac{L}{\eps} \Big)^{6d_\nu + 4}
\cdot
\frac1n.
$$
\end{proof}

\begin{proof}[Proof of Theorem~\ref{thm:manifold_bias}]
First notice that $\E[S_\eps(\hat \mu, \hat \nu)] - S_\eps(\mu, \nu) = \E[\hat \Phi_\eps(\hat f_\eps,
\hat g_\eps) - \hat \Phi_\eps(f_\eps, g_\eps)] \geqslant 0$, so that
$$
|\E[S_\eps(\hat \mu, \hat \nu)] - S_\eps(\mu, \nu)| = \E[\hat \Phi_\eps(\hat f_\eps,
\hat g_\eps) - \hat \Phi_\eps(f_\eps, g_\eps)].
$$
Invoking Lemma~\ref{lem:manifold_bias}
with $a = 1$,
and applying
Lemma~\ref{lem:f_control_by_g}
plus Lemma~\ref{lem:emp_g_convergence},
we obtain
\begin{align*}
\E[S_\eps(\hat \mu, \hat \nu)
- S_\eps(\mu, \nu)]
\lesssim
\big(1 + \eps^2 + \frac{1}{\eps^2}\big) \cdot
\Big( \frac{L}{\eps} \Big)^{9d_\nu + 4}
\cdot
\frac1n \leqslant
\big(\eps^2 + \frac{1}{\eps^2}\big)
\cdot
\Big( \frac{L}{\eps} \Big)^{9d_\nu + 4}
\cdot
\frac1n.
\end{align*}
and whence the result.
\end{proof}

The proofs of Lemma~\ref{lem:manifold_bias}
and Lemma~\ref{lem:manifold_g_conv}
follow.

\begin{proof}[Proof of Lemma~\ref{lem:manifold_bias}]
By Proposition~\ref{prop:entropic_dual_basics},
for all $a > 0$,
\begin{align*}
\E[S_\eps(\hat \mu, &\hat \nu)-
S_\eps(\mu, \nu)] = \E[\hat \Phi_\eps(\hat f_\eps,
\hat g_\eps) - \Phi_\eps(f_\eps, g_\eps)] \\
&= \E[ \hat \Phi_\eps(\hat f_\eps, \hat g_\eps)
- \hat \Phi_\eps(f_\eps, g_\eps) ]\\
&\leqslant
\E[ \langle \nabla \hat \Phi_\eps(f_\eps, g_\eps), (\hat f_\eps
- f_\eps, \hat g_\eps - g_\eps) \rangle_{L^2(\hat \mu) \times
L^2(\hat \nu)}] \\
&\leqslant \frac{1}{a}
\E[ \| \nabla \hat \Phi_\eps(f_\eps, g_\eps)\|_{L^2(\hat \mu) \times L^2(\hat \nu)}^2]
+ a \E\big[\|\hat f_\eps - f_\eps \|^2_{L^2(\hat \mu)}+
\|\hat g_\eps - g_\eps\|^2_{L^2(\hat \nu)}\big].
\end{align*}
By Equation~\eqref{eqn:gradient_calc}
and Lemma~\ref{lem:variance_control},
we have
$$
\E[S_\eps(\hat \mu, \hat \nu)-
S_\eps(\mu, \nu)]
\lesssim \frac{1}{an} \mathcal{N}(N, \frac{\eps}{L})
+ a \E\big[\|\hat f_\eps - f_\eps \|^2_{L^2(\hat \mu)}+
\|\hat g_\eps - g_\eps\|^2_{L^2(\hat \nu)}\big].
$$ Applying Proposition~\ref{prop:N_covering_numbers}
on the covering numbers of $N$
(proved in Appendix~\ref{subsec:additional_embedded}),
we may conclude.
\end{proof}

\begin{proof}[Proof of Lemma~\ref{lem:manifold_g_conv}]
We begin by re-centering
$g_\eps$ so that we can apply Lemma~\ref{lem:qg_emp_dual}:
\begin{align}
\E[\|\hat g_\eps - g_\eps \|_{L^2(\hat \nu)}^2]
&\leqslant 
\E[\|\hat g_\eps - (g_\eps - \hat \nu(g_\eps)) \|_{L^2(\hat \nu)}^2]
+ \E[(\hat \nu(g_\eps))^2]\nonumber \\
&\lesssim  
\E[\|\hat g_\eps - (g_\eps - \hat \nu(g_\eps)) \|_{L^2(\hat \nu)}^2] +
\frac{1}{n}, \label{eqn:fixed_point_inequality}
\end{align} where the last step
follows by our convention that $\nu(g_\eps) = 0$
and the pointwise control on $g_\eps$ from Proposition~\ref{prop:dual_pointwise_control}.
Let $\mathcal{E}$ denote the event
described in Lemma~\ref{lem:qg_emp_dual}.
Applying Proposition~\ref{prop:dual_pointwise_control}
once more we see that
\begin{align*}
\E[\|\hat g_\eps - (g_\eps - \hat \nu(g_\eps)) \|_{L^2(\hat \nu)}^2]
&=
\E[\mathbbold{1}[\mathcal{E}] \|\hat g_\eps - (g_\eps - \hat \nu(g_\eps)) \|_{L^2(\hat \nu)}^2]
+ \E[\mathbbold{1}[\mathcal{E}^c]] \\
&\lesssim 
\E[\mathbbold{1}[\mathcal{E}] \|\hat g_\eps - (g_\eps - \hat \nu(g_\eps)) \|_{L^2(\hat \nu)}^2]
+ \frac1n \\
&\lesssim \E\Big[\mathbbold{1}[\mathcal{E}]
\| \hat g_\eps - (g_\eps - \hat \nu(g_\eps))\|^2_{L^{\infty}(\hat\nu)}
\\
&\times \frac{1}{\eps} \Big( \frac{L}{\eps} \Big)^{d_\nu + 2}
\big( \Psi_\eps^{\hat \mu \hat \nu}(\hat g_\eps)
- \Psi_\eps^{\hat \mu \hat \nu}(g_\eps) \big)\Big] + \frac1n.
\end{align*}
Using $\mathbbold{1}[\mathcal{E}]\leqslant 1$
and Proposition~\ref{prop:dual_pointwise_control}
once more, we find
$$
\E[\|\hat g_\eps - (g_\eps - \hat \nu(g_\eps))\|_{L^2(\hat \nu)}^2]
\lesssim \frac{1}{\eps} \Big( \frac{L}{\eps} \Big)^{d_\nu + 2}
\E\big[
 \Psi_\eps^{\hat \mu \hat \nu}(\hat g_\eps)
- \Psi_\eps^{\hat \mu \hat \nu}(g_\eps)\big].
$$
By
Proposition~\ref{prop:entropic_dual_basics}
on marginal rounding
$$
\Psi_\eps^{\hat \mu \hat \nu}( g_\eps) \geqslant \hat \Phi_\eps( f_\eps, g_\eps).
$$ Hence
$$
\E[\|\hat g_\eps - (g_\eps - \hat \nu(g_\eps))\|_{L^2(\hat \nu)}^2]
\lesssim \frac{1}{\eps} \Big( \frac{L}{\eps} \Big)^{d_\nu + 2}
\E\big[
 \hat \Phi(\hat f_\eps, \hat g_\eps)
 - \hat \Phi_\eps(f_\eps, g_\eps)\big].
$$ Since $\E[\hat \Phi_\eps(f_\eps, g_\eps)]
= \Phi_\eps(f_\eps, g_\eps)$, the result
follows from Lemma~\ref{lem:manifold_bias}.
\end{proof}

\subsection{Fast rates for $g$-dual potentials in population norms}
\label{subsec:g_pop_norm}

In this section, we establish the remaining part
of Theorem~\ref{thm:dual_convergence}
by proving the following Lemma.

\begin{lemma}[Fast population norm convergence of $g$-dual potentials]\label{lem:pop_g_convergence}
We have
$$
\E[\|\hat g_\eps - g_\eps\|_{L^2(\nu)}^2]
\lesssim  \big( \eps^2 + \frac{1}{\eps^2}\big)
 \cdot \Big(\frac{L}{\eps}  \Big)^{13d_\nu + 8}
\cdot \frac{1}{n}.
$$
\end{lemma}

Lemma~\ref{lem:pop_g_convergence}
is proved
with another self-bounding
argument that we outline here, ignoring $\eps/L$ dependence for clarity. Note that
we will frequently use the canonical extension
of $\hat g_\eps$ to consider $(\hat f_\eps, \hat g_\eps)$
as a dual variable pair
for the semi-empirical problem between $\hat \mu$ and $\nu$.
In particular, notation such as $\nu(\hat g_\eps)$
is not referring to integration over any part of
the sample $y_1, \ldots, y_n \sim \nu^{\otimes n}$,
but rather over an \emph{independent} draw from $\nu$.
\begin{enumerate}
    \item[1.] Lemma~\ref{lem:pop_g_by_g_maps}: Use a Poincaré inequality
    on $\nu$ (which follows from our assumptions and is described in
    Proposition~\ref{prop:weighted_poincare} in Appendix~\ref{subsec:additional_embedded})
    to control $\|\hat g_\eps - g_\eps - \nu(\hat g_\eps) \|_{L^2(\nu)}^2$ in terms of entropic OT
    maps defined on $N$
    \item[2.] Lemma~\ref{lem:g_maps_by_semi_empirical}:
    Use the same analysis
    as in the proof of the empirical norm convergence
    of the entropic OT maps (section~\ref{subsec:empirical_map_proofs})
    to control the previous term
    in terms of $1/n$, the semi-empirical dual
    objective $\Phi^{\hat \mu\nu}_\eps$, 
    and $\|\hat g_\eps - g_\eps\|_{L^2(\nu)}^2$, with
    a free parameter from Young's inequality.
    \item[3.] Lemma~\ref{lem:semi_empirical_bound}: Control the semi-empirical objective
    term
    with concavity to yield a bound in terms of $1/n$
    and $\|\hat g_\eps - g_\eps\|_{L^2(\nu)}^2$, with
    a free parameter from Young's inequality.
    \item[4.] Lemma~\ref{lem:nu_hat_g}: Bound the shift $\E[ \nu(\hat g_\eps)^2]$ by $1/n$ using the semi-discrete
    objective once more
    \item[5.] Use the free parameters from Young's
    inequality to conclude
    a $1/n$ bound on $\|\hat g_\eps - g_\eps\|^2_{L^2(\nu)}$.
\end{enumerate}
The formal statements of the lemmas outlined
above are given below.

\begin{lemma}\label{lem:pop_g_by_g_maps} We have
$$
\E[\|\hat g_\eps - g_\eps - \nu(\hat g_\eps)\|_{L^2(\nu)}^2]
\lesssim 
\E\Big[\Big\|\E_{\pi_\eps}[\nabla_y c(x, y) \, | \, y]
-\frac1n \sum_{i = 1}^n  \nabla_y c(x_i, y) \hat p_\eps(x_i, y)\Big\|_{L^2(\nu)}^2
\Big].
$$
\end{lemma}

\begin{lemma}\label{lem:g_maps_by_semi_empirical}
For all $a > 0$,
we have
\begin{align*}
\E\Big[\Big\|\E_{\pi_\eps}[\nabla_y c(x, y) \, | \, y]
-\frac1n \sum_{i = 1}^n  \nabla_y c(x_i, y)& \hat p_\eps(x_i, y)\Big\|_{L^2(\nu)}^2
\Big]
\lesssim 
\frac{L^2}{\eps} \E[\Phi^{\hat \mu \nu}_\eps( \hat f_\eps, \hat g_\eps)
- \Phi^{\hat \mu \nu}_\eps( f_\eps, g_\eps)] \\
&+ a \frac{1}{\eps} \Big(\frac{L}{\eps} \Big)^{2 d_\nu + 2} 
 \E[ \|\hat g_\eps - g_\eps\|^2_{L^2(\nu)}]\\
 &+ L\big(\frac{a}{\eps^2} + \frac1a\big) \big( 
 \eps^2 + \frac{1}{\eps^2} \big)\Big(\frac{L}{\eps} \Big)^{11 d_\nu + 5} 
\cdot \frac1n.
\end{align*}
\end{lemma}

\begin{lemma}\label{lem:semi_empirical_bound}
For all $a > 0$, we have
$$
\E[\Phi^{\hat \mu \nu}_\eps( \hat f_\eps, \hat g_\eps)
- \Phi^{\hat \mu \nu}_\eps( f_\eps, g_\eps)]
\lesssim \frac1a \Big( \frac{L}{\eps} \Big)^{d_\nu} \cdot
\frac1n + a\E[\|\hat g_\eps - g_\eps\|^2_{L^2(\nu)}].
$$
\end{lemma}

\begin{lemma}\label{lem:nu_hat_g}
We have
$$
\E[\nu(\hat g_\eps)^2] \lesssim \big( \eps^2 + \frac{1}{\eps^2}\big) \Big(\frac{L}{\eps}  \Big)^{11d_\nu + 4}
\cdot \frac{1}{n}
$$
\end{lemma}

\begin{proof}[Proof of Lemma~\ref{lem:pop_g_convergence}]
Note that
$$
    \E[\|\hat g_\eps - g_\eps\|^2_{L^2(\nu)}]
    \lesssim \E[\|\hat g_\eps - g_\eps - \nu(\hat g_\eps)\|^2_{L^2(\nu)}]
    + \E[\nu(\hat g_\eps)^2].
$$ We can thus apply each of the above lemmas
to yield, for $a, a' > 0$,
\begin{align*}
    \E[\|\hat g_\eps - g_\eps\|^2_{L^2(\nu)}]
   &\lesssim
   \frac{L}{a}\Big(\frac{L}{\eps} \Big)^{d_\nu + 1} 
   + a\frac{L^2}{\eps}\E[ \|\hat g_\eps - g_\eps\|^2_{L^2(\nu)}] \\
   &+ \frac{a'}{\eps}\Big(\frac{L}{\eps} \Big)^{2d_\nu + 2}
   \E[ \|\hat g_\eps - g_\eps\|^2_{L^2(\nu)}] 
   +
   L\big(\frac{a'}{\eps^2} + \frac{1}{a'}\big) \big(\eps^2 + \frac{1}{\eps^2} \big)\Big(\frac{L}{\eps} \Big)^{11 d_\nu + 5} 
\cdot \frac1n \\
&+ \big( \eps^2 + \frac{1}{\eps^2}\big) \Big(\frac{L}{\eps}  \Big)^{11d_\nu + 4}
\cdot \frac{1}{n}.
\end{align*}
Taking $a = C \eps/L^2$
and $a' = C \eps(\eps/L)^{2d_\nu + 2}$
for $C$ a sufficiently small constant and re-arranging
yields
$$
 \E[\|\hat g_\eps - g_\eps\|^2_{L^2(\nu)}]
    \lesssim \Big\{ 
    L^2 \Big(\frac{L}{\eps}  \Big)^{d_\nu + 2} 
    + \big(\eps^2 + \frac{1}{\eps^2} \big)
\Big(\frac{L}{\eps} \Big)^{13 d_\nu + 8} 
    +\big( \eps^2 + \frac{1}{\eps^2}\big) \Big(\frac{L}{\eps}  \Big)^{11d_\nu + 4} 
    \Big\} \cdot \frac1n.
$$ Using $\eps/L \lesssim 1$ 
yields the result.
\end{proof}
\begin{proof}[Proof of Lemma~\ref{lem:pop_g_by_g_maps}]
Note that the term on the LHS is precisely
$\Var_\nu(\hat g_\eps - g_\eps)$ since
we have specified $g_\eps$ to be such that $\nu(g_\eps) = 0$.
We will thus apply the Poincaré inequality
on $\nu$ 
from Proposition~\ref{prop:weighted_poincare}
to bound this term; this is legitimate
because $\hat g_\eps$ and $g_\eps$ are each Lipschitz
with respect to the ambient Euclidean norm
and so are Lipschitz with respect to $(N, h)$, since
embedded manifold distances are always at least
as large as extrinsic distances.
We obtain
$$
\|\hat g_\eps - g_\eps - \nu(\hat g_\eps)\|^2_{L^2(\nu)}
\lesssim \int \|\nabla_N(\hat g_\eps - g_\eps)\|_h^2
\ud \nu(y).
$$ Since $N$
is an embedded manifold,
the manifold norm of the manifold gradient
is always upper bounded by the Euclidean
norm of the Euclidean gradient
(for the reader's convenience,
this fact is formally stated
and proved as Proposition~\ref{prop:embedded_gradient_vs_ambient}
in Appendix~\ref{subsec:additional_embedded})
so that 
$$
\|\hat g_\eps - g_\eps - \nu(\hat g_\eps)\|^2_{L^2(\nu)}
\lesssim \|\nabla \hat g_\eps - \nabla g_\eps\|^2_{L^2(\nu)}.
$$ To calculate $\nabla  g_\eps$,
consider the marginal constraint
$$
e^{-\frac{1}{\eps} g_\eps(y)} = \int e^{-\frac{1}{\eps}(c(x, y)
- f_\eps(x))} \ud \mu(x).
$$ Differentiating and re-arranging, we find that
$$
\nabla g_\eps(y) = \E_{\pi_\eps}[\nabla_y c(x, y) \, | \, y].
$$ Calculating in a similar fashion with the extended dual potential
$\hat g_\eps$ we find
$$
\nabla \hat g_\eps(y) = \E_{\hat \pi_\eps}[\nabla_y c(x, y) \, | \, y] = \frac1n \sum_{i = 1}^n \nabla_y c(x_i, y) \hat p_\eps(x_i, y).
$$
The result follows.
\end{proof}

\begin{proof}[Proof of Lemma~\ref{lem:g_maps_by_semi_empirical}]
For all $y$ in $\supp(\nu)$, put
$$
\tilde g_\eps(y) :=- \eps \log \Big( \frac1n
\sum_{i = 1}^n e^{-\eps^{-1}(c(x_i, y) - f_\eps(x_i))}\Big),
$$ and let
$$
\tilde p_\eps(x, y) := e^{-\eps^{-1}(c(x,y)
- f_\eps(x) - \tilde g_\eps(y)}.
$$ Then by Young's inequality,
\begin{align*}
    \E\Big[\Big\|&\E_{\pi_\eps}[\nabla_y c(x, y) \, | \, y]
-\frac1n \sum_{i = 1}^n  \nabla_y c(x_i, y) \hat p_\eps(x_i, y)\Big\|_{L^2(\nu)}^2
\Big] 
\lesssim  \\
&\E\Big[\Big\|\E_{\pi_\eps}[\nabla_y c(x, y) \, | \, y]
- \frac1n \sum_{i = 1}^n \nabla_y c(x_i, y) p_\eps(x_i, y) \Big\|^2_{L^2(\nu)}\Big] \\
&+ \E\Big[\Big\|\frac1n \sum_{i = 1}^n \nabla_y c(x_i, y) p_\eps(x_i, y)
- \frac1n \sum_{i = 1}^n \nabla_y c(x_i, y) \tilde p_\eps(x_i, y) \Big\|^2_{L^2(\nu)}\Big] \\
&+\E\Big[\Big\|\frac1n \sum_{i = 1}^n \nabla_y c(x_i, y) \tilde p_\eps(x_i, y)
- \frac1n \sum_{i = 1}^n \nabla_y c(x_i, y) \hat p_\eps(x_i, y) \Big\|^2_{L^2(\nu)}\Big].
\end{align*}
Using the same argument as in the
proof of Lemma~\ref{lem:map_mse_massage}
just with the role of $\mu$ and $\nu$ swapped,
the first two terms can each be bounded
as $\lesssim L^2\|p_\eps\|^2_{L^2(\mu \otimes \nu)}/n$.
By
Lemma~\ref{lem:variance_control}
and Proposition~\ref{prop:N_covering_numbers},
we can
bound $\|p_\eps\|^2_{L^2(\mu \otimes \nu)}
\lesssim (L/\eps)^{d_\nu}$
and so arrive at
\begin{align*}
    \E\Big[\Big\|&\E_{\pi_\eps}[\nabla_y c(x, y) \, | \, y]
-\frac1n \sum_{i = 1}^n  \nabla_y c(x_i, y) \hat p_\eps(x_i, y)\Big\|_{L^2(\nu)}^2
\Big] 
\lesssim \\ 
&\E\Big[\Big\|\frac1n \sum_{i = 1}^n \nabla_y c(x_i, y) \tilde p_\eps(x_i, y)
- \frac1n \sum_{i = 1}^n \nabla_y c(x_i, y) \hat p_\eps(x_i, y) \Big\|^2_{L^2(\nu)}\Big] + 
L^2 \Big(\frac{L}{\eps} \Big)^{d_\nu}\cdot \frac1n.
\end{align*}
We proceed as in the proof of
Lemma~\ref{lem:map_mse_kl}, by first
applying triangle inequality
and then Pinsker's. This yields
\begin{align*}
   \E\Big[\Big\|\frac1n \sum_{i = 1}^n &\nabla_y c(x_i, y) \tilde p_\eps(x_i, y)
- \frac1n \sum_{i = 1}^n \nabla_y c(x_i, y) \hat p_\eps(x_i, y) \Big\|^2_{L^2(\nu)}\Big] 
\\
&\leqslant L^2
\E\Big[\Big\|\frac1n \sum_{i = 1}^n |\tilde p_\eps(x_i, y)
- \hat p_\eps(x_i, y) |\Big\|^2_{L^2(\nu)}\Big] \\
&\lesssim L^2 \E\Big[\int \KL{\frac1n \hat p_\eps(\cdot, y)}{\frac 1n \tilde p_\eps(\cdot, y)} \ud \nu(y) \Big] \\
&=\frac{L^2}{\eps} \E\Big[
\int\Big( \frac1n \sum_{i = 1}^n
\hat p_\eps(x_i, y)(\hat f_\eps(x_i)
- f_\eps(x_i) + \hat g_\eps(y) - \tilde g_\eps(y))
\Big) \ud \nu(y)
\Big] \\
&= \frac{L^2}{\eps}
\E\Big[\Phi_\eps^{\hat \mu \nu}(\hat f_\eps,
\hat g_\eps) - \Phi_\eps^{\hat \mu \nu}(f_\eps, 
\tilde g_\eps) + \Big\langle \int \hat p_\eps(\cdot,y)
\ud \nu(y) - 1,f_\eps - f_\eps \Big\rangle_{L^2(\hat \mu)} \Big] \\
&\leqslant 
\frac{L^2}{\eps}
\E\Big[\Phi_\eps^{\hat \mu \nu}(\hat f_\eps,
\hat g_\eps) - \Phi_\eps^{\hat \mu \nu}(f_\eps, g_\eps) + \Big\langle \int \hat p_\eps(\cdot,y)
\ud \nu(y) - 1,\hat f_\eps - f_\eps \Big\rangle_{L^2(\hat \mu)} \Big],
\end{align*}
where the second equality follows because
$(\hat \mu \otimes \nu)(\hat p_\eps) = (\hat \mu \otimes \nu)(\tilde p_\eps) = 1$,
and the final inequality is via Proposition~\ref{prop:entropic_dual_basics}.
To conclude the argument, we apply
Young's inequality to the inner product term
to obtain, for all $a > 0$,
\begin{align*}
\E\Big[\Big\langle \int \hat p_\eps(\cdot,y)
\ud \nu(y) - 1,\hat f_\eps - f_\eps \Big\rangle_{L^2(\hat \mu)} \Big]
&\leqslant a \E\Big[ \Big\| 
\int (\hat p_\eps(\cdot, y ) - 1) \ud \nu(y) \Big\|^2_{L^2(\hat \mu)}\Big] \\
&+ \frac1a \E[\|\hat f_\eps - f_\eps\|^2_{L^2(\hat \nu)}] \\
&= a \E\Big[ \Big\| 
\int (\hat p_\eps(\cdot, y ) - p_\eps(\cdot, y)) \ud \nu(y) \Big\|^2_{L^2(\hat \mu)}\Big] \\
&+ \frac1a \E[\|\hat f_\eps - f_\eps\|^2_{L^2(\hat \nu)}] \\
&\leqslant a \E[\| \hat p_\eps - p_\eps \|_{L^2(\hat \mu \otimes \nu)}^2]
+ \frac1a  \E[\|\hat f_\eps - f_\eps\|^2_{L^2(\hat \nu)}].
\end{align*}
By Lemma~\ref{lem:L2_density}, which is stated and proved
in the following section, 
$$
\E[\|\hat p_\eps - p_\eps \|_{L^2(\hat \mu \otimes \nu)}^2] \lesssim \frac{1}{\eps^2}\cdot \Big( \frac{L}{\eps} \Big)^{2d_\nu}\E[ \|\hat f_\eps - f_\eps\|^2_{L^2(\hat \mu)}
+ \|\hat g_\eps - g_\eps\|^2_{L^2(\nu)}] + \frac1n.
$$ Applying Lemma~\ref{lem:f_control_by_g}
and Lemma~\ref{lem:emp_g_convergence} we have
$$
\E[\|\hat f_\eps - f_\eps\|^2_{L^2(\hat \mu)}]
\lesssim \big(\eps^2 + \frac{1}{\eps^2} \big) \cdot \Big( \frac{L}{\eps} \Big)^{9d_\nu + 4} \cdot \frac1n.
$$ The result follows by collecting terms.
\end{proof}

\begin{proof}[Proof of Lemma~\ref{lem:semi_empirical_bound}]
Proposition~\ref{prop:entropic_dual_basics}
on concavity of the
dual implies that
$$
\E[\Phi^{\hat \mu \nu}_\eps( \hat f_\eps, \hat g_\eps)
- \Phi^{\hat \mu \nu}_\eps( f_\eps, g_\eps)]
\leqslant \E [ \langle \nabla \Phi^{\hat \mu \nu}_\eps( f_\eps, g_\eps),
(\hat f_\eps - f_\eps, \hat g_\eps - g_\eps)\rangle_{L^2(\hat \mu) \times L^2(\nu)}].
$$ Observe that because $(f_\eps, g_\eps)$
attain the $\nu$-marginal,
the first component of $\nabla \Phi_\eps^{\hat \mu \nu}$ vanishes, implying
\begin{align*}
\langle \nabla \Phi^{\hat \mu \nu}_\eps( f_\eps, g_\eps),
(\hat f_\eps - f_\eps, \hat g_\eps - g_\eps)\rangle_{L^2(\hat \mu) \times L^2(\nu)} = 
\big\langle 1-  \frac1n \sum_{i = 1}^n p_\eps(x_i, y)),
\hat g_\eps - g_\eps
\big\rangle_{L^2(\nu)}.
\end{align*}
Young's inequality yields, for all $a > 0$,
$$
\E[\Phi^{\hat \mu \nu}_\eps( \hat f_\eps, \hat g_\eps)
- \Phi^{\hat \mu \nu}_\eps( f_\eps, g_\eps)]
\lesssim 
\frac1a \E\big[ \big\|1 - \frac1n \sum_{i = 1}^n p_\eps(x_i, y)\big\|_{L^2(\nu)}^2\big]
+ a \E[\|\hat g_\eps - g_\eps\|^2_{L^2(\nu)}].
$$
The first term can be calculated
as in~\eqref{eqn:gradient_calc} to yield
$$
\E\big[ \big\|1 - \frac1n \sum_{i = 1}^n p_\eps(x_i, y)\big\|_{L^2(\nu)}^2\big]
\lesssim \frac{\|p_\eps\|^2_{L^2(\mu \otimes \nu)}}{n}.
$$ Finally,
applying Lemma~\ref{lem:variance_control}
and Proposition~\ref{prop:N_covering_numbers},
we obtain the result.
\end{proof}

\begin{proof}[Proof of Lemma~\ref{lem:nu_hat_g}]
Note that, since $(\hat \mu \otimes \nu)(\hat p_\eps) = 
(\hat \mu \otimes \nu)(p_\eps) =1$
and we make the normalization assumption
$\nu(g_\eps) = 0$,
we have
$$
\nu(\hat g_\eps)
= \Phi^{\hat \mu \nu}_\eps(\hat f_\eps, 
\hat g_\eps)
- \Phi^{\hat \mu \nu}_\eps(f_\eps, g_\eps) +
\hat \mu(f_\eps - \hat f_\eps).
$$ By Jensen's inequality, we have that
$$
|\nu(\hat g_\eps)|
\leqslant |\Phi^{\hat \mu \nu}_\eps(\hat f_\eps, 
\hat g_\eps)
- \Phi^{\hat \mu \nu}_\eps(f_\eps, g_\eps) |+
\|\hat f_\eps - f_\eps\|_{L^2(\hat \mu)}.
$$
Observe that
the two-sided concavity statements from
Proposition~\ref{prop:entropic_dual_basics}
imply that
\begin{align*}
|\Phi^{\hat \mu \nu}_\eps(\hat f_\eps, 
\hat g_\eps)
- \Phi^{\hat \mu \nu}_\eps(f_\eps, g_\eps) |
&\leqslant 
|\langle \nabla \Phi^{\hat \mu \nu}_\eps(f_\eps,
g_\eps), (\hat f_\eps - f_\eps, \hat g_\eps - g_\eps)\rangle_{L^2(\hat \mu) \times L^2(\nu)}|\\
&+|\langle \nabla \Phi^{\hat \mu \nu}_\eps(\hat f_\eps,
\hat g_\eps), (\hat f_\eps - f_\eps, \hat g_\eps - g_\eps)\rangle_{L^2(\hat \mu) \times L^2(\nu)}|.
\end{align*} For the first term, Cauchy-Schwarz
and Proposition~\ref{prop:dual_pointwise_control}
on uniform bounds for the dual
potentials imply that
$$
|\langle \nabla \Phi^{\hat \mu \nu}_\eps(f_\eps,
g_\eps), (\hat f_\eps - f_\eps, \hat g_\eps - g_\eps)\rangle_{L^2(\hat \mu) \times L^2(\nu)}|
\lesssim \|\nabla \Phi_\eps^{\hat \mu \nu}
(f_\eps, g_\eps)\|_{L^2(\hat \mu) \times 
L^2(\nu)}.
$$ For the second term,
since $(\hat f_\eps, \hat g_\eps)$ attain
the $\hat \mu$
marginal, the second component of
$\nabla \Phi^{\hat \mu \nu}_\eps(\hat f_\eps, \hat g_\eps)$ vanishes, so that 
\begin{align*}
|\langle \nabla \Phi^{\hat \mu \nu}_\eps(\hat f_\eps,
\hat g_\eps), (\hat f_\eps - f_\eps, \hat g_\eps - g_\eps)\rangle_{L^2(\hat \mu) \times L^2(\nu)}|
&= \big|\big\langle 1 - \int \hat p_\eps(x, y) \ud \nu(y),
\hat f_\eps - f_\eps \big\rangle_{L^2(\hat \mu)}\big| \\
&\leqslant \|\int \hat p_\eps(x, y)\ud \nu(y) - 1 \|_{L^2(\hat \mu)}
\|\hat f_\eps - f_\eps\|_{L^2(\hat \mu)} \\
&\lesssim \|\hat p_\eps\|_{L^{\infty}(\hat \mu \otimes \nu)}
\|\hat f_\eps - f_\eps\|_{L^2(\hat \mu)}.
\end{align*}
Now, let $\mathcal{E}$ denote the event
described in Lemma~\ref{lem:manifold_density_sup_control}.
Then by the pointwise control given in Proposition~\ref{prop:dual_pointwise_control},
we have
\begin{align*}
\E[ \|\hat p_\eps\|_{L^{\infty}(\hat \mu \otimes \nu)}^2
\|\hat f_\eps - f_\eps\|_{L^2(\hat \mu)}^2]
&=
\E[ \mathbbold{1}[\mathcal{E}]\|\hat p_\eps\|_{L^{\infty}(\hat \mu \otimes \nu)}^2
\|\hat f_\eps - f_\eps\|_{L^2(\hat \mu)}^2]\\
&+ \E[ \mathbbold{1}[\mathcal{E}^c]\|\hat p_\eps\|_{L^{\infty}(\hat \mu \otimes \nu)}^2
\|\hat f_\eps - f_\eps\|_{L^2(\hat \mu)}^2] \\
&\lesssim 
\E[ \mathbbold{1}[\mathcal{E}]\|\hat p_\eps\|_{L^{\infty}(\hat \mu \otimes \nu)}^2
\|\hat f_\eps - f_\eps\|_{L^2(\hat \mu)}^2]
+ \E[ \mathbbold{1}[\mathcal{E}^c]] e^{\frac{10}{\eps}} \\
&\lesssim 
\E[ \mathbbold{1}[\mathcal{E}]\|\hat p_\eps\|_{L^{\infty}(\hat \mu \otimes \nu)}^2
\|\hat f_\eps - f_\eps\|_{L^2(\hat \mu)}^2] + \frac1n \\
&\lesssim 
\Big(\frac{L}{\eps} \Big)^{2d_\nu}
\E[ \mathbbold{1}[\mathcal{E}]
\|\hat f_\eps - f_\eps\|_{L^2(\hat \mu)}^2] + \frac1n \\
&\leqslant 
\Big(\frac{L}{\eps} \Big)^{2d_\nu}
\E[\|\hat f_\eps - f_\eps\|_{L^2(\hat \mu)}^2] + \frac1n \\
\end{align*}
We thus conclude that
\begin{align*}
\E[\nu(\hat g_\eps)^2]&\lesssim 
\E\big[
(\langle \nabla \Phi^{\hat \mu \nu}_\eps(f_\eps,
g_\eps), (\hat f_\eps - f_\eps, \hat g_\eps - g_\eps)\rangle_{L^2(\hat \mu) \times L^2(\nu)})^2\\
&+(\langle \nabla \Phi^{\hat \mu \nu}_\eps(\hat f_\eps,
\hat g_\eps), (\hat f_\eps - f_\eps, \hat g_\eps - g_\eps)\rangle_{L^2(\hat \mu) \times L^2(\nu)})^2 
+ \|\hat f_\eps - f_\eps\|_{L^2(\hat \mu)}^2 \big]\\
&\lesssim 
\E\big[\|\nabla \Phi_\eps^{\hat \mu \nu}
(f_\eps, g_\eps)\|_{L^2(\hat \mu) \times 
L^2(\nu)}^2   + \Big(\frac{L}{\eps} \Big)^{2d_\nu}
\|\hat f_\eps - f_\eps\|_{L^2(\hat \mu)}^2\big] + \frac1n.
\end{align*}
The squared gradient term can be calculated
as in~\eqref{eqn:gradient_calc}
and then bounded with Lemma~\ref{lem:variance_control}
and Proposition~\ref{prop:N_covering_numbers}, and the
latter bounded with Lemma~\ref{lem:f_control_by_g}
and Lemma~\ref{lem:emp_g_convergence},
to yield the result.
\end{proof}

\subsection{Fast rates with MID scaling for maps and densities}
\label{subsec:fast_mid_maps}

In this section, we give
the proof of Corollary~\ref{cor:manifold_map}
and Corollary~\ref{cor:fast_density}.
To this end, the following Lemma is convenient.

\begin{lemma}\label{lem:L2_density} We have
    $$
    \E[\|\hat p_\eps - p_\eps\|^2_{L^2(\mu \otimes \nu)}]
    \lesssim \frac{1}{\eps^2} \cdot \Big(\frac{L}{\eps} \Big)^{2d_\nu}
    \E[\|\hat f_\eps - f_\eps\|^2_{L^2(\mu)}
    + \|\hat g_\eps - g_\eps\|^2_{L^2(\nu)}] + \frac1n.
    $$
And we also have the analogous bounds
    $$
    \E[\|\hat p_\eps - p_\eps\|^2_{L^2(\mu \otimes \hat\nu)}]
    \lesssim \frac{1}{\eps^2} \cdot \Big(\frac{L}{\eps} \Big)^{2d_\nu}
    \E[\|\hat f_\eps - f_\eps\|^2_{L^2(\mu)}
    + \|\hat g_\eps - g_\eps\|^2_{L^2(\hat\nu)}] + \frac1n,
    $$
    as well as,
    $$
    \E[\|\hat p_\eps - p_\eps\|^2_{L^2(\hat\mu \otimes \nu)}]
    \lesssim \frac{1}{\eps^2} \cdot \Big(\frac{L}{\eps} \Big)^{2d_\nu}
    \E[\|\hat f_\eps - f_\eps\|^2_{L^2(\hat\mu)}
    + \|\hat g_\eps - g_\eps\|^2_{L^2(\nu)}] + \frac1n.
    $$
\end{lemma}
With this lemma in hand, the
proofs of Corollary~\ref{cor:manifold_map}
and Corollary~\ref{cor:fast_density}
are a simple application of the results from the previous section.

\begin{proof}[Proof of Corollary~\ref{cor:manifold_map}]
To prove this result,
let us first assume that
$0 \in \supp(\nu)$
so that for each $y \in \supp(\nu)$,
we have
$\|y\|\leqslant R$.
By Young's inequality,
\begin{align*}
\E[\|\hat T_\eps - T_\eps\|^2_{L^2(\mu)}] &= 
\E[\|\hat T_\eps - T_\eps \|^2_{L^2(\mu)}]\\
&\lesssim 
\E\Big[ \Big\|\hat T_\eps  - \frac{1}{n} \sum_{j = 1}^n y_j
p_\eps(x, y_j) \Big\|_{L^2(\mu)}^2 \Big] 
+ \E\Big[ \Big\|\frac{1}{n} \sum_{j = 1}^n y_j
p_\eps(x, y_j) - T_\eps  \Big\|_{L^2(\mu)}^2 \Big].
\end{align*}
The second term is merely a variance,
and can be bounded as
\begin{align*}
\E\Big[ \Big\|\frac{1}{n} \sum_{j = 1}^n y_j
p_\eps(x, y_j) - T_\eps\Big\|_{L^2(\mu)}^2 \Big]
&= \frac{1}{n} \|y p_\eps(x, y) - T_\eps(x)
\|_{L^2(\mu\otimes \nu)}^2 \\
&\lesssim \frac{R^2}{n}\|p_\eps\|^2_{L^2(\mu \otimes \nu)} \lesssim \frac{R^2}{n}\Big(\frac{L}{\eps} \Big)^{d_\nu},
\end{align*} where the final inequality follows by
applying Lemma~\ref{lem:variance_control}
with Proposition~\ref{prop:N_covering_numbers}.
We now focus on the remaining expectation,
and apply triangle inequality and Jensen's to yield,
for any $x \in \supp \mu$,
\begin{align*}
\Big\|\hat T_\eps - \frac{1}{n} \sum_{j = 1}^n y_j 
p_\eps(x, y_j) \Big\|^2
&\lesssim 
\frac1n \sum_{j = 1}^n 
\|y_j\|^2(\hat p_\eps(x, y_j) - p_\eps(x, y_j))^2\\\
&\lesssim \frac{R^2}{n} \sum_{j = 1}^n
(\hat p_\eps(x, y_j) - p_\eps(x, y_j))^2.
\end{align*}
Taking expectations and applying Lemma~\ref{lem:L2_density}
yields
\begin{align*}
\E\Big[\Big\|\hat T_\eps- \frac{1}{n} \sum_{j = 1}^n y_j 
&p_\eps(x, y_j) \Big\|^2_{L^2(\mu)}  \Big]
\lesssim R^2 \E[\|\hat p_\eps - p_\eps\|_{L^2(\mu \otimes \hat \nu)}] \\
&\lesssim \frac{R^2}{\eps^2} \Big(\frac{L}{\eps} \Big)^{2d_\nu}
\E[\|\hat f_\eps - f_\eps\|^2_{L^2(\mu)}
+ \|\hat g_\eps - g_\eps\|^2_{L^2(\hat \nu)}] +
\frac{R^2}{n}.
\end{align*}
The result follows from Lemma~\ref{lem:f_control_by_g}
and Lemma~\ref{lem:emp_g_convergence}.
For the general case
where $0$ may not be in $\supp(\nu)$,
we can perform the above argument
with suitably offset $y$.
\end{proof}

\begin{proof}[Proof of Corollary~\ref{cor:fast_density}]
By Lemma~\ref{lem:L2_density},
$$
\E[\|\hat p_\eps - p_\eps\|^2_{L^2(\mu \otimes \nu)}]
\lesssim \frac{1}{\eps^2}\Big(\frac{L}{\eps} \Big)^{2d_\nu}
\E[\|\hat f_\eps - f_\eps\|^2_{L^2(\mu)}
+ 
\|\hat g_\eps - g_\eps\|^2_{L^2(\nu)}]
+ \frac{1}{n}.
$$
The result follows by
applying Lemma~\ref{lem:f_control_by_g}
and Lemma~\ref{lem:emp_g_convergence}
to control the $\|\hat f_\eps -  f_\eps\|_{L^2(\mu)}^2$
term, and Lemma~\ref{lem:pop_g_convergence}
to control the $\|\hat g_\eps - g_\eps\|_{L^2(\nu)}^2$ term.
\end{proof}

\paragraph*{Acknowledgements.}
The author gratefully acknowledges
partial support from
NSF awards IIS-1838071 and
CCF-2106377,
and is indebted to 
Enric Boix-Adsera,
Sinho Chewi,
Simone Di Marino,
Augusto Gerolin,
Dheeraj Nagaraj,
Jonathan Niles-Weed,
Aram-Alexandre Pooladian,
Philippe Rigollet,
George Stepaniants,
and Stephanie Wu
for enlightening conversations.
Special thanks go
to Sinho Chewi, Aram-Alexandre Pooladian, and Stephanie Wu for helpful comments
on earlier drafts.

\appendix 

\section{Background on embedded
manifolds and RGGs}\label{sec:background_manifolds}

\subsection{Preliminaries on embedded manifolds}
\label{subsec:prelim_embedded}
For a comprehensive introduction
to Riemannian manifolds,
we refer the reader to the books~\cite{do2016differential,lee2018introduction}.
In this section, we will establish
notation while reviewing the definitions
of some necessary
geometric quantities.

To this end, suppose $(N_0,h_0)$ is a compact, connected Riemannian
manifold. Recall that, in coordinates,
the canonical Riemannian volume has density
$$
d\vol_{N_0} (y) = \sqrt{\det h_0(y)}.
$$ 
Given $p \in N_0$ and $u, v \in T_pN_0$,
we shall write $\langle u, v\rangle_{h_0} := h_0(u, v)$,
and similarly $\|u\|_{h_0}^2 := \langle u, u \rangle_{h_0}$.
Recall that for a point $p \in N_0$,
the exponential map $\exp_p \colon T_pN_0 \to N_0$
is defined as $\exp_p(v) := \gamma_v(1)$ where
$\gamma_v$ is the constant-speed
geodesic such that $\gamma_v(0) = p$ and $\dot{\gamma}_v(0)
= v$.
For a point $p \in N_0$,
the injectivity radius $\inj(p)$ is defined as
$$
\inj(p):= \sup \big\{ R \geqslant 0 \,\, \big| \,  \exp_p \colon B(0, R) \subset T_pN_0
\to N_0 \textrm{ is a diffeomorphism} \big\}.
$$ The injectivity radius of $N_0$ is then defined as
$$
\inj(N_0) := \inf_{p \in N_0} \inj(p).
$$ 
It is an elementary fact that when $N_0$ is compact,
$\inj(N_0) > 0$, see e.g.~\cite[Lemma 6.16]{lee2018introduction}.

The Riemannian curvature tensor $\Rm$
maps from $4$-tuples of smooth vector fields
to smooth functions $\Rm \colon \mathscr{X}(N_0)^{4} \to C^{\infty}(N_0)$
and
is defined as, for smooth vector fields $W, X, Y, Z
\in \mathscr{X}(N_0)$,
$$
\Rm(X, Y, Z, W) := \langle  \nabla_X \nabla_Y
Z - \nabla_Y \nabla_X Z - \nabla_{[X, Y]}Z, W \rangle_{h_0}.
$$ Given $p \in M$ and $u, v \in T_pN_0$ linearly
independent,
 the sectional curvature of $u, v$ is defined as
 $$
 \sec(u, v) := \frac{\Rm(u, v, v, u)}{\|u\|_{h_0}\|v\|_{h_0} - 
 \langle u, v \rangle_{h_0}}.
 $$
 The sectional curvatures
 describe the geometry of the embedded surface
 generated by $u, v$, and when $u,v$ range
 over $T_pN_0$, determine the full Riemannian
 curvature tensor $\Rm$.
 It is, again, an elementary fact that when
 $N_0$ is compact, the sectional curvatures
 are uniformly bounded in absolute value,
 see~\cite[Section 9.3]{bishop1969manifolds}.
 
 Given a smooth manifold $N_0$,
 we say that $N_0$ is embedded in $\R^D$ if
 there is a smooth injection $\iota \colon N_0 \hookrightarrow \R^D$
 which is a homeomorphism onto its image
 and
 such that for all $p \in N_0$, the differential $d \iota_p \colon T_pN_0 \to T_{\iota(p)}\R^D$ has
 trivial kernel.
 In such a case, we identify $N_0$ with its image 
 under $\iota$, and so write $N_0 \subseteq \R^d$.
 When $N_0$ is endowed with a Riemannian structure $(N_0, h_0)$,
 we say that $(N_0, h_0)$ is an embedded Riemannian manifold
 in $\R^d$ if $N_0$ is embedded in $\R^d$ and
 $h_0(u, v) = \langle d \iota_p(u), d \iota_p(v) \rangle$ for
 all $u, v \in T_pN_0$ and the inner product is the
 standard Euclidean one.
 
 The main quantitative property of subsets in $\R^D$ we
 need is called the reach.
 To define this quantity, suppose $S \subseteq \R^D$,
 and for all $x \in \R^D$ let $d(x, S) := \inf_{y \in S}
 \|x - y\|$. Then $\reach(S)$ is
 the supremum over all $\eps \geqslant 0$ such that
 for all $x \in \R^D$ for which $d(x, S) \leqslant \eps$,
 there is a unique $y \in S$ so that
 $\|x - y\| = d(x, S)$.
 Once more, in our setting of compact
 smooth embedded manifolds $N_0$,
 $\reach(N_0) > 0$~\cite[Prop. 14]{thale200850}.

\subsection{Spectral gap for the RGG of $\nu$}
\label{subsec:rgg_spectral}
The interface between our results and those
relating to embedded Riemannian manifolds in Euclidean
spaces is primarily contained
in Theorem~\ref{thm:graph_spectrum} below,
which gives a spectral gap for the random geometric
graph (RGG) of $\nu$, and is used
to establish
our main technical result in this embedded manifold
setting, Lemma~\ref{lem:qg_emp_dual},
which gives a weak
form of strong concavity for the empirical
dual function. 

The RGG of $\nu$ is defined, for a fixed
threshold $\delta > 0$,
as the graph on $[n]$ with weights
for $j, k \in [n]$ defined as
$$
w_{jk} := \begin{cases}
\frac{C}{n\delta^{d_\nu + 2}} & \textrm{ if   } \|y_j - y_k\| < \delta, \\
0 & \textrm{ else},
\end{cases}
$$ where $C = C(d_\nu)$ is constant
only depending on the intrinsic dimension $d_\nu$.
We write $j \sim k$ when $\|y_j - y_k\| < \delta$,
and the resulting weighted graph on $[n]$
is written as $\Gamma$.
For $\alpha \in L^{\infty}(\hat \nu)$, define
the associated Dirichlet form by
$$
D(\alpha) := \frac{1}{n}  \sum_{j\sim k}
w_{jk} (\alpha(y_{k}) - \alpha(y_{j}))^2,
$$ where
$C = C(d_\nu)$ is a constant only
depending on the intrinsic dimension $d_\nu$.
The un-normalized graph Laplacian on $\Gamma$ with respect to $L^2(\hat \nu)$ is then
$$
\Delta_\Gamma(\alpha)(y_j) := 
\sum_{k \colon k \sim j}w_{jk} (\alpha(y_k) - \alpha(y_j)).
$$

We seek a spectral
gap for $\Delta_\Gamma$, meaning that its second
eigenvalue is bounded away from $0$.
To sketch how this is possible,
note that
under our assumptions on $N$ and $\nu$,
such a result holds for the continuous analog
of $\Delta_\Gamma$. By employing
recent work on the convergence of the spectrum
of RGGs to their continuous analogs,
we can transfer a continuous spectral
gap to a (random) discrete one~\cite{garcia2020error}.
The result we ultimately obtain is as follows.

\begin{theorem}[Spectral gap for the RGG of $\nu$]
\label{thm:graph_spectrum}
Suppose $\delta \lesssim 1$
and $n$ is large enough that $\delta^{d_\nu} \gtrsim
(\log n)/n$. Then with probability at least $1 - C/n$,
$\lambda_2(\Delta_\Gamma) \gtrsim 1$.
In this case, we have the inequality
$$
D(\alpha) \gtrsim \hat \nu((\alpha - \hat \nu(\alpha))^2) \quad \quad \forall \alpha  
\in L^{\infty}(\hat \nu).
$$
\end{theorem}

To derive this theorem from the results of~\cite{garcia2020error},
we first introduce
the continuous analog
of the un-normalized graph Laplacian and subsequently discuss its spectrum.

\paragraph*{A certain weighted Laplacian.}
For a thorough introduction to weighted Riemannian Laplacians
we refer the reader to the book~\cite{grigoryan2009heat}.
Suppose $\nu$ has density $p\colon N \to \R_{\geqslant 0}$
with respect $\vol_N$. Consider the weighted Laplacian
operator
$$
\Delta_\nu(\cdot) := -\frac{1}{p}\div(p^2\nabla \cdot),
$$ where $\div$ is the Riemannian divergence on $N$.
In fact, this weighted Laplacian is the correct continuous
analog of the un-normalized graph Laplacian $\Delta_\Gamma$~\cite{bousquet2003measure,hein2007graph},
and we will thus consider its spectrum.
 Because $p$ is bounded
 away from $0$ and $N$ is assumed to be compact, 
 $\Delta_\nu$ has a discrete,
 non-negative spectrum,
 written $0 = \lambda_1 \leqslant \lambda_2 \leqslant \cdots$
 (with multiplicity), and the minimax principle holds~\cite[Section 1.4]{garcia2020error}.
 Since $N$ is connected and $p$ is bounded above and
 away from $0$, it follows that $\lambda_2 > 0$.
 Summarizing, we have the following Proposition.
\begin{prop}\label{prop:weighted_laplacian}
The operator $\Delta_\nu$ has a discrete, non-negative spectrum
$\{\lambda_k\}_{k \geqslant 1}$,
and $\lambda_1 = 0$,
while $\lambda_2 > 0$.
\end{prop}
We we will
simply treat $\lambda_2$ as a constant
depending on $N$ and $\nu$,
but we remark that
precise quantitative control
on $\lambda_2$ is intimately related to
the Ricci curvature
of the manifold $N$~\cite[Chapter 5]{li2012geometric}.

\paragraph*{Spectral convergence of the RGG of $\nu$.}
The results of~\cite{garcia2020error}
quantify the convergence of the spectrum
of $\Delta_\Gamma$ to that of $\Delta_\nu$.
For our purposes, it is enough to consider
their result only in terms of lower bound control
of the second eigenvalue of $\Delta_\Gamma$,
namely $\lambda_2(\Delta_\Gamma)$, in terms
of the second eigenvalue of $\Delta_\nu$,
namely $\lambda_2$,
but we emphasize that their results are far stronger
than we represent here.
To make the geometric constants involved in this
bound explicit, let $i_0 := \inj(N)$ denote
the injectivity radius of $N$,
$K$ be a uniform upper bound
on the absolute value of the sectional curvatures of $N$,
and $r := \reach(N)$ the reach (these
quantities are defined in the previous section). Since $N$ is compact,
$i_0, r > 0$ and $K < \infty$.

\begin{theorem}[Theorem 4 (adapted)~\cite{garcia2020error}]
\label{thm:lambda_2_lower_garcia}
Suppose $\delta > 0$ is such that
$$
\delta < \min \big\{ 1, \frac{i_0}{10}, \frac{1}{\sqrt{d_\nu K}}, \frac{r}{\sqrt{27d_\nu}} \big\}.
$$ Suppose $W_{\infty}(\hat \nu, \nu)(d_\nu + 5) < \delta$.
Then, if
$$
\sqrt{\lambda_2}\delta < C
$$ we have
$$
\lambda_2(\Delta_\Gamma) \geqslant \lambda_2(1 - C(\delta + 
\frac{1}{\delta} W_{\infty}(\hat\nu, \nu) + \sqrt{\lambda_2}\delta + K \delta^2)),
$$ where the constants $C$ only depend on $d_\nu$
and uniform bounds on $p$ and its Lipschitz constant.
\end{theorem}

This result combined
with Proposition~\ref{prop:weighted_laplacian}
shows that, so long as $\delta$ is sufficiently
small relative to constants depending on the geometry
of $N$, the intrinsic dimension
$d_\nu$, and uniform bounds on $p$ and its Lipschitz constant,
and $W_{\infty}(\hat \nu,  \nu) \lesssim \delta$,
we have
$\lambda_2(\Delta_\Gamma) \gtrsim 1$.

To make this bound usable, we also need
to get control on $W_{\infty}(\hat \nu, \nu)$.

\begin{theorem}[Theorem 2 (adapted)~\cite{garcia2020error}]
\label{thm:w_infty_garcia}
With probability at least $1 - C/n$,
$$
W_{\infty}(\hat \mu, \mu) \lesssim \Big(
\frac{\log n}{n} \Big)^{1/d_\nu},
$$ where the constants
only depend on $d_\nu$,
upper and lower bounds on
$p$ and its Lipschitz constant,
and
the injectivity radius of $N$,
uniform bound on the absolute
value of the sectional curvature of $N$,
and the reach of $N$.
\end{theorem}

Combining Proposition~\ref{prop:weighted_laplacian},
Theorem~\ref{thm:lambda_2_lower_garcia},
and Theorem~\ref{thm:w_infty_garcia},
we obtain Theorem~\ref{thm:graph_spectrum}.

\paragraph*{Scale of our results.} 
The statement of Theorem~\ref{thm:lambda_2_lower_garcia}
allows us
to describe the required
size of $\delta$ in Theorem~\ref{thm:graph_spectrum},
and hence $\eps/L$ in our results from section~\ref{subsec:embedded},
in terms of the injectivity radius $i_0$,
uniform bound $K$ on the
absolute value of the sectional curvature,
reach $r$ (these geometric 
quantities are described in detail in section~\ref{subsec:prelim_embedded})
and second eigenvalue $\lambda_2$,
as 
\begin{equation}\label{eqn:scale}
\frac{\eps}{L} < \min \Big\{ 1, \frac{i_0}{10},
\frac{1}{\sqrt{d_\nu K}}, \frac{r}{\sqrt{27d_\nu}}, 
\frac{C}{\sqrt{\lambda_2}} \Big\},
\end{equation} where $C$ is a constant depending on $d_\nu$, uniform
upper and lower bounds on $p$ and its Lipschitz
constant. In fact, we will use two
more small-scale
facts about the embedded manifold
$N$ beyond
Theorem~\ref{thm:lambda_2_lower_garcia},
namely Proposition~\ref{prop:nu_volume_control}
and Proposition~\ref{prop:N_covering_numbers} below,
but this bound is enough to permit their use.

\subsection{Additional facts about embedded manifolds}
\label{subsec:additional_embedded}

We also use Assumptions~\ref{assum:nu_manifold}
and~\ref{assum:nu} to derive a Poincaré inequality
for $\nu$ in Proposition~\ref{prop:weighted_poincare},
as well as to give
convenient bounds on the covering numbers
$\mathcal{N}(\nu, \cdot)$
in Proposition~\ref{prop:N_covering_numbers}.

Similar reasoning as in the previous section
implies that the
weighted Laplacian $\Delta(\cdot) = -\frac{1}{p} \div(p\nabla \cdot)$
sports a spectral gap, and entails the following result.
\begin{prop}\label{prop:weighted_poincare}
For all locally Lipschitz functions $\zeta \colon N \to \R$
we have
$$
\Var_{\nu}(\zeta) \lesssim 
\int \|\nabla_N \zeta(y)\|^2_h \ud \nu(y).
$$
\end{prop}

The following Proposition is an elementary
consequence of the fact that $N$
is an embedded manifold of $\R^d$.
\begin{prop}\label{prop:embedded_gradient_vs_ambient}
Let $\zeta \colon N \to \R$ and
$p \in N$.
Suppose $\zeta$ has a $\nabla_N$ gradient
at $p$, and admits a local extension
$\bar \zeta \colon U \to \R$ on some neighborhood
$U \subset \R^d$
of $p$ such that $\bar \zeta$ has
a Euclidean gradient at $p$.
Then
$$
\|\nabla_N \zeta(p)\|_h^2 \leqslant \|\nabla \bar \zeta(p)\|^2.
$$
\end{prop}

\begin{proof}[Proof of Proposition~\ref{prop:embedded_gradient_vs_ambient}]
Let $\iota \colon N \hookrightarrow \R^d$
be the embedding of $N$ into $\R^d$.
Then for all $v \in T_pN$,
$$
d\zeta_p(v) = \langle \nabla_N \zeta(p), v \rangle_h
= \langle d\iota_p(\nabla_N \zeta(p)), d\iota_p(v) \rangle,
$$ and on the other hand,
$$
d\zeta_p(v) = d\bar\zeta_{\iota(p)} \circ d\iota_p(v) = 
\langle \nabla \bar \zeta(\iota(p)), d\iota_p(v) \rangle.
$$ Since this is true for all $v\in T_pN$,
it follows that $d\iota_p(\nabla_N \zeta(p))$ is the orthogonal projection
of $\nabla \bar \zeta(p)$ onto $\img(d\iota_p) \subset T_{\iota(p)} \R^d$,
and so
$$
\|\nabla_N \zeta(p)\|^2_h = 
\|d\iota_p(\nabla_N \zeta(p))\|^2
\leqslant \|\nabla \bar \zeta(p)\|^2.
$$
\end{proof}

\begin{prop}
\label{prop:nu_volume_control}
Let $i_0:= \inj(N)$ 
be the injectivity radius of $N$
$K$ be a uniform upper bound on the absolute
value of the sectional curvatures of $N$, and $r:=\reach(N)$.
Suppose $\tau > 0$ is such that
$$
\tau \leqslant \min \Big\{ \frac{1}{\sqrt{K}}, i_0 , \frac{r}{2} \Big\}.
$$
Then for all $y \in N$,
$$
\tau^{d_\nu} \lesssim  \nu(B(y, \tau)) 
\lesssim \tau^{d_\nu}.
$$
\end{prop}
\begin{proof}[Proof of Proposition~\ref{prop:nu_volume_control}]
Fix some $y \in N$.
By~\cite[Proposition 2]{garcia2020error}, for any $y' \in N$
such that $\|y - y'\| \leqslant r/2$,
$$
\|y - y'\| \leqslant \dist_N(y, y') \leqslant 
\|y - y'\| + \frac{8}{r^2} \|y - y'\|^3
$$
Hence
$$
B_{\dist_N}(y, \tau) \subset B_{\|\cdot\|}(y, \tau) 
\cap N
\subset B_{\dist_N}\big(y, \tau + \frac{8}{r^2} \tau^3\big).
$$ Whence
$$
\nu(B(y, \tau)) = \nu(B_{\|\cdot\|}(y, \tau)) \geqslant
\nu(B_{\dist_N}(y, \tau))\gtrsim \vol_N(B_{\dist_N}(y, \tau)),
$$ where the final inequality follows because the density
of $\nu$ with respect to $\vol_N$ is bounded
away from $0$ under Assumptions~\ref{assum:nu_manifold}
and~\ref{assum:nu}. 
By~\cite[Equation 1.35]{garcia2020error},
when $\tau \lesssim 1$,
$$
\vol_N(B_{\dist_N}(y, \tau)) \gtrsim \tau^{d_\nu},
$$ yielding the lower bound. Similarly,
$$
\nu(B(y, \tau)) = \nu(B_{\|\cdot\|}(y, \tau)) \leqslant 
\nu\big(B_{\dist_N}\big(y, \tau + \frac{8}{r^2} \tau^3\big)\big)
\lesssim \vol_N\big(B_{\dist_N}\big(y,\tau + \frac{8}{r^2} \tau^3\big)\big),
$$ and we can conclude the upper bound
with another application of~\cite[Equation 1.35]{garcia2020error} where we have the opposite inequality
$$
\vol_N\big(B_{\dist_N}\big(y,\tau + \frac{8}{r^2} \tau^3\big)
\big) \lesssim \big( \tau + \frac{8}{r^2} \tau^3 \big)^{d_\nu}
\lesssim \tau^{d_\nu}.
$$

\end{proof}

The following control on the covering numbers
of $N$ will also be convenient.
\begin{prop}
\label{prop:N_covering_numbers}
Suppose $\tau$ is as in Proposition~\ref{prop:nu_volume_control}. Then
$$
\tau^{-d_\nu} \lesssim \mathcal{N}(N,\|\cdot\|, \tau) \lesssim 
\tau^{-d_\nu}.
$$
\end{prop}

\begin{proof}[Proof of Proposition~\ref{prop:N_covering_numbers}]
For the upper bound, take a maximal $\tau/2$ packing of $N$
in the Euclidean norm, $z_1, \ldots, z_K$. Then
$$
1 \geqslant \sum_{k = 1}^K \nu(B(z_k, \tau/2)).
$$ For $\tau \lesssim 1$, by Proposition~\ref{prop:nu_volume_control},
$$
1 \gtrsim K \tau^{d_\nu}.
$$ Observe that since the packing is maximal,
$K \geqslant \mathcal{N}(N,\|\cdot\|, \tau)$, completing the proof of the upper bound.

For the lower bound, we argue similarly.
Take any $\tau$ cover of $N$
in the Euclidean norm, $z_1, \ldots, z_K$. Then
by Proposition~\ref{prop:nu_volume_control}
$$
1 \leqslant \sum_{k = 1}^K \nu(B(z_k, \tau))
\lesssim K \tau^{d_\nu}.
$$ Taking a minimal $\tau$-covering yields the result.
\end{proof}

\section{Deferred proofs}
\label{sec:deferred}

\subsection{On the tightness of Theorem~\ref{thm:main_cost_covering}}
\label{subsec:tightness}

In this section, we give our results 
concerning
the tightness of Theorem~\ref{thm:main_cost_covering}
in terms of covering number dependence.
Our main observation
in this vein follows.

\begin{prop}\label{prop:tight}
Suppose there exists $r \geqslant 0$
such that for all $\eps > 0$, 
probability measures $\mu, \nu$ and cost functions
$c$ satisfying Assumption~\ref{assumption:cost}, there is a numerical constant such that
$$
\E[|S_\eps(\hat \mu, \hat \nu)
- S_\eps(\mu, \nu)|]
\lesssim (1 + \eps)
\sqrt{\frac{\mathcal{N}\big(\mu, \big(\frac{\eps}{L}\big)^r\big) \land \mathcal{N}\big(\nu, \big(\frac{\eps}{L}\big)^r\big)}{n}}.
$$ Then $r \geqslant 1$.
\end{prop}

To prove this result,
we establish the following fact.

\begin{prop}[Entropic estimation
of $W_1$ distances]\label{prop:W1_rate}
Suppose $c(x, y) = \|x -y\|$,
and there exists $r \geqslant 0$ 
such that
for all $\eps > 0$,
 and
all $\mu ,\nu$ supported in $B(0, 1/2)$,
there is a numerical constant for which
$$
\E[|S_\eps(\hat \mu, \hat \nu)
- S_\eps(\mu, \nu)|]
\lesssim (1 + \eps)
\sqrt{\frac{\mathcal{N}\big(\mu, \big(\frac{\eps}{L}\big)^r\big) \land \mathcal{N}\big(\nu, \big(\frac{\eps}{L}\big)^r\big)}{n}}.
$$
Then for dimension-dependent
constants $C_d, C_d'$,
$$
\mathbb{E}[|S_\eps(\hat \mu, \hat \nu) - W_1(\mu, \nu)|] 
\lesssim C_d \big\{ \frac{1}{rd + 4}  \log \big(C_d' n) + 1 \big\} \cdot n^{-1/(rd + 2)}
$$
In particular, Theorem~\ref{thm:main_cost_covering}
implies the rate $n^{-1/(d + 2)}$ 
for $W_1$ distance estimation, up
to logarithmic factors
and dimension-dependent constants.
\end{prop}

Proposition~\ref{prop:tight}
follows from this result
by recalling that for
$d \geqslant 3$, the minimax rates for $W_1$
distance estimation over all distributions
supported in $B(0, 1/2)$ are $n^{-1/d}$
up to logarithmic factors~\cite[Theorem 11]{niles2022estimation}.

\begin{proof}[Proof of Proposition~\ref{prop:W1_rate}]
Using standard upper bounds on covering numbers
in Euclidean spaces~\cite[Proposition 4.2.12]{vershynin2018high},
the hypothesis implies
\begin{equation}\label{eqn:hypothetical_value_rate}
\mathbb{E}[|S_\eps(\hat \mu, \hat \nu) - S_\eps(\mu, \nu)|] 
 \lesssim (1 + \eps) \cdot \Big(1 + \frac{1}{\eps^r} \Big)^{d/2}
\cdot \frac{1}{\sqrt{n}}.
\end{equation}
Note that for $\eps = 0$ we 
recover the Wasserstein-$1$ distance,
and the rate of approximation is known to be bounded as
\begin{equation}\label{eqn:W_1_approx}
 |S_\eps(\mu, \nu) - W_1(\mu, \nu)| \leqslant C_d \eps \log \Big( \frac{C'_d}{\eps} \Big),
 \end{equation}
 where $C_d, C'_d$ are dimension-dependent constants~\cite[Theorem 1]{genevay2019sample}.
Using $S_\eps(\hat \mu, \hat \nu)$
to estimate $W_1(\mu, \nu)$, suppose
we take $\eps = n^{-1/(rd + 2)}$.
We then find, for
dimension-dependent
constants $C_d, C_d'$ (potentially different than before), that
$$
\mathbb{E}[|S_\eps(\hat \mu, \hat \nu) - W_1(\mu, \nu)|] 
\lesssim C_d \big\{ \frac{1}{rd + 4}  \log \big(C_d' n) + 1 \big\} \cdot n^{-1/(rd + 2)}
$$

\end{proof}

\subsection{Deferred proofs from section~\ref{sec:proofs_fast_rates}}
\label{subsec:deferred_proofs_fast_rates}

\begin{proof}[Proof of Lemma~\ref{lem:empirical_mass_balls}]
Let $z_1, \ldots, z_K$
be a proper $\eps/2L$-net of $N$ such that
$K = \mathcal{N}^{\rm pr}(N, \frac{\eps}{2L}).$
For each $z \in N$, there exists
$k \in [K]$ such that $z \in B(z_k, \frac{\eps}{2L})$,
so that $B(z_k, \frac{\eps}{2L}) \subset B(z, \frac{\eps}{L})$.
Thus
$$
\inf_{z \in N} \hat \nu(B(z, \frac{\eps}{L}))
\geqslant \min_{k \in [K]} \hat \nu(B(z_k, \frac{\eps}{2L})).
$$ Therefore, for all $t > 0$,
\begin{align*}
\P[\inf_{z \in N}&  \hat \nu(B(z, \frac{\eps}{L}))
\leqslant \inf_{z \in N} \nu(B(z, \frac{\eps}{2L}))- t]
 \\
 &\leqslant \P[ \min_{k \in [K]}
 \hat \nu (B(z_k, \frac{\eps}{2L})) \leqslant
 \inf_{z \in N} \nu(B(z, \frac{\eps}{2L})) - t ] \\
 &\leqslant \sum_{k = 1}^K
 \P[\hat \nu (B(z_k, \frac{\eps}{2L})) \leqslant
 \inf_{z \in N} \nu(B(z, \frac{\eps}{2L})) - t] \\
 &\leqslant \sum_{k = 1}^K
 \P[\hat \nu (B(z_k, \frac{\eps}{2L})) \leqslant
 \nu(B(z_k, \frac{\eps}{2L})) - t].
\end{align*}
By Hoeffding's inequality, 
for each $k \in [K]$,
$$
\P[\hat \nu (B(z_k, \frac{\eps}{2L})) \leqslant
 \nu(B(z_k, \frac{\eps}{2L})) - t]
 \leqslant e^{- 2nt^2}.
$$
Therefore,
with probability at least
$1 - \frac{1}{n} e^{-\frac{10}{\eps}}$,
$$
\inf_{z \in N} \hat \nu(B(z, \frac{\eps}{L}))
\gtrsim \inf_{z \in N} \nu(B(z, \frac{\eps}{2L}))
 - \sqrt{\frac{1/\eps + \log(nK)}{n}}.
$$ Note that $K = \mathcal{N}^{\rm pr}(N, \frac{\eps}{2L})
\leqslant \mathcal{N}(N, \frac{\eps}{4L})$,
and so we may conclude
with Proposition~\ref{prop:nu_volume_control}
and Proposition~\ref{prop:N_covering_numbers}
combined with our assumption on $n$
in Equation~\eqref{eqn:n_large}.
\end{proof}

\begin{proof}[Proof of Lemma~\ref{lem:manifold_density_sup_control}]
For the first statement, observe that
for all $x \in \supp \mu$ 
$$
1 = \int p_\eps(x, y') \ud \nu(y').
$$ By Proposition~\ref{prop:consequences_of_lipschitz_cost}
for any $y \in \supp(\nu)$,
$$
1 \gtrsim p_\eps(x, y) \nu(B(y, \frac{\eps}{L}))
\gtrsim p_\eps(x, y) \big( \frac{\eps}{L} \big)^{d_\nu},
$$ where the second inequality
follows by Proposition~\ref{prop:nu_volume_control}.
Re-arranging yields the first statement
since $x \in \supp \mu$ and $y \in \supp \nu$
were arbitrary.

For the second statement, it suffices to show
the statement in the event described
by Lemma~\ref{lem:empirical_mass_balls}.
Use Proposition~\ref{prop:consequences_of_lipschitz_cost}
to find that, for any $x \in \supp \mu, y \in \supp \nu$,
$$
\hat p_\eps(x, y) \leqslant \min_{j \in [n]}
\big\{e^{\frac{2L}{\eps}\|y - y_j\|} \hat p_\eps(x, y_j)
\big\}.
$$ Reasoning as above,
$$
\hat p_\eps(x, y_j) \lesssim 
\hat \nu (B(y_j, \frac{\eps}{L}))^{-1}
\leqslant \sup_{y' \in N} \hat \nu(B(y', \frac{\eps}{L}
))^{-1}.
$$ Since we are working
in the event described by Lemma~\ref{lem:empirical_mass_balls},
$$
\sup_{y' \in N}
\hat \nu(B(y', \frac{\eps}{L} ))^{-1} \lesssim
\Big(\frac{L}{\eps}\Big)^{d_\nu}.
$$ It follows that there exists a $y_j$
such that $\|y - y_j\| \leqslant \frac{\eps}{L}$,
in which case
$$
\hat p_\eps(x, y) \lesssim 
\Big(\frac{L}{\eps}\Big)^{d_\nu}.
$$

\end{proof}

\begin{proof}[Proof of Lemma~\ref{lem:L2_density}]
We give the proof of the first
claim, the remaining two follow in the same fashion.
Let $\mathcal{E}$ denote the event in 
Lemma~\ref{lem:manifold_density_sup_control}.
Observe that, by the pointwise bounds in Proposition~\ref{prop:dual_pointwise_control},
\begin{align*}
\E[\|\hat p_\eps - p_\eps\|^2_{L^2(\mu \otimes \nu)}]
&=  \E[\mathbbold{1}[\mathcal{E}]
\|\hat p_\eps - p_\eps\|^2_{L^2(\mu \otimes \nu)}]
+ 
\E[\mathbbold{1}[\mathcal{E}^c]
\|\hat p_\eps - p_\eps\|^2_{L^2(\mu \otimes \nu)}] \\
&\lesssim 
\E[\mathbbold{1}[\mathcal{E}]
\|\hat p_\eps - p_\eps\|^2_{L^2(\mu \otimes \nu)}]
+ e^{\frac{10}{\eps}}
\E[\mathbbold{1}[\mathcal{E}^c]]  \\
&\leqslant
\E[\mathbbold{1}[\mathcal{E}]
\|\hat p_\eps - p_\eps\|^2_{L^2(\mu \otimes \nu)}]
+ \frac{1}{n}.
\end{align*}
Note that
$|e^a - e^b| \leqslant e^{a \lor b}|a- b|
= (e^a \lor e^b)|a - b|$
for all $a, b \in \R$,
and so
\begin{align*}
\E[\mathbbold{1}[\mathcal{E}]
\|\hat p_\eps - p_\eps\|^2_{L^2(\mu \otimes \nu)}]
&\leqslant \Big(\frac{L}{\eps} \Big)^{2d_\nu}
\E[\mathbbold{1}[\mathcal{E}]
\| \log \hat p_\eps - \log p_\eps \|_{L^2(\mu \otimes \nu)}^2] \\
&\leqslant \Big(\frac{L}{\eps} \Big)^{2d_\nu}
\E[ \| \log \hat p_\eps - \log p_\eps \|_{L^2(\mu \otimes \nu)}^2] \\
&= \frac{1}{\eps^2}\cdot  \Big(\frac{L}{\eps} \Big)^{2d_\nu}
\E[ \| \hat f_\eps - f_\eps\|^2_{L^2(\mu)}
+ \|\hat g_\eps - g_\eps\|^2_{L^2(\nu)}].
\end{align*}
The first inequality follows. \end{proof}

\bibliographystyle{alpha}
\bibliography{annot}

\newcommand{\etalchar}[1]{$^{#1}$}
\begin{thebibliography}{MBNWW21}

\bibitem[AKT84]{ajtai1984optimal}
Mikl{\'o}s Ajtai, János Koml{\'o}s, and Gábor Tusnády.
\newblock On optimal matchings.
\newblock {\em Combinatorica}, 4(4):259--264, 1984.

\bibitem[BCH03]{bousquet2003measure}
Olivier Bousquet, Olivier Chapelle, and Matthias Hein.
\newblock Measure based regularization.
\newblock {\em Advances in Neural Information Processing Systems}, 16, 2003.

\bibitem[BCP19]{bigot2019central}
J{\'e}r{\'e}mie Bigot, Elsa Cazelles, and Nicolas Papadakis.
\newblock Central limit theorems for entropy-regularized optimal transport on
  finite spaces and statistical applications.
\newblock {\em Electron. J. Stat.}, 13:5120--5150, 2019.

\bibitem[Ber20]{berman2020sinkhorn}
Robert~J Berman.
\newblock The {S}inkhorn algorithm, parabolic optimal transport and geometric
  {M}onge--{A}mp{\`e}re equations.
\newblock {\em Numer. Math.}, 145(4):771--836, 2020.

\bibitem[BEZ22]{bayraktar2022stability}
Erhan Bayraktar, Stephan Eckstein, and Xin Zhang.
\newblock Stability and sample complexity of divergence regularized optimal
  transport.
\newblock {\em arXiv preprint}, 2022.

\bibitem[BGC17]{bengio2017deep}
Yoshua Bengio, Ian Goodfellow, and Aaron Courville.
\newblock {\em Deep learning}, volume~1.
\newblock MIT press Cambridge, MA, USA, 2017.

\bibitem[BLG14]{boissard2014mean}
Emmanuel Boissard and Thibaut Le~Gouic.
\newblock On the mean speed of convergence of empirical and occupation measures
  in {W}asserstein distance.
\newblock In {\em Annales de l'IHP Probabilit{\'e}s et statistiques},
  volume~50, pages 539--563, 2014.

\bibitem[BO69]{bishop1969manifolds}
Richard~L Bishop and Barrett O'Neill.
\newblock Manifolds of negative curvature.
\newblock {\em Trans. Am. Math. Soc.}, 145:1--49, 1969.

\bibitem[CKA23]{cuturi2023monge}
Marco Cuturi, Michal Klein, and Pierre Ablin.
\newblock Monge, {B}regman and {O}ccam: Interpretable optimal transport in
  high-dimensions with feature-sparse maps.
\newblock {\em arXiv preprint}, 2023.

\bibitem[CRL{\etalchar{+}}20]{chizat2020faster}
L{\'e}na{\"\i}c Chizat, Pierre Roussillon, Flavien L{\'e}ger,
  Fran{\c{c}}ois-Xavier Vialard, and Gabriel Peyr{\'e}.
\newblock Faster {W}asserstein distance estimation with the {S}inkhorn
  divergence.
\newblock {\em Advances in Neural Information Processing Systems}, 2020.

\bibitem[Cut13]{cuturi2013sinkhorn}
Marco Cuturi.
\newblock Sinkhorn distances: Lightspeed computation of optimal transport.
\newblock In {\em Advances in Neural Information Processing Systems}, 2013.

\bibitem[dBSLNW23]{del2023improved}
Eustasio del Barrio, Alberto~Gonz{\'a}lez Sanz, Jean-Michel Loubes, and
  Jonathan Niles-Weed.
\newblock An improved central limit theorem and fast convergence rates for
  entropic transportation costs.
\newblock {\em SIAM J. Math. Data Sci.}, 5(3):639--669, 2023.

\bibitem[DC16]{do2016differential}
Manfredo~P Do~Carmo.
\newblock {\em Differential geometry of curves and surfaces: revised and
  updated second edition}.
\newblock Courier Dover Publications, 2016.

\bibitem[DGP{\etalchar{+}}20]{dong2020study}
Yihe Dong, Yu~Gao, Richard Peng, Ilya Razenshteyn, and Saurabh Sawlani.
\newblock A study of performance of optimal transport.
\newblock {\em arXiv preprint}, 2020.

\bibitem[DMG20]{di2020optimal}
Simone Di~Marino and Augusto Gerolin.
\newblock An optimal transport approach for the {S}chr{\"o}dinger bridge
  problem and convergence of {S}inkhorn algorithm.
\newblock {\em Journal of Scientific Computing}, 85(2):1--28, 2020.

\bibitem[Dud69]{dudley1969speed}
Richard~Mansfield Dudley.
\newblock The speed of mean {G}livenko-{C}antelli convergence.
\newblock {\em Ann. Math. Stat.}, 40(1):40--50, 1969.

\bibitem[DY95]{dobric1995asymptotics}
V~Dobri{\'c} and Joseph~E Yukich.
\newblock Asymptotics for transportation cost in high dimensions.
\newblock {\em J. Theor. Prob.}, 8:97--118, 1995.

\bibitem[EN23]{eckstein2023convergence}
Stephan Eckstein and Marcel Nutz.
\newblock Convergence rates for regularized optimal transport via quantization.
\newblock {\em Math. Oper. Res.}, 2023.

\bibitem[FCTR16]{flamary2016optimal}
R~Flamary, N~Courty, D~Tuia, and A~Rakotomamonjy.
\newblock Optimal transport for domain adaptation.
\newblock {\em IEEE Trans. Pattern Anal. Mach. Intell}, 1, 2016.

\bibitem[FMN16]{fefferman2016testing}
Charles Fefferman, Sanjoy Mitter, and Hariharan Narayanan.
\newblock Testing the manifold hypothesis.
\newblock {\em J. Am. Math. Soc.}, 29(4):983--1049, 2016.

\bibitem[GCB{\etalchar{+}}19]{genevay2019sample}
Aude Genevay, L{\'e}na{\"\i}c Chizat, Francis Bach, Marco Cuturi, and Gabriel
  Peyr{\'e}.
\newblock Sample complexity of {S}inkhorn divergences.
\newblock {\em International Conference on Artificial Intelligence and
  Statistics}, 2019.

\bibitem[GH23]{groppe2023lower}
Michel Groppe and Shayan Hundrieser.
\newblock Lower complexity adaptation for empirical entropic optimal transport.
\newblock {\em arXiv preprint}, 2023.

\bibitem[GKRS22]{goldfeld2022statistical}
Ziv Goldfeld, Kengo Kato, Gabriel Rioux, and Ritwik Sadhu.
\newblock Statistical inference with regularized optimal transport.
\newblock {\em arXiv preprint}, 2022.

\bibitem[Gri09]{grigoryan2009heat}
Alexander Grigoryan.
\newblock {\em Heat kernel and analysis on manifolds}.
\newblock American Mathematical Society, 2009.

\bibitem[GSH23]{gonzalez2023weak}
Alberto Gonz{\'a}lez-Sanz and Shayan Hundrieser.
\newblock Weak limits for empirical entropic optimal transport: Beyond smooth
  costs.
\newblock {\em arXiv preprint}, 2023.

\bibitem[GSLNW22]{gonzalez2022weak}
Alberto Gonzalez-Sanz, Jean-Michel Loubes, and Jonathan Niles-Weed.
\newblock Weak limits of entropy regularized optimal transport; potentials,
  plans and divergences.
\newblock {\em arXiv preprint}, 2022.

\bibitem[GTGHS20]{garcia2020error}
Nicolás García~Trillos, Moritz Gerlach, Matthias Hein, and Dejan Slepčev.
\newblock Error estimates for spectral convergence of the graph {L}aplacian on
  random geometric graphs toward the {L}aplace--{B}eltrami operator.
\newblock {\em Found. Comput. Math.}, 20(4):827--887, 2020.

\bibitem[HAL07]{hein2007graph}
Matthias Hein, Jean-Yves Audibert, and Ulrike~von Luxburg.
\newblock Graph {L}aplacians and their convergence on random neighborhood
  graphs.
\newblock {\em J. Mach. Learn. Res.}, 8(6), 2007.

\bibitem[HR21]{hutter2021minimax}
Jan-Christian H\"utter and Philippe Rigollet.
\newblock Minimax estimation of smooth optimal transport maps.
\newblock {\em Ann. Stat.}, 49(2):1166--1194, 2021.

\bibitem[HSM22]{hundrieser2022empirical}
Shayan Hundrieser, Thomas Staudt, and Axel Munk.
\newblock Empirical optimal transport between different measures adapts to
  lower complexity.
\newblock {\em arXiv preprint}, 2022.

\bibitem[KNS16]{karimi2016linear}
Hamed Karimi, Julie Nutini, and Mark Schmidt.
\newblock Linear convergence of gradient and proximal-gradient methods under
  the {{P}olyak-{{\L}{}}ojasiewicz condition}.
\newblock {\em Joint European Conference on Machine Learning and Knowledge
  Discovery in Databases}, 2016.

\bibitem[KTM20]{klatt2020empirical}
Marcel Klatt, Carla Tameling, and Axel Munk.
\newblock Empirical regularized optimal transport: {s}tatistical theory and
  applications.
\newblock {\em SIAM J. Math. Data Sci.}, 2(2):419--443, 2020.

\bibitem[Lee18]{lee2018introduction}
John~M Lee.
\newblock {\em Introduction to Riemannian manifolds}, volume~2.
\newblock Springer, 2018.

\bibitem[Li12]{li2012geometric}
Peter Li.
\newblock {\em Geometric analysis}.
\newblock Cambridge University Press, 2012.

\bibitem[LS86]{leighton1986tight}
Frank~Thomson Leighton and Peter Shor.
\newblock Tight bounds for minimax grid matching, with applications to the
  average case analysis of algorithms.
\newblock {\em ACM Symposium on Theory of Computing}, 1986.

\bibitem[LSPC19]{luise2019sinkhorn}
Giulia Luise, Saverio Salzo, Massimiliano Pontil, and Carlo Ciliberto.
\newblock Sinkhorn barycenters with free support via {F}rank-{W}olfe algorithm.
\newblock {\em Advances in Neural Information Processing Systems},
  32:9322--9333, 2019.

\bibitem[Mar21]{dimarino2021linear}
Simone~Di Marino.
\newblock Linear convergence of {S}inkhorn algorithm and asymptotic rate.
\newblock Unpublished manuscript, 2021.

\bibitem[MBNWW21]{manole2021plugin}
Tudor Manole, Sivaraman Balakrishnan, Jonathan Niles-Weed, and Larry Wasserman.
\newblock Plugin estimation of smooth optimal transport maps.
\newblock {\em arXiv preprint}, 2021.

\bibitem[MNW19]{mena2019statistical}
Gonzalo Mena and Jonathan Niles-Weed.
\newblock Statistical bounds for entropic optimal transport: {S}ample
  complexity and the central limit theorem.
\newblock {\em Advances in Neural Information Processing Systems}, 2019.

\bibitem[MNW21]{manole2021sharp}
Tudor Manole and Jonathan Niles-Weed.
\newblock Sharp convergence rates for empirical optimal transport with smooth
  costs.
\newblock {\em arXiv preprint}, 2021.

\bibitem[MWMA21]{masud2021multivariate}
Shoaib~Bin Masud, Matthew Werenski, James~M Murphy, and Shuchin Aeron.
\newblock Multivariate rank via entropic optimal transport: Sample efficiency
  and generative modeling.
\newblock {\em arXiv preprint}, 2021.

\bibitem[NWR22]{niles2022estimation}
Jonathan Niles-Weed and Philippe Rigollet.
\newblock Estimation of {W}asserstein distances in the spiked transport model.
\newblock {\em Bernoulli}, 28(4):2663--2688, 2022.

\bibitem[PC19]{peyre2019computational}
Gabriel Peyr{é} and Marco Cuturi.
\newblock Computational optimal transport.
\newblock {\em Found. Trends Mach. Learn.}, 11(5-6):355--607, 2019.

\bibitem[PDNW23]{pooladian2023minimax}
Aram-Alexandre Pooladian, Vincent Divol, and Jonathan Niles-Weed.
\newblock Minimax estimation of discontinuous optimal transport maps: The
  semi-discrete case.
\newblock {\em arXiv preprint}, 2023.

\bibitem[PNW21]{pooladian2021entropic}
Aram-Alexandre Pooladian and Jonathan Niles-Weed.
\newblock Entropic estimation of optimal transport maps.
\newblock {\em arXiv preprint}, 2021.

\bibitem[PZ19]{panaretos2019statistical}
Victor~M. Panaretos and Yoav Zemel.
\newblock Statistical aspects of {W}asserstein distances.
\newblock {\em Annu. Rev. Stat. Appl.}, 6:405--431, 2019.

\bibitem[RS22]{rigollet2022sample}
Philippe Rigollet and Austin~J. Stromme.
\newblock On the sample complexity of entropic optimal transport.
\newblock {\em arXiv preprint}, 2022.

\bibitem[San15]{santambrogio2015optimal}
Filippo Santambrogio.
\newblock {\em Optimal transport for applied mathematicians}, volume~87 of {\em
  Progress in Nonlinear Differential Equations and their Applications}.
\newblock Birkh\"{a}user/Springer, Cham, 2015.
\newblock Calculus of variations, PDEs, and modeling.

\bibitem[SDF{\etalchar{+}}18]{seguy2018large}
Vivien Seguy, Bharath~Bhushan Damodaran, Remi Flamary, Nicolas Courty, Antoine
  Rolet, and Mathieu Blondel.
\newblock Large-scale optimal transport and mapping estimation.
\newblock {\em International Conference on Learning Representations}, 2018.

\bibitem[Sin64]{sinkhorn1964relationship}
Richard Sinkhorn.
\newblock A relationship between arbitrary positive matrices and doubly
  stochastic matrices.
\newblock {\em Ann. Math. Stat.}, 35(2):876--879, 1964.

\bibitem[SK67]{sinkhorn1967concerning}
Richard Sinkhorn and Paul Knopp.
\newblock Concerning nonnegative matrices and doubly stochastic matrices.
\newblock {\em Pac. J. Math.}, 21(2):343--348, 1967.

\bibitem[SST{\etalchar{+}}19]{schiebinger2019optimal}
Geoffrey Schiebinger, Jian Shu, Marcin Tabaka, Brian Cleary, Vidya Subramanian,
  Aryeh Solomon, Joshua Gould, Siyan Liu, Stacie Lin, Peter Berube, et~al.
\newblock Optimal-transport analysis of single-cell gene expression identifies
  developmental trajectories in reprogramming.
\newblock {\em Cell}, 176(4):928--943, 2019.

\bibitem[Str23]{stromme2023sampling}
Austin~J. Stromme.
\newblock Sampling from a {S}chr{\"o}dinger bridge.
\newblock {\em International Conference on Artificial Intelligence and
  Statistics}, 2023.

\bibitem[Th{\"a}08]{thale200850}
Christoph Th{\"a}le.
\newblock 50 years sets with positive reach--a survey.
\newblock {\em Surv. Math. Appl.}, 3:123--165, 2008.

\bibitem[Ver18]{vershynin2018high}
Roman Vershynin.
\newblock {\em High-dimensional probability: An introduction with applications
  in data science}, volume~47.
\newblock Cambridge University Press, 2018.

\bibitem[Vil08]{villani2008optimal}
C{\'e}dric Villani.
\newblock {\em Optimal transport: Old and new}.
\newblock Springer Science \& Business Media, 2008.

\bibitem[WB19]{weed2019sharp}
Jonathan Weed and Francis Bach.
\newblock Sharp asymptotic and finite-sample rates of convergence of empirical
  measures in {W}asserstein distance.
\newblock {\em Bernoulli}, 25(4 A):2620--2648, 2019.

\bibitem[WMMA23]{werenski2023rank}
Matthew Werenski, Shoaib~Bin Masud, James~M Murphy, and Shuchin Aeron.
\newblock On rank energy statistics via optimal transport: Continuity,
  convergence, and change point detection.
\newblock {\em arXiv preprint}, 2023.

\bibitem[ZGMS22]{zhang2022gromov}
Zhengxin Zhang, Ziv Goldfeld, Youssef Mroueh, and Bharath~K Sriperumbudur.
\newblock Gromov-{W}asserstein distances: Entropic regularization, duality, and
  sample complexity.
\newblock {\em arXiv preprint}, 2022.

\end{thebibliography}
\end{document}